\documentclass{article}%
\usepackage{amsfonts}
\usepackage{amssymb}
\usepackage{amsmath}
\usepackage{fancyhdr}
\usepackage{graphicx}%
\providecommand{\U}[1]{\protect\rule{.1in}{.1in}}
\newtheorem{theorem}{Theorem}

\newtheorem{corollary}[theorem]{Corollary}

\newtheorem{example}[theorem]{Example}

\newtheorem{lemma}[theorem]{Lemma}

\newtheorem{proposition}[theorem]{Proposition}

\newenvironment{proof}[1][Proof]{\noindent\textbf{#1.} }{\ \rule{0.5em}{0.5em}}
\ifx\pdfoutput\relax\let\pdfoutput=\undefined\fi
\newcount\msipdfoutput
\ifx\pdfoutput\undefined\else
\ifcase\pdfoutput\else
\ifx\paperwidth\undefined\else
\ifdim\paperheight=0pt\relax\else\pdfpageheight\paperheight\fi
\ifdim\paperwidth=0pt\relax\else\pdfpagewidth\paperwidth\fi
\fi\fi\fi
\begin{document}

\title{Complex Hyperbolic Geometry and Hilbert Spaces with the Complete Pick Property}
\author{Richard Rochberg}
\maketitle

\begin{abstract}
Suppose $H$ is a finite dimensional reproducing kernel Hilbert space of
functions on $X.$ If $H$ has the complete Pick property then there is an
isometric map, $\Phi,$ from $X,$ with the metric induced by $H,$ into complex
hyperbolic space, $\mathbb{CH}^{n},$ with its pseudohyperbolic metric. We
investigate the relationships between the geometry of $\Phi(X)$ and the
function theory of $H$ and its multiplier algebra.


\end{abstract}
\tableofcontents

\section{Introduction and Summary}

We begin with an informal overview; definitions and details are in the later sections.

The Hilbert spaces in this paper are finite dimensional.

Suppose $H$ is a reproducing kernel Hilbert space, RKHS, of functions on a set
$X.$ If $H$ has the complete Pick property, CPP, then there is a map, $\Phi,$
of $X$ into the complex unit ball, $\mathbb{B}^{n}\subset\mathbb{C}^{n},$ so
that $H$ is equivalent to $DA_{n}\left(  \Phi(X)\right)  ,$ the subspace of
the Drury Arveson space, $DA_{n},$ generated by the $DA_{n}$ kernel functions
for points of $\Phi(X).$ The map $\Phi$ is an isometry, mapping $X$ with the
metric induced by $H$ to $\Phi(X)$ with the metric induced by $DA_{n}.$ That
latter metric is the restriction to $\Phi(X)$ of the pseudohyperbolic
metric$\ $on $\mathbb{B}^{n}.$ Thus, the passage from $H$ to $\Phi(X)$
establishes a correspondence between finite dimensional RKHS with the CPP and
finite subsets of complex hyperbolic space, $\mathbb{CH}^{n}$. In this paper
we study the relationship between the analytic structure of $H$ and its
multiplier algebra, and the geometry of $\Phi(X)$.

Because we are working at the interface of different areas our expository
material is richer than usual.

The next section contains notation and background about Hilbert spaces and
about complex hyperbolic space. In the section after that we introduce and
develop numerical invariants of RKHS; some are based on Gram matrix entries,
others are defined using extremal problems in $H$ or its multiplier algebra.
We also discuss related geometric invariants of sets in hyperbolic space. In
Section 4 we consider rescalings of a RKHS $H.$ For $H$ with the CPP we
establish a close relation between rescalings and the action of the
automorphism group of hyperbolic space on the associated set $\Phi(X).$ We
also discuss other ways of modifying a RKHS to produce new spaces, and
establish a relation between those modifications and conjugation operators on
the Hilbert space.

Section 5 contains results about the existence and properties of the embedding
$\Phi.$ We begin with an analysis of a strengthened version of the triangle
inequality which must hold if there is such a $\Phi,$ but is not a sufficient
condition. We then consider two dimensional $H;$ there the embedding is always
possible and is easy to describe. For a three dimensional $H$ many aspects of
the general finite dimensional case appear, the most important being that the
embedding may not be possible; for it to be possible $H$ must have the Pick
property. To go beyond three dimensional spaces we use induction on dimension.
The induction step resembles the three dimensional construction but it
requires that $H$ have the more subtle complete Pick property.

If $H$ is three dimensional then describing $\Phi(X)$ up to automorphisms of
$\mathbb{CH}^{2}$ is a version of the question studied systematically by Brehm
\cite{B} of describing congruence classes of triangles in complex hyperbolic
space$.$ He parameterizes those classes by the three distances between pairs
of vertices and a fourth quantity, the "shape invariant", an invariant from
hyperbolic geometry that has no Euclidean analog. In particular, even when
$\Phi(X)$ only has three points, its geometric structure is richer than that
encoded in its metric. One of our implicit goals is to understand this
geometry which is more general than metric geometry.

In Section 6 we relate the Hilbert space invariants we introduced to the
geometry of $\Phi(X).$ In particular we describe conditions on the numerical
invariants which correspond to the having $\Phi(X)$ inside a single geodesic
or inside a totally geodesically embedded copy of the Poincare disk,
$\mathbb{CH}^{1},$ or of the real hyperbolic plane $\mathbb{RH}^{2}.$ We also
establish a relationship between $\Phi(X)$ being in a copy of the Poincare
disk and $H$ having a conjugation operator which interchanges the basis of
reproducing kernels with its dual basis.

In Section \ref{tree spaces} we consider a class of RKHS of functions on trees
and describe the associated maps $\Phi(X).$ That class of spaces includes the
dyadic Dirichlet space studied systematically in \cite{ARS02} as well as more
recent variants \cite{ARSW18}.

In Section \ref{multiplier algebra} we briefly discuss the multiplier algebra
of $H.$ If $H$ has the CPP then one can recover $H$ using numerical data
derived from its multiplier algebra. Less is known about general $H.$

A brief final section contains remarks and questions.

\section{Background and Notation}

General background on the material we discuss is in \cite{Ru}, \cite{Go},
\cite{AM}, \cite{Sa}, \cite{Sh}, and \cite{ARSW18}.

\subsection{RKHS}

An $n$ dimensional reproducing kernel Hilbert space, RKHS, is an $n$
dimensional complex inner product space, $H,$ together with a distinguished
basis, $B=B(H)=\left\{  k_{i}\right\}  _{i=1}^{n},$ of vectors called
\textit{reproducing kernels}. Associated with $H$ is a set $X=X(H)$ with the
same index set as $B.$ Vectors $h\in H$ can be, and generally are, regarded as
functions on $X(H)$ by setting, for $x\in H,$ $h(x)=\left\langle
h,k_{x}\right\rangle .$ We suppose throughout that $H$ is irreducible; that
is, for any $x,y\in X(H),$ $x\neq y,$ the functions $k_{x}$ and $k_{y}$ are
linearly independent and $\left\langle k_{x},k_{y}\right\rangle \neq0$

Sometimes we will write $H$ for such a space without further comment and, also
without comment, write $X=\left\{  x_{i}\right\}  _{i=1}^{n}$ for $X(H)$. If
$x=x_{i}$ and $y=x_{j}$ are in $X(H)$ we may write $k_{x}$ or $k_{i}$ for
$k_{x_{i}},$ and write $k_{xy}$ or $k_{ij}$ for $\left\langle k_{i}%
,k_{j}\right\rangle =k_{x_{i}}(x_{j}).$ We will denote the normalized kernel
$k_{ii}^{-1/2}k_{i}$ by $\widehat{k_{i}},$ and write the inner product of two
such as $\widehat{k_{xy}};$ thus $\widehat{k_{xy}}=$ $<\widehat{k_{x}%
},\widehat{k_{y}}>.$ The Gram matrix of $H$ is the positive $n\times n$ matrix
$G(H)=\left(  k_{ij}\right)  .$

The function $\delta=\delta_{H}$ defined by; for $x_{i},x_{j}\in X(H),$%
\begin{equation}
\delta_{ij}^{2}=\delta_{H}^{2}(x_{i},x_{j})=1-\frac{\left\vert k_{ij}%
\right\vert ^{2}}{k_{ii}k_{jj}}=1-|\widehat{k_{ij}}|^{2}, \label{def d}%
\end{equation}
is a metric on $X(H).$ It is an elementary exercise that the same quantity is
described by%
\begin{equation}
\delta_{H}(x,y)=\delta(x,y)=\frac{1}{\left\Vert k_{x}\right\Vert }\sup\left\{
\operatorname{Re}h(x):h\in H\text{, }h(y)=0,\left\Vert h\right\Vert _{H}%
\leq1\right\}  . \label{extremal d}%
\end{equation}
Also, it is not hard to show that, with $P_{x}$ denoting the orthogonal
projection onto the span of the kernel function $k_{x},$ $\delta_{H}$ can be
described in terms of operator norms:
\begin{equation}
\delta_{H}(x,y)=\left\Vert P_{x}-P_{y}\right\Vert . \label{normm}%
\end{equation}

The metric $\delta$ is a generalization of the classic pseudohyperbolic
metric, $\rho,$ on the disk. If $H$ is the Hardy space $H^{2}$ then, on the
unit disk $\delta_{H}=\rho.$ For more about $\delta_{H}$ see \cite{ARS07},
\cite{ARSW18}.

\subsection{Multiplier Algebras}

Given a \textit{symbol function}, $m,$ defined on $X(H),$ the associated
\textit{multiplier operator}, $M_{n},$ is the linear operator on $H$ defined
by, for $h\in H,x\in X,$ $\left(  M_{m}h\right)  (x)=m(x)h(x).$ The collection
of all multiplier operators on $H$ is the \textit{multiplier algebra} of $H,$
$\operatorname*{Mult}\left(  H\right)  $. With the operator norm
$\operatorname*{Mult}\left(  H\right)  $ is a commutative Banach algebra
generated by $n$ orthogonal idempotents. We denote its spectrum, its maximal
ideal space, by $\operatorname*{Spec}\left(  \operatorname*{Mult}\left(
H\right)  \right)  .$ The Gleason metric on the spectrum is defined by, for
$x,y\in\operatorname*{Spec}\left(  \operatorname*{Mult}\left(  H\right)
\right)  $,
\begin{equation}
\delta_{G}(x,y)=\sup\{\operatorname{Re}m(x):m\in\operatorname*{Mult}\left(
H\right)  \text{, }m(y)=0,\left\Vert m\right\Vert _{_{\operatorname*{Mult}%
\left(  H\right)  }}\leq1\}. \label{extremal g}%
\end{equation}
It is an exercise in the use of von Neumann's inequality that $\delta_{G}$ can
also be described using the pseudohyperbolic metric \cite{ARSW18}.%
\begin{equation}
\delta_{G}(x,y)=\sup\{\rho(m(x),m(y)):m\in\operatorname*{Mult}\left(
H\right)  \text{, }\left\Vert m\right\Vert _{_{\operatorname*{Mult}\left(
H\right)  }}\leq1\}.
\end{equation}

Identifying $x_{i}\in X(H)$ with the maximal ideal of multipliers which vanish
at $x_{i}$ gives a natural identification of $X(H)$ with $\operatorname*{Spec}%
\left(  \operatorname*{Mult}\left(  H\right)  \right)  .$ Using this
identification we also regard $\delta_{G}$ as a metric on $X(H).$

\subsection{Rescaling and Invariance\label{invariance}}

We want to note when two RHKS are the same in a natural sense. We do this with
the equivalence relation \textit{rescaling}.

Suppose $H$ and $\widetilde{H}$ are two RKHS of the same finite dimension. We
say that $\widetilde{H}$ is a \textit{rescaling} of $H,$ or is obtained from
$H$ by rescaling, and write $H\sim\widetilde{H},$ if there is a one to one map
$\Xi:X(H)\rightarrow X(\widetilde{H})$ and a nonvanishing complex valued
function $\gamma$ defined on $X(H)$ so that, with $\left\{  k_{i}\right\}  $
and $\{\tilde{k}_{i}\}$ denoting the kernel functions for $H$ and
$\widetilde{H}$ respectively, we have for all $x\in X(H),$
\begin{equation}
\tilde{k}_{\Xi(x)}(\Xi(\cdot))=\gamma(x)k_{x}(\cdot); \label{rescaling}%
\end{equation}
or, equivalently, $\forall x,y\in X(H)$%
\begin{equation}
\tilde{k}_{\Xi(x)\Xi(y)}=\left\langle \tilde{k}_{\Xi(x)},\tilde{k}_{\Xi
(y)}\right\rangle =\gamma(x)\overline{\gamma(y)}k_{xy}. \label{rescaling 2}%
\end{equation}
Another equivalent formulation is that the linear map $A:H\rightarrow
\widetilde{H}$ defined by $A(k_{x})=\tilde{k}_{\Xi(x)}$ and linearity has the
property that $A^{\ast}A$ is diagonalized by the $\left\{  k_{i}\right\}  $
and has nonzero eigenvalues.

Rescaling is an equivalence relation, more details about it are \cite[Sec.
2.6]{AM}. If $H\sim\widetilde{H}$ we can use $\Xi$ to identify
$X(\widetilde{H})$ with $X(H),$ thus reducing to the case of
$X(H)=X(\widetilde{H})$ and $\Xi$ the identity map. We may do this without mention.

Associated with $\widetilde{H}$ is the new Gram matrix, $G(\widetilde{H}).$ If
$X(H)=X(\widetilde{H})$ and if $\Xi$ is the identity map then $G(H)$ and
$G(\widetilde{H})$ are related by%
\begin{equation}
\Gamma(\gamma(x_{1}),...,\gamma(x_{n}))\text{ }G(H)\text{ }\overline
{\Gamma(\gamma(x_{1}),...,\gamma(x_{n}))}=G(\widetilde{H}).
\label{modified gram}%
\end{equation}
Here $\Gamma(c_{i}...,c_{n})$ is the $n\times n$ matrix with $c_{1},...,c_{n}$
on the diagonal and zeros elsewhere. If matrices $A$ and $B$ are related in
this way then we will write $A\sim B.$ Different choices, $\Gamma=\Gamma_{1}$
and $\Gamma=\Gamma_{2},$ produce different $G(\widetilde{H})$ unless
$\Gamma_{1}=\alpha\Gamma_{2}$ for some unimodular $\alpha,$ $.$

One convenient rescaling is the \textit{basepoint rescaling. }A point $x\in
X(H)$ is selected as basepoint and $H$ is rescaled so that the rescaled kernel
for $x$ is identically one. The Gram matrix of the rescaled space will have
ones in the row and column corresponding to $y.$ Two spaces are equivalent
under rescaling if and only if they have the same Gram matrix after basepoint
rescaling. Another useful rescaling is \textit{normalized kernels rescaling
}in which all the kernel functions are rescaled to be unit vectors. That
rescaled space has a Gram matrix with all ones on the diagonal. That rescaling
becomes unique after a further rescaling to insure, for instance, that the
entries in the first row of the Gram matrix are real. We will encounter a
different type of rescaling in the proof of Theorem \ref{is orthogonal}.

We will call quantities built from $H$ invariant if they are unchanged under
rescaling. For instance, neither the Gram matrix entries, $k_{ij},$ nor the
normalized kernel functions $\widehat{k_{i}}$ are invariant, but both
$|\widehat{k_{xy}}|$ and $\delta_{H}(x,y)$ are invariant. The Gram matrix of
the basepoint normalized rescaling is invariant as is Gram matrix of the
normalized kernel rescaling once it is further rescaled so that the first row
is real. The multiplier algebra is invariant. That is, if $H\sim\widetilde{H}%
$, then $\operatorname*{Mult}\left(  H\right)  $ and $\operatorname*{Mult}%
(\widetilde{H})\ $are the same sets of functions with the same algebraic
structure and with the same norm.

Some statements which are not invariant under rescaling can be viewed as the
specializations of invariant statements obtained by basepoint rescaling. For
example, the statement $k_{xx}=k_{xy}$ is not invariant. However the
statement
\[
\frac{k_{xx}}{k_{x\alpha}k_{\alpha x}}=\frac{k_{xy}}{k_{x\alpha}k_{\alpha y}%
}.
\]
is invariant; and, after basepoint rescaling with $a$ as the basepoint,
specializes to $k_{xx}=k_{xy}$.

There is an interesting discussion of this type of transformation in
\cite[7.2.3]{Go}.

\subsection{The Complete Pick Property and the Spaces $DA_{n}$}

We are particularly interested in spaces $H$ with the CPP. There is a
substantial literature on this class and we will take what we need from
\cite{AM}, \cite{Sa}, \cite{Sh}, and \cite{ARSW18}.

The Pick property is an extension property for multipliers. Suppose an
$n-$dimensional RKHS $H$ is given, $n\geq2,$ along with a subset $Y$ of
$X(H).$ Let $H_{Y}$ be the RKHS that is the span of $\left\{  k_{y}\right\}
_{y\in Y}$. Given $M=M_{m}\in\operatorname*{Mult}\left(  H\right)  ,$
$\left\Vert M\right\Vert \leq1$, we can define a multiplier $M_{Y}$ on $H_{Y}$
by restricting $m,$ which is a function on $X(H),$ to a function, now called
$m_{Y},$ on the subset $Y=X(H_{Y})\subset X(H).$ We define $M_{Y}$ to be the
multiplier on $H_{Y}$ with symbol function $m_{Y}.$ The adjoint, $M_{Y}^{\ast
},$ is the restriction of $M^{\ast}$ to the $M^{\ast}$ invariant subspace
$H_{Y};$ hence $\left\Vert M_{Y}^{\ast}\right\Vert \leq$ 1, and thus, also,
$\left\Vert M_{Y}\right\Vert \leq1.$ The extension problem which defines the
Pick property is the converse question. Given $N_{Y},$ a multiplier on $H_{Y}$
of norm one, is there a multiplier $M$ on $H$, $\left\Vert M\right\Vert =1,$
so that $N_{Y}=M_{Y}$? If this question always has a positive answer then $H$
is said to have the Pick property (or the scalar Pick property). The stronger
and more subtle CPP is defined by also having a positive answer to the
matricial analog of that multiplier extension question.

In many cases below where we hypothesize that a space has the CPP, it would
suffice to just assume the Pick property. We leave it to the interested reader
to note those refinements as we go.

The Drury Arveson spaces, $DA_{n},$ are fundamental example of spaces with the
CPP. Let $\mathbb{B}^{n}\subset\mathbb{C}^{n}$ be the ball in complex $n-$
space, and denote the inner product on $\mathbb{C}^{n}$ by $\left\langle
\cdot,\cdot\right\rangle ,$ The space $DA_{n}$ is the RKHS of holomorphic
functions on $\mathbb{B}^{n}$ defined by the reproducing kernels
$\{k_{z}(\cdot)=\left(  1-\left(  \cdot,z\right)  \right)  ^{-1}%
:z\in\mathbb{B}^{n}\}.$ In particular $DA_{1}$ is the classical Hardy space
$H^{2}$ on the unit disk. These spaces are discussed in detail in \cite{Sh}.

For any finite $Y\subset\mathbb{B}^{n}=X(DA_{n})$ let $DA_{n}(Y)$ be the
subspace of $DA_{n}$ spanned by the subset $\left\{  k_{y}\right\}  _{y\in Y}$
of the $DA_{n}$ reproducing kernels. Each of these spaces inherits the CPP
from its containing $DA_{n}.$ Any $H$ with the CPP is a rescaling of a space
$DA_{n}(X)$ and that fact is the starting point for our discussions.

\begin{theorem}
[{\cite[Thm. 8.2]{AM}}]\label{embed}A finite dimensional RKHS $H$ has the
complete Pick property if and only if there is a finite set $X$ in some
$\mathbb{CH}^{n}$ such that $H\sim DA_{n}(X).$
\end{theorem}

Thus, associated to any such $H$ is a map $\Phi$ of $X(H)$ into $\mathbb{CH}%
^{n}$ so that $H\sim DA_{n}(\Phi(X(H))).$ Our interest here is is the relation
between the structural properties of $H$ and $\operatorname*{Mult}(H)$ and the
geometry of $\Phi(X(H))$.

\subsection{Complex Hyperbolic Space \label{complex hyperbolic space}}

We now discuss $\mathbb{CH}^{n},$ complex hyperbolic $n-$space. Our basic
reference is \cite{Go}.

We begin with $\mathbb{CH}^{1}.$ The unit disk, $\mathbb{D=B}^{1}%
\subset\mathbb{C}$, is a complex manifold which has a transitive group of
holomorphic automorphism, $\operatorname*{Aut}\left(  \mathbb{B}^{1}\right)
,$ the Mobius maps of the disk to itself. $\mathbb{CH}^{1}$ carries the
$\operatorname*{Aut}\left(  \mathbb{B}^{1}\right)  $ invariant
pseudohyperbolic metric, $\rho, $%
\[
\rho(z,w)=\left\vert \frac{z-w}{1-\bar{w}z}\right\vert ,
\]
which can also be defined by setting $\rho(0,z)=|z|$ and requiring that $\rho$
be $\operatorname*{Aut}\left(  \mathbb{B}^{1}\right)  $ invariant. The complex
manifold $\mathbb{B}^{1},$ together with the metric $\rho,$ and the isometry
group $\operatorname*{Aut}\left(  \mathbb{B}^{1}\right)  $ is the disk model
of one dimensional complex hyperbolic space, $\mathbb{CH}^{1}.$ The metric
$\rho$ is not a length metric. The length metric which it induces, the
Bergman-Poincare metric, is an $\operatorname*{Aut}\left(  \mathbb{B}%
^{1}\right)  $ invariant Riemannian metric of constant curvature $-1/4.$ (Care
is needed here, the Bergman-Poincare metric is sometimes defined to be twice
what we just offered, in which case it has constant curvature $-1.$ Our choice
here insures that $\beta$ is the length metric induced by $\rho.)$ The full
set of isometries of $\mathbb{CH}^{1}$ consists of the holomorphic isometries
of $\operatorname*{Aut}\left(  \mathbb{B}^{1}\right)  $ and the complex
conjugates of elements of $\operatorname*{Aut}\left(  \mathbb{B}^{1}\right)
.$ For $X,Y\subset\mathbb{CH}^{1}$ we say $X$ and $Y$ are \textit{congruent},
$X\sim Y$, if there is $\Lambda\in\operatorname*{Aut}\left(  \mathbb{B}%
^{1}\right)  $ with $X=\Lambda Y.$ If $X$ and $Y$ are ordered sets we take the
terminology and notation to include the requirement that $\Lambda$ respect the ordering.

Similar facts on the unit ball, $\mathbb{B}^{n}\subset\mathbb{C}^{n}$, give a
model for complex hyperbolic $n-$space, $\mathbb{CH}^{n}$. Details about the
ball are in \cite{Ru}, about the metric $\rho$ in \cite{DW}, and about this
realization of $\mathbb{CH}^{n}$ in \cite{Go}. We just list some highlights.

The ball has a transitive group of holomorphic automorphisms,
$\operatorname*{Aut}\left(  \mathbb{B}^{n}\right)  .$ For each $a\in$
$\mathbb{B}^{n}$ there is a $\varphi_{a}\in\operatorname*{Aut}\left(
\mathbb{B}^{n}\right)  ,$ an involution of $\mathbb{B}^{n}$ which interchanges
$0$ and $a.$ Every unitary map of $\mathbb{C}^{n}$ is in $\operatorname*{Aut}%
\left(  \mathbb{B}^{n}\right)  $ and the unitary maps together with the
involutions generate $\operatorname*{Aut}\left(  \mathbb{B}^{n}\right)  $. In
particular, any automorphism which fixes the origin is given by a unitary
map.\ As with $n=1,$ for $X,Y\subset\mathbb{CH}^{n}$ we will write $X\sim Y$
if there is an element of $\operatorname*{Aut}\left(  \mathbb{B}^{n}\right)  $
which takes $X$ to $Y.$ Also, as with $n=1,$ there are $\rho-$isometries of
$\mathbb{CH}^{n}$ which are not holomorphic but are complex conjugates of
elements of $\operatorname*{Aut}\left(  \mathbb{B}^{n}\right)  .$

The pseudohyperbolic metric, $\rho,$ on the ball can be defined by saying that
for $z,w\in\mathbb{B}^{n}$ we have $\rho(z,w)=\left\vert \varphi
_{z}(w)\right\vert =\left\vert \varphi_{w}(z)\right\vert .$ Alternatively we
can set $\rho(0,z)=\left\vert z\right\vert $ for $z\in\mathbb{B}^{n}$ and
require that $\rho$ is $\operatorname*{Aut}\left(  \mathbb{B}^{n}\right)  $
invariant. The length metric generated by $\rho$ is $\beta,$ the
Bergman-Poincare metric; a Riemannian metric which is invariant under
$\operatorname*{Aut}\left(  \mathbb{B}^{n}\right)  .$and agrees
infinitesimally with the Euclidean metric at the origin. In contrast to one
dimensional complex hyperbolic space, $\mathbb{CH}^{1}$, and to real
hyperbolic $n-$space, $\mathbb{RH}^{n},$ the space $\mathbb{CH}^{n}$ with the
metric $\beta$ does not have constant sectional curvature. This lack of
isotropy is a fundamental feature in the metric geometry of $\mathbb{CH}^{n}.$

This same model of $\mathbb{CH}^{n}\ $has an alternative description which is
often used in geometric studies. In that model $\mathbb{CH}^{n}$ is defined as
the set of "negative points in projective space". Begin with $\mathbb{C}%
^{n+1}$ and the Hermitian form $\left[  \cdot,\cdot\right]  $ of signature
$\left(  n,1\right)  $ given by%
\[
\left[  (x_{1},x_{2},...,x_{n+1}),(y_{1},y_{2},...,y_{n+1})\right]
=-x_{n+1}\bar{y}_{n+1}+\sum\nolimits_{i=1}^{n}x_{i}\bar{y}_{i},
\]
Next, form the projective space $\mathbb{CP}^{n}$ from this $\mathbb{C}%
^{n+1}.$ Although $\left[  \cdot,\cdot\right]  $ is not well defined on
$\mathbb{CP}^{n},$ the quantity $\left[  x,x\right]  $ is always real and its
sign is constant on lines in $\mathbb{C}^{n+1}.$ Thus, this sign is well
defined on $\mathbb{CP}^{n}\ $and we define $\mathbb{CH}^{n}$ to be the subset
of $\mathbb{CP}^{n}$ on which it is negative. That will never happen on a line
which has $x_{n+1}=0,$ hence we can focus on the coordinate chart where
$x_{n+1}\neq0.$ There we can use the inhomogenous coordinates on projective
space obtained by representing points using $\left(  n+1\right)  -$tuples with
$x_{n+1}=1,$ and then abusing notation by writing $(x_{1},x_{2},...,x_{n})$
for $(x_{1},x_{2},...x_{n},1).$ In those coordinates the set of negative
points, $\mathbb{CH}^{n},$ is $\left\{  (x_{1},x_{2},...,x_{n}):\sum\left\vert
s_{i}\right\vert ^{2}<1\right\}  =\mathbb{B}^{n}.$ In those coordinates, we
regard $\left[  \cdot,\cdot\right]  $ as being defined on $\mathbb{CH}^{n}$
by
\[
\left[  x,y\right]  =\left[  (x_{1},x_{2},...,x_{n}),(y_{1},y_{2}%
,...y_{n})\right]  =-1+\sum\nolimits_{i=1}^{n}x_{i}\bar{y}_{i},=-(1-\left(
x,y\right)  ),
\]
In particular the $DA_{n}$ kernel functions can be written as $k_{xy}%
=-1/\left[  y,x\right]  $ and that relation allows translation between what we
do here and the literature centered on the geometry of $\mathbb{CH}^{n}.$

In this description of hyperbolic space the automorphisms of $\mathbb{CH}%
^{n}.$ which define the geometry of the model, are taken to be be those
natural automorphism of $\mathbb{CP}^{n}$ which preserve this set of negative
points. Although it is not obvious, these are the same as the automorphism in
$\operatorname*{Aut}\left(  \mathbb{B}^{n}\right)  $ which were discussed
earlier, and so we have the same model.

We call properties of sets in $\mathbb{CH}^{n}$ invariant if they are
preserved by automorphisms. Thus a set's being in a geodesic is an invariant
statement, that two geodesics cross at a right angle is not.

\subsubsection{Invariant Submanifolds}

We will be interested in some classes of submanifolds of $\mathbb{CH}^{n}$
which are preserved by automorphisms. Geodesic arcs are the totally geodesic
submanifolds of $\mathbb{CH}^{n}$ of real dimension one. Because automorphisms
are isometries they map geodesics to geodesics. and similarly for higher
dimension totally geodesic submanifolds. In particular, the class of geodesic
segments is invariant.

There are two classes of totally geodesic submanifolds of real dimension two
and both are preserved by automorphisms. The first consists of totally real
totally geodesic submanifolds. The slice $J=\{(x,y,0,,...,0):x.y\in\mathbb{R}%
$, $\left\vert x\right\vert ^{2}+\left\vert y\right\vert ^{2}<1\}$ is a model
case. The general elements of this class, which we will call \textit{real
geodesic disks,} are the images of $J$ under the action of
$\operatorname*{Aut}\left(  \mathbb{B}^{n}\right)  .$ $J$ is isometric to the
real hyperbolic plane, $\mathbb{RH}^{2}$; however it is the Beltrami-Klein
model of that plane, not the more familiar Poincare model. In the
Beltrami-Klein model the geodesics are Euclidean straight line segments. The
Poincare model is a conformal model of $\mathbb{RH}^{2},$ the Beltrami-Klein
model is not. \ More discussion and useful figures are in \cite[Section
3.1.9]{Go}.

The other class of totally geodesic submanifolds of real dimension two
consists of \textit{complex geodesics}. The horizontal slice $L=\left\{
\left\{  z,0,,...,0\right\}  :\left\vert z\right\vert <1\right\}  ,$ which is
isometric to $\mathbb{CH}^{1},$ is a model case, the others are the images of
$L$ under the automorphism group.

These classes also have higher dimensional analogs.

\section{Numerical Parameters\label{numerical}}

Fix, for this section: a RKHS $H$ with associated set $\left\{  x_{i}\right\}
_{i=1}^{n}=X=X(H),$ kernel functions $\left\{  k_{i}\right\}  _{i=1}^{n},$ and
multiplier algebra $A=\operatorname*{Mult}\left(  H\right)  .$

We are not supposing that $H=DA_{n}(X).$ However if $H$ is of that form then,
recalling the relation\ $k_{xy}=-1/\left[  y,x\right]  $ discussed in Section
\ref{complex hyperbolic space}, the parameters we describe can also be
regarded as functionals of $X.$ Furthermore, noting the discussion in Section
\ref{invariance}, the values only depend on the congruence class of $X.$

\subsection{Invariant Parameters From the Gram Matrix}

The Gram matrix of $H$ is $G(H)=\left(  k_{ij}\right)  _{i,j=1}^{n}.$ Those
matrix entries change when $H$ is rescaled, but there are quantities built
from those numbers which are invariant under rescaling and which will be
useful. The first is the distance function $\delta=\delta_{H}$ which we
introduced in (\ref{def d}). Here are several others.

\subsubsection{The Angular Invariant}

For $x,y,z\in X(H)$ we define the \textit{angular invariant} $A(x,y,z),$ by%
\begin{equation}
A(x,y,z)=\arg k_{xy}k_{yz}k_{zx}=\arg\widehat{k_{xy}}\widehat{k_{yz}%
}\widehat{k_{zx}}. \label{define A}%
\end{equation}
where we take $\left\vert \arg\left(  \zeta\right)  \right\vert \leq\pi.$
(When working with classical function spaces$,$ the ambiguity in $\arg\left(
\cdot\right)  $ is often removed by specifying that $\arg k_{xy}$ is a
continuous function of both variables and vanishes when $x=y.$ In that case,
as shown by the family of spaces of holomorphic functions on the disk with
$k_{z}(w)=$ $(1-zw)^{-\lambda}$ for $\lambda>0,$ there is no natural upper
bound for $A,)$ As with $\delta_{ij},$ we will write $A_{ijk}.$ The
interpretation of these invariants is subtle, we discuss it in Sections
\ref{about A} and \ref{projection}.

The $A_{ijk}$ are unchanged by rescaling of $H,$ and unchanged by cyclic
permutation of the indices, however they change sign when adjacent indices are
interchanged. Also, it is straightforward from the definitions that $A$
satisfies a cocycle identity; for indices $i,j,k,l$%
\begin{equation}
A_{i,j,k}-A_{i,j,l}+A_{i,k,l}-A_{j,k,l}=0. \label{cocycle}%
\end{equation}

\subsubsection{MQ Matrices}

In Section \ref{n>3} we will work with the matrices $MQ_{r}(H)$ used by
McCullough and Quiggen in characterizing $H$ with the CPP \cite[Thm \ 7.6]{AM}.

Suppose $H$ is $n$ dimensional. For $1\leq r\leq n$ define the $\left(
n-1\right)  \times$ $\left(  n-1\right)  $ matrices $MQ_{r}(H)$, by%
\begin{equation}
MQ_{r}(H)=MQ_{r}(X(H))=\left(  1-\frac{k_{ir}k_{rj}}{k_{ij}k_{rr}}\right)
_{\substack{1\leq i,j\leq n \\i\neq r,j\neq r}}. \label{MQ}%
\end{equation}

\subsubsection{LF}

Later we will also find the following invariants useful:%

\begin{equation}
LF_{123}^{2}=LF^{2}(x_{1},x_{2},x_{3})=\frac{1}{\delta_{12}^{2}}\left\vert
1-\frac{k_{21}k_{13}}{k_{23}k_{11}}\right\vert ^{2}=\frac{1}{\delta_{12}^{2}%
}\left\vert 1-\left\vert \frac{k_{21}k_{13}}{k_{23}k_{11}}\right\vert
e^{iA_{213}}\right\vert ^{2}, \label{L}%
\end{equation}
with similar notation for other indices. We will describe the geometric
interpretation of this quantity and the reason for its name in Section
\ref{geometry}.

\subsection{Describing Spaces and Counting Parameters\label{describing}}

Describing $H$ requires that we specify the basis $\left\{  k_{i}\right\}  $
of $\mathbb{C}^{n}$ and that requires $n^{2}$ complex parameters, $2n^{2}$
real parameters. If we are only interested in the $\left\{  k_{ij}\right\}  $,
the entries of $G(H),$ then the number is reduced; $G(H)$ is positive and
hence determined by $n^{2}$ real parameters. Further, if we only consider
equivalences classes modulo rescaling, then we have larger equivalence classes
and fewer parameters. The rescaling is determined by the matrix $\Gamma$ in
(\ref{modified gram}). That matrix is determined by $2n$ real parameters, but
the comment there about $\alpha$ shows that the rescaling is actually
described by $2n-1$ parameters. Thus, our count of real parameter is $n^{2}$
for the Gram matrices, diminished by $2n-1$ for possible rescalings, a total
of $\left(  n-1\right)  ^{2}.$ That is also the number of parameters required
to describe a configuration of $n$ points in complex hyperbolic space modulo automorphisms.

We are particularly interested in describing $H,$ up to rescaling, using
geometric data about $X(H)$. The $\left\{  \delta_{ij}\right\}  $ are part of
the answer, but, already in dimension $n=3$, there are too few of them. The
previous discussion suggests we need four parameters, and the distances only
provide three. For a fourth we will use the angular invariant $A(x,y,z)$ given
in (\ref{define A}).

For instance, a three dimensional space can be rescaled as a space $H$ with
Gram matrix
\begin{equation}
G(H)=%
\begin{pmatrix}
1 & 1 & 1\\
1 & k_{22} & k_{23}\\
1 & \overline{k_{23}} & k_{33}%
\end{pmatrix}
\label{three k}%
\end{equation}
with $k_{22},k_{33}>0.$ Thus the set $\kappa=\{k_{22},$ $k_{33},$
$\operatorname{Re} $ $k_{23},$ $\operatorname{Im}k_{23}\}$ is a set of
$\left(  3-1\right)  ^{2}=4$ real numbers which determine the Gram matrix, and
hence describes $H$ up to rescaling. The set $\delta=\{\delta_{12},\delta
_{13},\delta_{23},A_{123}\}$ carries the same information. We can write the
elements of $\delta$ in terms of the elements of $\kappa.$
\[
\delta=\left\{  1-\frac{1}{k_{22}},\text{ }1-\frac{1}{k_{33}},\text{ }%
1-\frac{\left\vert k_{23}\right\vert ^{2}}{k_{22}k_{33}},\text{ }\arg
k_{23}\right\}  ;
\]
and the passage from $\delta$ to $\kappa$ is similarly straightforward. Our
preference here is for the set of invariants $\delta.$ Those numbers are
invariant under rescaling and they also determine the Gram matrix of a
rescaled version of $H.$ Furthermore, if $H$ has the CPP and hence is of the
form $H\sim DA_{n}(X)$ for some $X,$ those numbers are geometric invariants of
$X$ which determine $X$ up to congruence (Theorem \ref{n=3} below).

A similar analysis holds if $H$ is $n$-dimensional. After basepoint rescaling
$G(H)$ is determined by the $(n-1)^{2}$ real parameters
\begin{equation}
J(X)=\left\{  \delta_{ij}:1\leq i<j\leq n\right\}  \cup\left\{  A_{1rs}%
:1<r<s\leq n\right\}  . \label{G(X)}%
\end{equation}
Again, these numbers are rescaling invariants and it is mechanical to pass
between this set and the entries of $G(H).$ Taking note of Theorems \ref{n=3}
and \ref{reduction} we see that if $H\sim DA_{m}(X)$ then these numbers also
determine the congruence class of $X.$

\subsection{Larger Spaces\label{larger spaces}}

Suppose we are given spaces $H\subset H^{\prime},$ Given $x,y\in X(H)\subset
X(H^{\prime})$ we could measure the distance between $x$ and $y$ two ways;
$\delta_{H}(x,y)$ and $\delta_{H^{\prime}}(x,y).$ In fact, however, those two
values are the same, and a similar comment holds for many of the invariants we
consider. The invariant $\delta_{H}$ defined by (\ref{def d}), as well as
$A_{ijk}$, and $LF_{ijk},$ are defined using entries of the Gram matrix $G(H)$
and those matrix entries do not change when $G(H)$ is included in the natural
way as a submatrix of $G(H^{\prime}).$ Other invariants, such $\delta_{H}$
defined using (\ref{extremal d}) or $\Delta_{H}$ defined in (\ref{delta h})
below, are defined using extremal problems which involve quantifying over all
elements of $H.$ In those cases, the analogous extremal problem for
$H^{\prime},$ involving quantifying over all of $H^{\prime},$ is not formally
equivalent to the first problem. However in the problems we consider the two
different extremal problems produce the same extremal value. That happens
because in those problems if $h^{\prime}\in H^{\prime}$ is a candidate to
solve the extremal problem formulated in $H^{\prime},$ then $h,$ the
orthogonal projection of $h^{\prime}$ onto $H,$ will give a superior
candidate, one that meets the same conditions and has smaller norm. In those
cases the larger set of candidates affects neither the value of the extremal,
nor even the identity of the extremal function.

The situation with invariants such as $\delta_{G}$ defined in
(\ref{extremal g}) and $\Delta_{G}$ defined in (\ref{delta g}) is more subtle.
It $H$ is a subspace of $H^{\prime}$ then, algebraically,
$\operatorname*{Mult}\left(  H\right)  $ is the quotient of
$\operatorname*{Mult}\left(  H^{\prime}\right)  $ by the ideal of functions
which vanish on $X(H^{\prime})\smallsetminus X(H).$ However, in general there
is no reason that the quotient norm should agree with the operator norm on
$\operatorname*{Mult}\left(  H\right)  ,$ which is what would insure that the
values of $\delta_{G}$ and $\Delta_{G}$ were not influenced by bringing the
larger space $H^{\prime}$ into consideration. In fact, it is exactly the
statement that $H^{\prime}$ has the CPP which insures that the quotient norm
for $\operatorname*{Mult}\left(  H\right)  $ is the same as that operator
norm. In all the cases where we consider a space $H$ and there is a larger,
containing, space $H^{\prime}$ lurking in the discussion, this will be the case.

\subsection{Extremal Problems and{}\label{extremal problems} Generalized
Distances}

We described distance $\delta_{G}$ and $\delta=\delta_{H}$ on $X$ in terms of
extremal problems (\ref{extremal g}) and (\ref{extremal d}). We now introduce
generalizations of those quantities. For $x,y,z\in X$ set
\begin{align}
\Delta_{H}(x;y,z)  &  =\frac{1}{\left\Vert k_{x}\right\Vert }\sup\left\{
\operatorname{Re}j(x):j\in H\text{, }j(y)=j(z)=0,\left\Vert j\right\Vert
_{H}\leq1\right\}  .\label{delta h}\\
\Delta_{G}(x;y,z)  &  =\sup\left\{  \operatorname{Re}m(x):m\in A\text{,
}m(y)=m(z)=0,\left\Vert m\right\Vert _{A}\leq1\right\}  \label{delta g}%
\end{align}

\noindent Both of these are invariant.

Suppose $m_{\delta}\in\operatorname*{Mult}\left(  H\right)  $ and $h_{\delta
}\in H $ \ are the functions which attain the extreme values in
(\ref{extremal g}) and (\ref{extremal d}) respectively. It then follows from
the definitions that $m_{\delta}\hat{k}_{x}$ is a competitor for the extremal
problem which defines $h_{\delta},$ and hence $\delta_{G}\leq\delta_{H}.$ \ A
completely analogous argument, with $M_{\Delta}\in\operatorname*{Mult}\left(
H\right)  $ and $H_{\Delta}\in H$ the extremal functions for the problems
(\ref{delta g}) and (\ref{delta h}), shows that $\Delta_{G}\leq\Delta_{H}.$

The distinctive feature of RKHS with the CPP is that there is a particularly
close relation between extremal problems in the multiplier algebra and in the
space. In the particular case we just described the two inequalities are, in
fact, equalities. The following is a special case of \cite[Theorem 9.33]{AM}.

\begin{proposition}
\label{same delta}If $H$ has the CPP then the functions $m_{\delta},$
$h_{\delta}.$ $M_{\Delta}$ and $H_{\Delta}$ are unique and satisfy
\begin{equation}
m_{\delta}\hat{k}_{x}=h_{\delta}\text{ and }M_{\Delta}\hat{k}_{x}=H_{\Delta}.
\label{mk=h}%
\end{equation}
In particular
\begin{equation}
\delta_{G}=\delta_{H}\text{ and }\Delta_{G}=\Delta_{H}. \label{d=d}%
\end{equation}

\end{proposition}

It is straightforward to solve the extremal problem (\ref{extremal d}) and
obtain a formula for $h_{x}.$ Using that and (\ref{mk=h}) then gives a formula
for $m_{x}.$ The two formulas are:%

\begin{align}
k_{\delta}(\cdot)  &  =\frac{1}{\left\Vert k_{x}\right\Vert \delta
(x,y)}\left(  k_{x}\left(  \cdot\right)  -\frac{k_{xy}k_{y}\left(
\cdot\right)  }{k_{yy}}\right) \label{hilbert space formula}\\
m_{xy}(\cdot)  &  =m_{\delta}(\cdot)=\frac{1}{\delta(x,y)}\left(
1-\frac{k_{xy}k_{y}\left(  \cdot\right)  }{k_{yy}k_{x}\left(  \cdot\right)
}\right)  . \label{multiplier formula}%
\end{align}

There are also some simple relations between the $\delta$'s and the $\Delta
$'s; for $x,y,z\in X,$%
\begin{align}
\delta_{G}(x,y)\delta_{G}(x,z)  &  \leq\Delta_{G}(x;y,z)\leq\delta
_{G}(x,y)\wedge\delta_{G}(x,z),\label{m inequality}\\
\Delta_{H}(x;y,z)  &  \leq\delta_{H}(x,y)\wedge\delta_{H}(x,z).\nonumber
\end{align}
The left inequality in the first line holds because the product of competitors
in the extremal problems defining the $\delta_{G}$'s is a competitor for the
extremal problem defining $\Delta_{G}.$ The other estimates hold because of
the monotonicity of the solution to a restricted maximum problem when the
restrictions are loosened.

\subsubsection{Evaluating $\Delta_{G}$ and $\Delta_{H}$}

We now evaluate the quantities $\Delta_{G}$ and $\Delta_{H}$ for spaces $H$
with the CPP. Thus $H\sim DA_{n}(X).$ and by Proposition \ref{same delta}
$\delta_{G}=\delta_{H},$ $\Delta_{G}=\Delta_{H}.$ We will generally drop the subscripts.

\begin{theorem}
\label{three point extremals}Suppose $H$ is a RKHS with the CPP and
$X(H)=\left\{  x_{i}\right\}  _{i=1}^{n}.$ Then
\begin{align}
\Delta_{G}^{2}=\Delta_{H}^{2}  &  =\frac{1}{\delta_{23}^{2}}\left(
\delta_{23}^{2}+\delta_{12}^{2}+\delta_{13}^{2}-2+2\operatorname{Re}%
\widehat{k_{12}}\widehat{k_{23}}\widehat{k_{31}}\right) \label{deltaa}\\
&  =\delta_{12}^{2}\delta_{23}^{-2}\left(  \delta_{13}^{2}-LF_{123}^{2}\left(
1-\delta_{23}^{2}\right)  \right) \label{delta with fl}\\
&  =\delta_{12}^{2}\left(  \delta_{23}^{-2}\left(  \delta_{13}^{2}%
-LF_{123}^{2}\right)  +LF_{123}^{2}\right)  \label{if lf - 1}%
\end{align}

\end{theorem}

In the next section describe geometric conditions on $X\ $ which correspond to
having $\Delta_{G}=\Delta_{H}$ simplify to $\delta_{12}\delta_{13}/\delta
_{23}$, or to $\delta_{12}\delta_{13},\ $or to $\delta_{12}$.

\begin{proof}
Taking note of the discussion in Section \ref{larger spaces} we may assume $H
$ is three dimensional. Using the definitions and some algebra, including the
fact that $2\operatorname{Re}k_{23}=\left\vert 1-k_{23}\right\vert
^{2}-1-\left\vert k_{23}\right\vert ^{2},$ the formulas (\ref{deltaa}) and
(\ref{delta with fl}) are equivalent. Line (\ref{if lf - 1}) is an algebraic
rewriting of (\ref{delta with fl}) which will be convenient later.

We now compute $\Delta_{H}^{2}.$ Let $v\in H$ be the function which takes the
values $1,0,0$ \ at $x_{1},$ $x_{2},$ and $x_{3}.$ $v$ spans the one
dimensional subspace of functions in $H$ which vanish at $x_{2}$ and $x_{3}.$
Hence $v/\left\Vert v\right\Vert $ is the extremal function in the problem
defining $\Delta_{H}$ and so $\Delta_{H}=\left(  \left\Vert k_{1}\right\Vert
\left\Vert v\right\Vert \right)  ^{-1}.$ We now compute $\left\Vert
v\right\Vert .$ The vector $v$ can be written as $v=\sum\nolimits_{i=1}%
^{3}b_{i}k_{i}$ for scalars $\left\{  b_{i}\right\}  .$ By evaluating at the
$x_{i}$ and comparing with $V=(1,0,0)$ we get a system of equations for the
$\left\{  b_{i}\right\}  \ $which we write in matrix form. Let $K=G(H)$ and
set $B=(b_{1},b_{2},b_{3}).$ Here and later we will use $T^{t}$ to denote the
transpose of the matrix $T.$ We have $BK=V$ and hence, \ setting
$K^{-1}=\left(  \gamma_{ij}\right)  ,$ we have
\begin{align*}
\left\Vert V\right\Vert ^{2}  &  =BK\bar{B}^{t}=\left(  VK^{-1}\right)
K\overline{\left(  VK^{-1}\right)  }^{t}\\
&  =V\overline{K^{-1}}^{t}\overline{\,V^{t}}=VK^{-1}V^{t}=\left(  \gamma
_{11}\right)  ,
\end{align*}
Thus our solution is
\[
\Delta_{H}^{2}=\frac{1}{k_{11}\gamma_{11}}.
\]
We now compute $\gamma_{11}$ using Cramer's rule.

Let $K_{1\rightarrow V^{t}}$ be the matrix obtained from $K$ by replacing the
first column of $K$ with the column $V^{t}.$ Cramer's rule tells us that
$\gamma_{11}=\det$ $K_{1\rightarrow V^{t}}/\det K.$ Thus our solution is
\begin{align}
\Delta_{H}^{2}  &  =\frac{\det K}{k_{11}\det K_{1\rightarrow V^{t}}%
}\label{delta}\\
&  =\frac{k_{11}k_{22}k_{33}+2\operatorname{Re}k_{12}k_{23}k_{31}%
-k_{11}\left\vert k_{23}\right\vert ^{2}-k_{22}\left\vert k_{13}\right\vert
^{2}-k_{33}\left\vert k_{12}\right\vert ^{2}}{k_{11}\left(  k_{22}%
k_{33}-\left\vert k_{23}\right\vert ^{2}\right)  }.\nonumber
\end{align}
Dividing top and bottom by $k_{11}k_{22}k_{33}$ we get
\[
\Delta_{H}^{2}=\frac{1+2\operatorname{Re}\widehat{k_{12}}\widehat{k_{23}%
}\widehat{k_{31}}-|\widehat{k_{23}}|^{2}-|\widehat{k_{12}}|^{2}%
-|\widehat{k_{13}}|^{2}}{1-|\widehat{k_{23}}|^{2}}.
\]
Recalling that $\delta_{ij}^{2}=1-|\widehat{k_{ij}}|^{2}$ we can rewrite that
as%
\begin{equation}
\delta_{23}^{2}\Delta_{H}^{2}=\delta_{23}^{2}+\delta_{12}^{2}+\delta_{13}%
^{2}-2+2\operatorname{Re}\widehat{k_{12}}\widehat{k_{23}}\widehat{k_{31}},
\label{formula for H}%
\end{equation}
which is what we wanted.

Finally, by Proposition \ref{same delta} we also obtain the result for
$\Delta_{G}.$

An alternative proof, computing $\Delta_{G}^{2}$ using the Pick matrix of the
associated multiplier extremal problem, is of comparable length.
\end{proof}

\section{Modifying Spaces and Sets}

\subsection{Rescalings and Automorphisms; Normal Form\label{normal}}

The involutive automorphisms of the ball, $\varphi_{a},$ satisfy a number of
useful identities \cite{Ru}.\ For $a,z,w\in\mathbb{B}^{n},$ and $k$ the
$DA_{n}$ kernel function,
\begin{align}
\left\vert \varphi_{a}\left(  z\right)  \right\vert ^{2}  &  =1-\frac
{(1-\left\vert a\right\vert ^{2})(1-\left\vert z\right\vert ^{2})}{\left\vert
1-\overline{a}\cdot z\right\vert ^{2}},\label{Ru1}\\
\frac{1}{1-\left(  \varphi_{a}\left(  w\right)  ,\varphi_{a}\left(  z\right)
\right)  }  &  =\frac{\left(  1-\left(  w,a\right)  \right)  }{\left(
1-\left(  a,a\right)  \right)  ^{1/2}}\frac{\left(  1-\left(  a,z\right)
\right)  }{\left(  1-\left(  a,a\right)  \right)  ^{1/2}}\frac{1}{\left(
1-\left(  w,z\right)  \right)  },\label{Ru2}\\
k_{z}(w)  &  =\frac{k_{z}\left(  a\right)  }{k_{a}(a)^{1/2}}\frac
{\overline{k_{w}\left(  a\right)  }}{k_{a}(a)^{1/2}}k_{\varphi_{a}\left(
z\right)  }(\varphi_{a}\left(  w\right)  ). \label{Ru3}%
\end{align}
\qquad

There is a natural identification of $X(DA_{n})$ with $\mathbb{B}^{n}.$ Using
that identification the metric $\delta_{DA_{n}}$ can be regarded as a metric
on $\mathbb{B}^{n}$ and that metric equals the pseudohyperbolic metric $\rho$
on $\mathbb{B}^{n}=\mathbb{CH}^{n}.$ This can be seen from (\ref{Ru1}) where
the left side is $\rho(a,z)^{2},$ the square of the pseudohyperbolic distance
between $a$ and $z,$ and the right side is $\delta_{DA_{n}}(a,z)^{2}.$

By comparing (\ref{rescaling 2}) and (\ref{Ru3}) we see that automorphisms of
the ball induce rescalings$;$ if $X=\left\{  x_{i}\right\}  $ is a finite
subset of $\mathbb{B}^{n}$ and $\Phi\in$ $\operatorname*{Aut}\left(
\mathbb{B}^{n}\right)  $ then $DA_{n}(X)\sim DA_{n}(\Phi\left(  X\right)  ).$
We now introduce a notion of normal form for a set in $\mathbb{B}^{n}$ and use
it to prove a converse statement; if $Y\subset\mathbb{B}^{n}$ and
$DA_{n}(X)\sim DA_{n}(Y)$ then $X$ and $Y$ are congruent, $X\sim Y.$

We say a finite ordered set $X=\left\{  x_{i}\right\}  _{i=1}^{M}%
\subset\mathbb{CH}^{n}=\mathbb{B}^{n}$ is in normal form, $X\in\mathcal{N},$
if the coordinate description of $X$ with respect to the standard orthonormal
basis, $\left\{  e_{i}\right\}  _{i=1}^{n},$ of $\mathbb{C}^{n}$ takes the
following roughly triangular form. The first point, $x_{1,}$ is at the origin,
and the coordinates of the remaining points have the form
\[
x_{j}=(\alpha_{j1},\alpha_{j2},...,\alpha_{jN(j)},0,...,0)
\]
with $\left\{  N(k)\right\}  $ a nondecreasing. sequence with differences,
$N(k+1)-N(k),$ always $0$ or $1.$ We further require the positivity conditions
that if $N(k+1)>N(k)$ then $a_{(k+1)\,N(k+1)}>0.$

Let $\mathcal{N}$ be collection of sets in normal form$.$

\begin{proposition}
Suppose $X$ is a finite ordered set, $X=\left\{  x_{i}\right\}  _{i=1}^{M},$
contained in $\mathbb{B}^{n}=\mathbb{CH}^{n}.\ $There is a unique $\Psi_{X}%
\in\operatorname*{Aut}\left(  \mathbb{B}^{n}\right)  $ such that $Y=\Psi
_{X}\left(  X\right)  \in\mathcal{N}$. In particular there is exactly one
$Y\in\mathcal{N}$ with $Y\sim X$.
\end{proposition}

\begin{proof}
First apply the involution $\varphi_{x_{1}}$ to $X.$ That produces a congruent
set with (the new) $x_{1}$ at the origin. Now split $X$ as a disjoint union
$X=\left\{  y_{i}\right\}  _{i=1}^{r}\cup\left\{  z_{j}\right\}  _{j=1}%
^{s}=Y\cup Z$. The set $Y$ is constructed by setting $y_{1}=x_{1}=0$ and then
going through the remaining $x_{r}$'s in the order of their indices and
designating each $x_{r}$ to be the next $y_{i}$ if that $x_{r}$ is not in the
linear span of the $y_{i}$ already selected. Otherwise put $x_{r}$ in $Z.$
Thus, for instance, $y_{2}=x_{2}.$ Now set $v_{1}=0$ and apply the
Gram-Schmidt process to the vectors $y_{2}.....y_{r}$ to produce an
orthonormal sequence $v_{2},....,v_{j}$ with $j-1\leq n.$ If $j-1<n$ then
complete the sequence in an arbitrary way to an orthonormal basis of
$\mathbb{C}^{n}.$ The structure of the Gram-Schmidt process insures that the
coordinate representation of the $\left\{  x_{i}\right\}  $ with respect to
the basis $\left\{  v_{j}\right\}  $ have nonzero entries in the pattern
required for a set in $\mathcal{N}$. Next, replace the basis $\left\{
v_{j}\right\}  $ with an orthonormal basis $\left\{  \alpha_{j}v_{j}\right\}
$ where the $\left\{  \alpha_{i}\right\}  $ are unimodular constants selected
so that the coordinate entries in the positions where positivity is required
are, in fact, positive. This is possible because the positivity rule requires
that each $v_{j}$ be modified at most once.

If the basis $\left\{  \alpha_{j}v_{j}\right\}  $ which we constructed
happened to be the canonical basis $\left\{  e_{i}\right\}  $ we would be
done. Otherwise we now move $X$ using the unitary map $U$ which takes the
elements $\left\{  \alpha_{j}v_{j}\right\}  $ to the elements $\left\{
e_{i}\right\}  .$ This is possible because any two orthonormal bases of
$\mathbb{C}^{n}$ are connected by a unitary map. Because the $\left\{
x_{i}\right\}  $ are linear combinations of the $\left\{  \alpha_{j}%
v_{j}\right\}  $ with coefficients having the desired pattern, the points
$\left\{  Ux_{i}\right\}  $ have coordinate representations in the desired
pattern with respect to the basis $\left\{  U\alpha_{j}v_{j}\right\}
=\left\{  e_{i}\right\}  .$ Finally, recall that any unitary map is in
$\operatorname*{Aut}\left(  \mathbb{B}^{n}\right)  .$ Combining $U$,
$\varphi_{x_{1},}$ and the rotations used to generate the $\alpha_{j}$
produces the required $\Psi_{X}.$

Suppose now there were another automorphism $\widetilde{\Psi}_{X}$ with
$\widetilde{\Psi}_{X}(X)\in\mathcal{N}$. Consider the automorphism
$\Lambda=\widetilde{\Psi}_{X}\Psi_{X}^{-1}.$ Tracing through the definitions
shows $\Lambda(0)=0$ hence $\Lambda$ is a unitary map. Tracing the definitions
again shows that $\Lambda e_{1}$ must be a positive multiple of $e_{1};$ but
$\Lambda$ is unitary and hence $\Lambda e_{1}=e_{1}.$ This pattern continues
through the $e$'s and that is enough to conclude that $\Lambda$ is the
identity on the span of $\Psi_{X}(X).$ That establishes the uniqueness of
$\Psi_{X}$ and hence of the normal form.
\end{proof}

In the proof we possibly did not use all of the dimensions of $\mathbb{B}^{n}
$.

\begin{corollary}
If $X\subset\mathbb{B}^{n},$ $\left\vert X\right\vert =k$ then $X$ $\sim Y$
for some $Y$ in the $\mathbb{B}^{k-1}$in $\mathbb{B}^{n}$ consisting of all
points with their last $n-k+1$ coordinates zero.
\end{corollary}

\begin{corollary}
\label{small space}If $H\sim DA_{n}(X)$ for some $X$ and $\dim(H)=k$ then
$H\sim DA_{k-1}(Y)$ for some $Y.$
\end{corollary}

\begin{theorem}
\label{reduction}Suppose $X=\left\{  x_{i}\right\}  $ and $Y=\left\{
y_{i}\right\}  $ are ordered finite sets in $\mathbb{CH}^{n},$ The following
are equivalent:

\begin{enumerate}
\item $X$ is congruent to $Y:X\sim Y.$

\item $X$ and $Y$ have the same normal forms: $\Psi_{X}X=\Psi_{Y}Y.$

\item The spaces $DA_{n}(X)$ and $DA_{n}(Y)$ are rescalings of each other:
$DA_{n}(X)\sim DA_{n}(Y).$

\item The Gram matrices of the associated spaces are equivalent:
$G(DA_{n}(X))\sim G(DA_{n}(Y)).$

\item The triangles of $X$ are congruent to the triangles of $Y:$ For any
triple $i,$ $j.k$ there is a $\Gamma_{ijk}\in\operatorname*{Aut}\left(
\mathbb{B}^{n}\right)  $ taking $\{x_{i},$ $x_{j},$ $x_{k}\}$ to $\{y_{i},$
$y_{j},$ $y_{k}\}$.
\end{enumerate}
\end{theorem}

With this result as background, the discussion going forward is in the spirit
of Klein's \textit{Erlangen Program. }The geometry of $X\subset\mathbb{CH}%
^{n}$\textit{\ }is described by numerical data that is invariant under the
automorphism group of $\mathbb{CH}^{n}$. The structure of a RKHS $H $ is
described by numerical data invariant under the rescaling group. Much of the
work here focuses of identifying useful invariants and establishing a
dictionary between analytic and geometric invariants.

\begin{proof}
If \textbf{(1) }holds, and thus $Y=\Lambda X$ for $\Lambda\in
\operatorname*{Aut}\left(  \mathbb{B}^{n}\right)  ,$ then $\Psi_{Y}Y=\Psi
_{Y}\Lambda X$ is in $\mathcal{N}$ and is congruent to $X.$ Hence by the
uniqueness statement in the previous proposition, $\Psi_{Y}\Lambda X=\Psi
_{X}X.$ Hence \textbf{(1)} implies \textbf{(2)}. If \textbf{(2)} holds then
$DA_{n}(\Psi_{X}X)=DA_{n}(\Psi_{Y}Y),$ Also, as we noted earlier, $X\sim
\Psi_{X}X$ implies that $DA_{n}(X)\sim DA_{n}(\Psi_{X}X),$ with a similar
statement for $Y.$ Combining those equivalences we see that \textbf{(3)
}holds. If \textbf{(3) }holds then, by formula (\ref{modified gram}%
)\textbf{\ }and the discussion surrounding it, \textbf{(4)} holds.

We now show that \textbf{(4) }implies \textbf{(1). } We know from the previous
paragraph that $G(DA_{n}(X))\sim G(DA_{n}(\Psi_{X}X))$ and similarly for $Y.$
Hence we can replace \textbf{(4)} with $G(DA_{n}(\Psi_{X}X))\sim G(DA_{n}%
(\Psi_{Y}Y)).$ Consider now $G(DA_{n}(\Psi_{X}X)).$ The set $\Psi_{X}X$ has
its point $\Psi x_{1}$ at the origin and hence $DA_{n}(\Psi_{X}X)$ is
basepoint normalized with $\Psi x_{1}$ as the basepoint, similarly with $Y.$
We noted earlier that Gram matrices of basepoint normalized spaces are
equivalent if and only if they are equal. Hence we are reduced to the case of
equal Gram matrices.. Thus we will be finished if we can show that if
$Z\in\mathcal{N}$ then $G=G(DA_{n}(Z))$ determines $Z.$ The matrix $G$ has
entries $k_{ij}=(1-\left\langle z_{j},z_{i}\right\rangle )^{-1}$ and hence
knowing $G$ insures that we know the matrix $\left(  \left\langle z_{j}%
,z_{i}\right\rangle \right)  _{i,j=2}^{m}.$ This matrix is the Gram matrix of
a set of $m-1$ points in $\mathbb{C}^{n}$ and hence determines that set of
points up to unitary equivalence. In the case of interest to us, $Z=\Psi
_{X}X,$ the set is assumed to be in normal form and that removes the ambiguity
associated with the unitary equivalence.

Certainly \textbf{(1) }implies\textbf{\ (5). }To finish we show that
\textbf{(5) }implies\textbf{\ (4). }To do this it suffices to show that if $X$
and $Y $ are both in normal form then they have the same Gram matrix. The
first row and first column of those matrices agree by construction. Select
$i,j>1,$ $i\neq j$ and consider the triple $X_{ij}=\left\{  x_{1}%
=0,x_{i},x_{j}\right\}  $ and similarly for $Y_{ij}.$ By assumption the two
are congruent. Hence the invariant data set $\delta=\delta(X_{ij})$ defined as
in Section \ref{describing} is equal to the corresponding set $\delta
(Y_{ij}).$ As noted there, this implies the corresponding data sets
$\kappa(X_{ij})$ and $\kappa(Y_{ij})$ also agree, Hence, also, the associated
three by three Gram matrices agree. Further, the elements in those small
matrices are determined by the position of the points, independently of any
containing superset. Hence the corresponding entries on the Gram matrices for
$X$ and $Y$ agree.
\end{proof}

The previous result is specific to finite dimensional spaces. If $X$ is
infinite then $DA_{n}(X)$ gives more complicated \ information about $X.$ For
instance, if $X\subset\mathbb{CH}^{1}=\mathbb{D}$ satisfies $DA_{1}%
(X)=DA_{1}=H^{2}$ then we can only conclude that $X$ contains a sequence which
fails the Blaschke condition. More information about the general, infinite
dimensional, situation is in \cite{Sh}.

In \cite{HS}, \cite{BE}, and \cite{G} the authors study congruence classes of
finite point sets in $\mathbb{CH}^{n}\ $and obtain results that are similar to
the equivalence of conditions \textbf{(1),\ (4)}, and \textbf{(5)} in the
previous theorem. Their proofs follow the same general line as the previous
proof; they move from the point set to an associated matrix, develop an
appropriate notion of normal form for the matrix, and show that equality of
the normal forms is equivalent to the congruence of the sets. However the
details of their analysis differ. They view $\mathbb{CH}^{n}$ as the negative
points of $\mathbb{CP}^{n}$ and study the matrix $([x_{i},x_{j}])$ using tools
from projective geometry. We view $\mathbb{CH}^{n}$ as the ball in
$\mathbb{C}^{n}$ and use Euclidean coordinate geometry to study the matrix
$\left(  k_{ij}\right)  =(1/[x_{i},x_{j}]).$

Using this theorem we see two sets of data which can be used to describe $X$
up to congruence. The first, $E(X)$, is the set of $\left(  n-1\right)  ^{2}$
real numbers which specify the Euclidean coordinates of the points of $X$ in
normal form. This is an inductive description of the set, adding points to the
set one at a time and describing each new point by its relation to the
previous points. It is similar in spirit to an inductive description which was
suggested by Hakin and Sandler in \cite{HS}. \ A second set of data which
describes $X$ is $J(X)$ introduced in (\ref{G(X)}). Taking into account the
cocycle identity for angular invariants $J(X)$ is described by $\left(
n-1\right)  ^{2}$ real numbers. That data is rescaling invariant and
determines the Gram matrix of a rescaled version of $H$. Those numbers are
also invariant under automorphisms and hence should be viewed as geometric
descriptors of $X.$ In particular, considering the previous theorem and the
discussion in Section \ref{describing}, we see that this data determines the
congruence class of triangles with vertices in $X$ and that data determines
$X.$

The Euclidean parameters $E(X)$ do not clearly capture the hyperbolic geometry
of $X,$ but they do allow a very simple description of which parameter sets
are attainable. In contrast, the set $J(X),$ which contains explicit
information about the hyperbolic geometry, does not give a clear vision of the
allowable parameter set. The description for three point sets is given in
(\ref{B}) of Theorem \ref{n=3}, but the situation for $n>3$ is unclear.

\subsection{The Conjugate Space, $\overline{H}$\label{conjugations}}

A RKHS, $H,$ consists of a vector space, a Hermitian inner product$,$ and a
distinguished basis, called reproducing kernels. In this section and the next
we describe two ways of constructing a new RKHS from $H$; one by modifying the
inner product, the other by changing to a new set of reproducing kernels. If
$K$ is the Gram matrix for $H$ then the new spaces will have Gram matrices
$\overline{K}$ and $K^{-1}$ respectively$.$ We then discuss the particularly
interesting case when the two constructions give identical spaces. That
happens when the matrix $K$ is orthogonal, $\overline{K}=K^{t}=K^{-1}.$

Given $H,$ we define $\overline{H},$ the \textit{conjugate space} of $H,$ to
be the RKHS formed using the same vector space, the same set of vectors as
reproducing kernels, but a different Hermitian inner product, $\left[
\cdot,\cdot\right]  ,$ defined by
\begin{equation}
\left[  k_{i},k_{j}\right]  =\overline{\left\langle k_{i},k_{j}\right\rangle
}. \label{bracket}%
\end{equation}
It is immediate that $G\left(  \overline{H}\right)  =\overline{G(H)}%
=G(H)^{t}.$ It is also immediate that the \underline{conjugate}
\underline{linear} map $\Lambda$ defined by
\begin{equation}
\Lambda\left(  \sum\alpha_{i}k_{i}\right)  =\sum\bar{\alpha}_{i}k_{i}.
\label{lambda}%
\end{equation}
is an isometry from $H$ to $\overline{H};$ that is
\[
\left\Vert \sum\alpha_{i}k_{i}\right\Vert _{H}^{2}=\sum\alpha_{i}\bar{a}%
_{j}\left\langle k_{i}.k_{j}\right\rangle =\sum\bar{\alpha}_{i}a_{j}\left[
k_{i},k_{j}\right]  =\left\Vert \sum\bar{\alpha}_{i}k_{i}\right\Vert
_{\overline{\overline{H}}}^{2}.
\]
If $H$ has the CPP and thus satisfies $H\sim DA_{n}(X)$ for some
$X\subset\mathbb{C}^{n}$ then $\overline{H}$ also has the CPP and satisfies
$\overline{H}\sim DA_{n}(\overline{X}).$ Here $\overline{X}$ is the set of
points obtained by expressing the points of $X$ in terms of coordinates with
respect to standard basis and then conjugating those coordinates. In fact, if
we knew from the start that $H$ had the CPP and thus $H\sim DA_{n}(X)$ then we
could have based the construction of $\overline{H}$ on the conjugate linear
isometry of $\mathbb{CH}^{n}$ given by conjugating the coordinates.

If $X=\overline{X}$ then $H\sim\overline{H}.$ This holds, for instance, for
the Hilbert spaces of functions on trees which we discuss in Section
\ref{tree spaces} and for the RKHS obtained as subspaces of the diameter
spaces of \cite{ARS07}.

\subsection{The Dualized Space, $H^{\#}$}

A RKHS $H$ is a Hilbert space together with the distinguished basis
$B=B(H)=\left\{  k_{i}\right\}  .$ Associated with $B$ is the \textit{dual
basis} $B^{\#}=\left\{  f_{j}\right\}  $ defined by the requirement that
$\left\langle k_{i},f_{j}\right\rangle =\delta_{ij}.$ We define the
\textit{dualized space} $H^{\#}$ to be the RKHS obtained by using the same
Hilbert space, $H,$ but selecting $B^{\#}$ as the distinguished basis rather
than $B. $

Let $K$ be the Gram matrix of $H$ and $K^{\#}$ the Gram matrix of $H^{\#},$
$K^{\#}=\left(  \left\langle f_{i},f_{j}\right\rangle \right)  =\left(
f_{ij}\right)  .$ Let $\Theta=\left(  \theta_{ij}\right)  $ be the matrix
which takes $B$ to $B^{\#};$ for all $i$
\begin{equation}
f_{i}=\sum\nolimits_{j}\theta_{ij}k_{j} \label{f equals}%
\end{equation}
The transformation in the other direction is then given by $\Theta
^{-1}=\left(  \gamma_{ij}\right)  ,$ $k_{i}=\sum_{j}\gamma_{ij}f_{j}.$

\begin{proposition}%
\begin{align}
\Theta K  &  =I\\
K^{\#}  &  =\Theta K\Theta^{\ast}\nonumber
\end{align}
Hence the matrices $K,\ K^{\#}$ and $\Theta$ are self adjoint and
\begin{equation}
\Theta=K^{\#}=K^{-1}. \label{later}%
\end{equation}

\end{proposition}

\begin{proof}
The calculation
\[
\delta_{ij}=\left\langle f_{i},k_{j}\right\rangle =\left\langle \sum
\nolimits_{s}\theta_{is}k_{s},k_{j}\right\rangle =\sum\nolimits_{s}\theta
_{is}k_{sj},
\]
gives the first equation. The second follows from
\[
\left\langle f_{i},f_{j}\right\rangle =\left\langle \sum\nolimits_{s}%
\theta_{is}k_{s}.\sum\nolimits_{t}\theta_{jt}k_{t}\right\rangle =\sum
\nolimits_{s,t}\theta_{is}k_{st}\overline{\theta_{jt}}.
\]

\end{proof}

\subsection{Orthogonal Spaces and Pick Spaces\label{orthogonal}}

Associated with the construction of $H^{\#}$ is a mapping $\Omega,$ the
\underline{conjugate} \underline{linear} map from $H$ to itself that takes the
reproducing kernel basis $\left\{  k_{i}\right\}  \ $to to the dual basis
$\left\{  f_{i}\right\}  :$
\begin{equation}
\Omega(\sum\alpha_{i}k_{i})=\sum\overline{\alpha_{i}}\Omega(k_{i}%
)=\sum\overline{\alpha_{i}}f_{i}. \label{Omega}%
\end{equation}
Using this operator and $\Lambda$ defined by (\ref{lambda}) we define the
operator $S=\Omega\Lambda.$ Thus $S:\overline{H}\rightarrow H^{\#},$
\[
S(\sum\alpha_{i}k_{i})=\sum\alpha_{i}f_{i}\text{\ }%
\]

A conjugate linear map $\Gamma$ from a Hilbert space to itself which is an
involution, i.e. $\Gamma^{2}=1,$ and an isometry, i.e.$\forall h$ $\left\Vert
\Gamma h\right\Vert =\left\Vert h\right\Vert ;$ is called a
\textit{conjugation. }We will be particularly interested in cases where the
operator $\Omega$ we just defined is a conjugation. Because $\Omega$ has the
additional structural property of taking the kernel basis to the dual basis
the conditions for $\Omega$ to be a conjugation simplify.

\begin{theorem}
\label{equivalent}Suppose $\Omega$ is defined by (\ref{Omega}). The following
are equivalent:

\begin{enumerate}
\item The matrix $K$ is orthogonal: $K^{t}K=\bar{K}K=I,$

\item Let $\sigma=\sum_{i=1}^{n}k_{i}$. The matrix $\left(  k_{i}k_{j}%
,\sigma\right)  $ is the identity.

\item $\Omega$ is an isometry of $H$: $K^{t}=\Theta K\Theta^{\ast}.$

\item $\Omega$ is an involution of $H$: $\Theta^{t}\Theta=\bar{\Theta}%
\Theta=I.$

\item $\Omega$ is a conjugation of $H.$

\item $S$ is an an isometry between $\overline{H}$ and$\ H^{\#}$.
\end{enumerate}
\end{theorem}

\begin{proof}
The second statement is a rewriting of the first. The definition of $\Omega$
together with the previous proposition shows that the equations in statements
one, three, and four are equivalent, we must show why the verbal statements
correspond to the equations. The matrix $K$ is selfadjoint and hence the
equations in the first statement follow from the definition. For the third,
suppose $\Omega$ is an isometry. In that case we must have $\left\Vert
\sum\nolimits_{i}a_{i}k_{i}\right\Vert =\left\Vert \sum\nolimits_{i}\bar
{a}_{i}f_{i}\right\Vert .$ Squaring and expanding gives
\[
\sum\nolimits_{i}a_{i}\bar{a}_{j}\left\langle k_{i}.k_{j}\right\rangle
=\sum\nolimits_{i}\bar{a}_{i}a_{j}\left\langle f_{i}.f_{j}\right\rangle .
\]
The right hand side is real and hence we can replace it with its complex
conjugate. This produces an equality which will hold for all $\left\{
a_{i}\right\}  $ if and only if $\forall i,j$
\begin{equation}
\overline{\left\langle k_{i},k_{j}\right\rangle }=\left\langle f_{i}%
,f_{j}\right\rangle . \label{conj}%
\end{equation}
We now use (\ref{f equals}) in that equality to obtain $\bar{K}=\Theta
K\bar{\Theta}^{t},$ which is equivalent to the equation in the third
statement. Similarly straightforward calculations show that the fourth
statement, requiring that $\Omega\Omega h=h$ for a general $h\in H,$ is
equivalent to the equation $\Theta^{t}\Theta=I.$ The fifth statement is, by
definition, the union of two before it.

Given the definition of $S$ the final statement is equivalent to the equality
of inner products%
\[
\left[  k_{i}.k_{j}\right]  =\left\langle f_{i},f_{j}\right\rangle .
\]
Given the definition (\ref{bracket}) this is equivalent to (\ref{conj}) and
hence to the fourth statement.
\end{proof}

We will say that a RKHS $H$ is \textit{orthogonal} if any, and hence all, of
the conditions in the previous theorem hold. Thus the orthogonal $H$ are those
which have a conjugation operator taking reproducing kernel basis to the dual
basis. Using (\ref{conj}) we see that the orthogonal $H$ are also those for
which the linear map $S$ of $\overline{H}$ to $H^{\#}$ given by $S(\sum
\alpha_{i}k_{i})=\sum\alpha_{i}f_{i}$ is an isometry of RKHS. We say that a
RKHS is \textit{r-orthogonal} if it is a rescaling of an orthogonal $H.$

\begin{corollary}
Given a RKHS $H,$ either all or none of the spaces $\left\{  H,H^{\#}%
,\overline{H}\right\}  $\ are r-orthogonal.
\end{corollary}

We now show that every space of the form $DA_{1}(X),$ is an r-orthogonal RKHS.
In fact we do not know of any other examples. Later, in Theorem
\ref{three dimensional}, we will show that there are no other three
dimensional examples.

The spaces\ $DA_{1}(X)$ are exactly the generic finite dimensional model
spaces; that is, model spaces corresponding to finite Blaschke products with
only simple zeros. Model spaces are discussed systematically in \cite{GMR}.
Here we collect some facts about them and about conjugation operators acting
on them.

Recall that $DA_{1}$ is the classical Hardy space, $H^{2}.$ Given a finite
Blaschke product, $\Theta\in H^{2},$ the associated finite dimensional
\textit{model space }is the subspace $J_{\Theta}\subset H^{2}\ $which is the
orthogonal complement of $\Theta H^{2},$ $J_{\Theta}=H^{2}\ominus\Theta
H^{2}.$ If $\Theta$ has only simple zeros then $J_{\Theta}$ can be regarded as
a space of functions on $X=X_{\Theta}=\left\{  x_{i}\right\}  ,$ the zero set
for $\Theta.$ This space inherits from $H^{2}$ the structure of a RKHS, and
the reproducing kernel functions for $J_{\Theta}$ are the restrictions to $X$
of the Hardy space kernels. Thus $J_{\Theta}=$ $DA_{1}(X_{\Theta}).$ We will
call such a space, $DA_{1}(X)$ for a finite $X, $ a \textit{Pick space, }both
in recognition of the fact that the classical Pick interpolation theorem can
be cast as a theorem about the multiplier algebra of such a space, and in
parallel with the usage in \cite{CLW} where algebras isomorphic to multiplier
algebras of such a space are called \textit{Pick algebras}. We will call a
RKHS which is a rescaling of a Pick space an \textit{r-Pick space}.

\begin{theorem}
\label{is orthogonal}Any finite dimensional r-Pick space $H$ is r-orthogonal.
\end{theorem}

\begin{proof}
It is a basic fact about Pick spaces that each space carries a conjugation
operator taking the basis of reproducing kernels to a rescaled version of its
dual basis \cite{GMR}. Specifically, if we denote the basis of $J_{\Theta}$
consisting of reproducing kernels by $\left\{  j_{i}\right\}  $ and its dual
basis by $\left\{  g_{i}\right\}  ;$ $\left\langle j_{r},g_{s}\right\rangle
=\delta_{rs}$, then the conjugate linear map $\Omega$ which satisfies%
\begin{equation}
\Omega(j_{i})=\Theta^{\prime}(x_{i})g_{i} \label{omega j}%
\end{equation}
is a conjugation. Hence if we rescale $J_{\Theta}$ we obtain an orthogonal
space. Specifically, let $\widetilde{H}$ be the rescaling of $H$ which is the
same Hilbert space, but with the new distinguished basis of kernel functions
$\tilde{B}=\left\{  r_{j}\right\}  =\{\overline{\Theta^{\prime}(x_{i})^{-1/2}%
}j_{i}\}$. Direct computation shows that the dual basis of $\tilde{B},$
$\tilde{B}^{\#}=\left\{  s_{i}\right\}  ,$ is given by setting $s_{i}%
=\Theta^{\prime}(x_{i})^{1/2}g_{i},$ $i=1,..,n.$ Using (\ref{omega j}) we
check that $\Omega$ takes the basis $\tilde{B}$ to its dual basis $\tilde
{B}^{\#}:$
\[
\Omega(r_{i})=\Omega(\overline{\Theta^{\prime}(x_{i})^{-1/2}}j_{i}%
)=\Theta^{\prime}(x_{i})^{-1/2}\Omega(j_{i})=\Theta^{\prime}(x_{i}%
)^{-1/2}\Theta^{\prime}(x_{i})g_{i}=s_{i}%
\]

The rescaled space $\widetilde{H}$ has the same norm as $H$ and hence $\Omega$
is also isometric on $\widetilde{H}.$ Thus we have shown that the previous
theorem applies to $\Omega$ and that $\Omega$ satisfies condition \textbf{(3)}
of that theorem. Hence, by that theorem, $\Omega$ is a conjugation operator on
$\widetilde{H,}$ Thus $\widetilde{H}$ is orthogonal and hence our original
space, $H=J_{\Theta},$ is r-orthogonal.
\end{proof}

The previous result together with Theorem \ref{equivalent} shows that for
$X\subset\mathbb{D}$ there is a very close relation between the Gram matrix of
$DA_{1}(X)$ and the Gram matrix of $DA_{1}(X)^{\#}.$ That relationship has
been used very effectively in analysis of interpolating sequences for the
Hardy space; see \cite[9.5, 9.6]{AM} or \cite[Ch 5, Remark 26]{Sa}. The
explicit analyses there as well as the facts used here about model spaces make
crucial use of the theory of Blaschke products. It is not clear what, if any,
analogous results hold for spaces $DA_{n}(X),$ $n>1$.

\section{Embedding{}$\,X$ in $\mathbb{CH}^{n}$}

\subsection{The Strong Triangle Inequality}

The metric $\rho$ is not a length metric and so there is no reason to believe
equality could happen in the triangle inequality for $\rho$. In fact it never
does, and points in $\mathbb{B}^{n}$ satisfy a strengthened triangle
inequality, \textit{STI}. For any $a,b,c\in\mathbb{B}^{n}$%
\begin{equation}
\frac{\left\vert \rho(a,b)-\rho(b,c)\right\vert }{1-\rho(a,b)\rho(b,c)}%
\leq\rho(a,c)\leq\frac{\rho(a,b)+\rho(b,c)}{1+\rho(a,b)\rho(b,c)}.
\tag{STI}\label{STI}%
\end{equation}
One way to verify this is to note that the Poincare-Bergman metric, $\beta,$
on the disk is a length metric and so satisfies the standard triangle
inequality, including the possibility of equality.\ Further $\rho=\tanh
c\beta,$ Here $c$ is a constant which we set to one. (The choice $c=1/2$ is
also common.) \ Combining the addition theorem for $\tanh$ and the triangle
inequality for the metric $c\beta$ produces (\ref{STI})$.$ As this suggests,
the same configuration which produce equality in the triangle inequality for
$\beta,$ namely three points on the same hyperbolic geodesic, will also
produce equality in (\ref{STI})$,$ More discussion of $\rho,$ including a
free-standing proof of (\ref{STI}), is in \cite{DW}.

We are interested in understanding conditions on $H$ related to the
possibility that $H\sim DA_{n}(X),$ as in Theorem \ref{embed}. If there is
such an $X$ then the metric space $(X,\delta_{H})$ must satisfy the STI, so we
begin by examining that.

\begin{proposition}
\label{sti proposition}Suppose for $i,j=1,2,3$ we have $\delta_{ij}>0,$
$k_{ij,}$and $\widehat{k_{ij}}$ , and they are related by
\[
\text{ }\widehat{k_{ij}}=k_{ii}^{-1/2}k_{jj}^{-1/2}k_{ij},\text{ }\delta
_{ij}^{2}=1-|\widehat{k_{ij}}|^{2},
\]
then the following are equivalent:

\begin{enumerate}
\item
\begin{equation}
\frac{\left\vert \delta_{12}-\delta_{13}\right\vert }{1-\delta_{12}\delta
_{13}}\leq\delta_{23}\leq\frac{\delta_{12}+\delta_{13}}{1+\delta_{12}%
\delta_{13}}, \label{first}%
\end{equation}

\item
\begin{equation}
\left\vert 1-\left\vert \frac{k_{21}k_{13}}{k_{23}k_{11}}\right\vert
\right\vert =\left\vert 1-\left\vert \frac{\widehat{k_{21}}\widehat{k_{13}}%
}{\widehat{k_{23}}}\right\vert \right\vert \leq\delta_{12}\delta_{13},
\label{second}%
\end{equation}

\item
\begin{equation}
\frac{1}{|\widehat{k_{12}}|^{2}}+\frac{1}{|\widehat{k_{23}}|^{2}}+\frac
{1}{|\widehat{k_{13}}|^{2}}-1\leq\frac{2}{|\widehat{k_{12}}||\widehat{k_{23}%
}||\widehat{k_{13}}|}. \label{third}%
\end{equation}

\end{enumerate}
\end{proposition}

\begin{proof}
We square all three expressions in (\ref{first}), replace $\delta_{23}^{2}$ by
$1-\left\vert \widehat{k_{23}}\right\vert ^{2}$ and rearrange to obtain%
\begin{equation}
1-\left(  \frac{\delta_{12}-\delta_{13}}{1-\delta_{12}\delta_{13}}\right)
^{2}\geq\left\vert \widehat{k_{23}}\right\vert ^{2}\geq1-\left(  \frac
{\delta_{12}+\delta_{13}}{1+\delta_{12}\delta_{13}}\right)  ^{2}. \label{a}%
\end{equation}
Now note that
\[
1-\left(  \frac{\delta_{12}+\delta_{13}}{1+\delta_{12}\delta_{13}}\right)
^{2}=\frac{\left(  1-\delta_{12}^{2}\right)  \left(  1-\delta_{13}^{2}\right)
}{\left(  1+\delta_{12}\delta_{13}\right)  ^{2}}=\frac{\left\vert
\widehat{k_{12}}\right\vert ^{2}\left\vert \widehat{k_{13}}\right\vert ^{2}%
}{\left(  1+\delta_{12}\delta_{13}\right)  ^{2}}%
\]
and there is a similar formula for the left side of (\ref{a}). Hence from
(\ref{a}) we move to%
\begin{equation}
\frac{\left\vert \widehat{k_{12}}\right\vert ^{2}\left\vert \widehat{k_{13}%
}\right\vert ^{2}}{\left(  1-\delta_{12}\delta_{13}\right)  ^{2}}%
\geq\left\vert \widehat{k_{23}}\right\vert ^{2}\geq\frac{\left\vert
\widehat{k_{12}}\right\vert ^{2}\left\vert \widehat{k_{13}}\right\vert ^{2}%
}{\left(  1+\delta_{12}\delta_{13}\right)  ^{2}}.
\end{equation}
We now extract square roots, divide by $\left\vert \widehat{k_{12}}\right\vert
\left\vert \widehat{k_{13}}\right\vert ,$ take reciprocals, and rearrange to
obtain%
\[
1-\delta_{12}\delta_{13}\leq\frac{\left\vert \widehat{k_{12}}\right\vert
\left\vert \widehat{k_{13}}\right\vert }{\left\vert \widehat{k_{23}%
}\right\vert }\leq1+\delta_{12}\delta_{13},
\]
or, equivalently%
\begin{equation}
\left\vert \frac{\left\vert \widehat{k_{12}}\right\vert \left\vert
\widehat{k_{13}}\right\vert }{\left\vert \widehat{k_{23}}\right\vert
}-1\right\vert \leq\delta_{12}\delta_{13}, \label{intermediate}%
\end{equation}
which gives (\ref{second}). To obtain (\ref{third}) we square both sides of
(\ref{intermediate}) and replace the $\delta$'s with their definition in terms
of the $k$'s and obtain%
\[
\frac{\left\vert \widehat{k_{12}}\right\vert ^{2}\left\vert \widehat{k_{13}%
}\right\vert ^{2}}{\left\vert \widehat{k_{23}}\right\vert ^{2}}-2\frac
{\left\vert \widehat{k_{12}}\right\vert \left\vert \widehat{k_{13}}\right\vert
}{\left\vert \widehat{k_{23}}\right\vert }+1\leq\left(  1-\left\vert
\widehat{k_{12}}\right\vert ^{2}\right)  \left(  1-\left\vert \widehat{k_{13}%
}\right\vert ^{2}\right)  .
\]
Dividing by $\left\vert \widehat{k_{12}}\right\vert ^{2}\left\vert
\widehat{k_{13}}\right\vert ^{2}$ and rearranging then produces (\ref{third}).
\end{proof}

This result is just a statement that several numerical inequalities are
equivalent. However, if the $k_{ij}$ are the Gram matrix entries for some RKHS
$H$ and the $\delta^{\prime}$s are the $\delta_{H}$ distances between points
in $X(H),$ then the proposition shows how an inequality about the distances
can be reformulated using Gram matrix entries$.$ In particular, if
$H=DA_{n}(X)$ then the strong triangle inequality for $DA_{n}$ insures that
the first statement holds, and the proposition then insures that the other two
also hold. Furthermore, if $H$ has a complete Pick kernel then there is an $X$
so that $H\sim DA_{n}(X)$. In that case $\delta_{H}=\delta_{DA_{n}(X)}$ and
the STI, which is automatic for $\delta_{DA_{n}(X)},$ also holds for
$\delta_{H}$. Hence, also in that case all three statements hold for
$\delta_{H}$ and the kernels from $H$.

\begin{example}
Here is an example of a space $H$ for which the points of $\left(
X(H),\delta_{H}\right)  $ fail to satisfy (\ref{STI}). Suppose $0<r<1$ and let
$K$ be the $3\times3$ matrix with entries%
\begin{align*}
&  k_{12}=k_{22}=k_{32}=k_{21}=k_{23}=1\\
&  k_{11}=k_{33}=\left(  1-r^{2}\right)  ^{-2}\\
&  k_{13}=k_{31}=\left(  1+r^{2}\right)  ^{-2}.
\end{align*}
The matrix $K$ is positive definite and hence is the Gram matrix of a RKHS
$H.$ We write $X(H)=\left\{  x_{1},x_{2},x_{3}\right\}  $ and $\delta
=\delta_{H}.$ For small values of $r$ we have
\begin{align*}
\delta_{13}  &  =2\sqrt{2}r-4\sqrt{2}r^{3}+O\left(  r^{5}\right) \\
\frac{\delta_{12}+\delta_{23}}{1+\delta_{12}\delta_{23}}  &  =2\sqrt{2}%
r-\frac{9}{2}\sqrt{2}r^{3}+O\left(  r^{5}\right)  .
\end{align*}
For small $r$ the second line is smaller than the first and the STI fails$.$

To see this example in a larger context, recall that the Bergman space,
$A^{2}=A^{2}\left(  \mathbb{D}\right)  ,$ has kernel functions $k_{z}%
(w)=\left(  1-\bar{z}w\right)  ^{-2}.$ The $A^{2}$ kernel functions for the
points $\left\{  -r,0,r\right\}  $ have Gram matrix $K$ and hence their span
is (a rescaling of) $H.$ Either because the points $\left\{  -r,0,r\right\}  $
lie on a hyperbolic geodesic, or by direct computation, the pseudohyperbolic
distances, $\rho,$ of the three points satisfy the STI with equality:%
\[
\rho_{13}=\frac{2\rho_{12}}{1+\rho_{12}^{2}}.
\]
The Hardy space, $H^{2}.$ has kernel functions $k_{z}(w)=\left(  1-\bar
{z}w\right)  ^{-1}$ and $\delta_{H^{2}}=\rho.$ Using this fact, the formulas
for the kernel functions, and the definition of $\delta,$ we find that
$\delta_{A^{2}}^{2}=\rho^{2}\left(  2-\rho^{2}\right)  .$ In particular, for
small distances
\[
\delta_{A^{2}}\sim\sqrt{2}\rho.
\]
These last two displays are not compatible with what the STI calls for in $H,
$ which is
\[
\delta_{A^{2}}(1,3)\leq\frac{2\delta_{A^{2}}(1,2)}{1+\delta_{A^{2}}(1.2)^{2}%
}.
\]

\end{example}

In this example the failure of (\ref{STI}) insures that we do not have $H\sim
DA_{n}(X).$ However we will see in Example \ref{arg} below that having
(\ref{STI}) is not enough to insure that $H\sim DA_{n}(X).$ On the other hand,
if we are only interested in the metric structure of a three point set, and
not any additional structure, then (\ref{STI}) is a complete condition for
isometric embedding in hyperbolic space.

\begin{proposition}
\label{metric embed}A three point metric space $(Z,\delta)$ with $\delta<1$
can be mapped isometrically into $(\mathbb{CH}^{n},\rho)$ if and only if it
satisfies (\ref{STI}). If that holds then the map $\Phi$ \text{can be chosen
to map into }$\mathbb{D=CH}^{1},$ in which case the image is uniquely
determined up to the action of (a possibly antiholomorphic) isometry of
$\mathbb{CH}^{1}.$
\end{proposition}

\begin{proof}
We noted when we introduced (\ref{STI}) that the inequality is always
satisfied by points of $(\mathbb{CH}^{n},\rho)$. Hence, if we have the mapping
of $Z$ then (\ref{STI}) follows.

Now suppose we have $\left(  Z,\delta\right)  $\ which satisfies (\ref{STI})
and write $Z=\left\{  \zeta_{i}\right\}  _{i=1}^{3}.$ We want to find $\Phi$
mapping $Z$ into $\mathbb{CH}^{1}.$ By considering composition with Mobius
transformations we see that if we can find a map $\Phi$ with the right mapping
property, then we can find a $\Phi$ with $\Phi(\zeta_{1})=0$ and $\Phi
(\zeta_{2})=\delta(\zeta_{i},\zeta_{2})=s.$ Further, this normalization
determines $\Phi$ uniquely up to possible complex conjugation. Thus we are
reduced to showing that if we set $\Phi(\zeta_{1})=0$ and $\Phi(\zeta
_{2})=\delta(\zeta_{i},\zeta_{2})$ then we can find a $\Phi(\zeta_{3})=w,$
unique up to complex conjugation, so $\rho(0,w)=\delta(\zeta_{1},\zeta_{3})$
and $\rho(s.w)=\delta(\zeta_{2},\zeta_{3}).$

Those conditions state that $w$ must lie on the intersection of two
pseudohyperbolic circles, one centered at $0,$ the other centered at $s,$ with
radii given by the $\delta$'s. However those pseudohyperbolic circles are also
Euclidean circles with centers on the real axis. From this we see that the
intersection is either empty, or one point on the real axis, or two points,
conjugate to each other. The condition that the intersection be nonempty is
exactly that the triangle inequality for the hyperbolic metric be satisfied.
However that is equivalent to the pseudohyperbolic metric satisfying the STI.
If the intersection is nonempty, then selecting $w$ to be an intersection
point completes the proof.
\end{proof}

In short, the isometric congruence class of a three point set in
$\mathbb{CH}^{1}$ is uniquely determined by its distances. We are not
claiming, and it is not true, that$\ $the same holds for three point sets in
$\mathbb{CH}^{n}, $ $n>1.$

The fact that there are isometries of $\mathbb{CH}^{1}$ that are not
holomorphic persists in higher dimensions and is part of the discussion of
congruence in $\mathbb{CH}^{n}$, see, for instance, \cite{BE}. Going forward
when we refer to isometries we will mean the holomorphic ones,

\subsection{Two Dimensional Spaces\label{two}}

We now look in more detail at the possibility, given $H,$ of finding $\Phi$
such that $H\sim DA_{n}(\Phi(X(H))).$

If $\dim(H)=1$ there is nothing to say.

If $\dim(H)=2\ $then $H$ can be rescaled so that the Gram matrix is
\[
G(H)=%
\begin{pmatrix}
1 & 1\\
1 & g
\end{pmatrix}
,
\]
and because $G(H)$ is positive we must have $g>1.$ Set $\gamma=\sqrt{1-1/g}$.
The Gram matrix of $J=DA_{1}(\left\{  0,\gamma\right\}  )$ is identical to
$G(H).$ Hence $H\sim J$, and thus any two dimensional RKHS $H$ is a rescaling
of a space $DA_{1}(X).$

We can also describe the multiplier algebra, $\operatorname*{Mult}\left(
H\right)  . $ The multipliers are diagonal operators on a two dimensional
space, and hence can be analyzed without recourse to general theory. However,
it is convenient to take advantage of von Neumann's inequality which insures
us that if $M_{m}$ is the operator of multiplication by $m$ and it satisfies
$\left\Vert M_{m}\right\Vert =1,$ and if $\varphi$ is a conformal automorphism
of the disk, then $\varphi\left(  M_{m}\right)  =M_{\varphi(m)}$ is also a
multiplier of norm one. We also want the following elementary computational
fact about $\rho.$

\begin{lemma}
Given $\alpha.\beta\in\mathbb{C}$ and $0<\gamma<1$, there is a unique
$\lambda>0$ such that $\lambda\alpha,\lambda\beta\in\mathbb{D}$ and
$\rho(\lambda\alpha,\lambda\beta)=\gamma.$
\end{lemma}

Given a nonzero $M_{m}\in\operatorname*{Mult}\left(  J\right)  $ the lemma
produces a unique $\lambda>0$ such that $\rho(\lambda m(0),\lambda
m(\gamma))=\gamma=\rho(0,\gamma).$ Given that equality of distances, there is
a unique $\sigma\in\operatorname*{Aut}\left(  \mathbb{B}^{1}\right)  $ with
$\sigma(0)=\lambda m(0)\ $and $\sigma(\gamma)=$ $\lambda m(\gamma)$. The
coordinate multiplier, $M_{z},$ has norm one. That can be checked quickly by
computing the norm of the adjoint, $M_{z}^{\ast},$ using the basis of kernel
functions. Hence, by von Neumann's inequality the multiplier $N=\sigma
(M_{z})=M_{\sigma(z)}$ also has norm one. By comparing values we see that
$N=\lambda M_{m}$ and hence $\left\Vert M_{w}\right\Vert =1/\lambda.$
Furthermore, $\lambda$ could be written explicitly in terms of the values
taken by $m$ and the parameter $\gamma$ which is determined by the space $J$.

\subsection{Three Dimensional Spaces}

We now look at the case $\dim(H)=3$ in some detail. The situation is more
complicated than $\dim(H)=2$ because the realization of $H$ as $DA_{n}(X)$ is
not automatically possible. On the other hand, in three dimensions the Pick
property is equivalent to the CPP and hence some complications which appear in
higher dimensions are avoided.

\begin{theorem}
\label{n=3}Suppose $H$ is a three dimensional RKHS, $X=X(H)=\left\{
x_{i}\right\}  _{i=1}^{3}.$ The following are equivalent:

\begin{enumerate}
\item $H$ has the the complete Pick property.

\item $H$ \ has the Pick property.

\item $\exists i,j,$ $i\neq j$ with $\delta_{H}(x_{i},x_{j})=\delta_{G}%
(x_{i},x_{j}).$

\item $LF_{123}^{2}\leq\delta_{13}^{2}.$

\item
\begin{equation}
\frac{1}{|\widehat{k_{12}}|^{2}}+\frac{1}{|\widehat{k_{23}}|^{2}}+\frac
{1}{|\widehat{k_{13}}|^{2}}-1\leq\frac{2\cos A_{123}}{|\widehat{k_{12}%
}||\widehat{k_{23}}||\widehat{k_{13}}|} \label{B}%
\end{equation}

\item There are $w\in\mathbb{C}$, $s,t>0$ such that with
\begin{equation}
\Phi(X)=\left\{  (0,0),(s,0),(w,t)\right\}  \subset\mathbb{B}^{2}%
=\mathbb{CH}^{2}, \label{normal form}%
\end{equation}
we have $H\sim DA_{2}\left(  \Phi(X)\right)  ,$
\end{enumerate}

Furthermore, the location of the points of $\Phi(X),$ the rescaling
equivalence class of $H,$ and the congruence class of the triangle with
vertices $\Phi(X)$ are uniquely determined by the rescaling invariant
parameters $\delta=\{\delta_{12},\delta_{13},\delta_{23},A_{123}\}.$
\end{theorem}

\begin{corollary}
\label{pi/2}If $H$ is a three dimensional RKHS with the CPP then $\cos
A_{123}>0,$ $\left\vert A_{123}\right\vert \leq\pi/2.$
\end{corollary}

\begin{proof}
[Proof of the Corollary]An application of the Cauchy-Schwartz inequality shows
$|\widehat{k_{ij}}|\,<1.$ Hence the left hand side of (\ref{B}) is positive,
which shows $\cos A_{123}$ must be positive.
\end{proof}

The first two statements in the theorem are general properties of $H$ and
$\operatorname*{Mult}(H),$ the next three concern numerical invariants derived
from those function spaces. Statement \textbf{(5)} is Brehm's classical
description of parameters which determine the congruence class of triangles in
$\mathbb{CH}^{2},$ as given in \cite[Pg. 92]{BE} and translated into our
notation. The final statement describes a set $\Phi(X)$ in $\mathbb{CH}^{n}$
whose existence is required by Theorem \ref{embed}.

Even if $n=\dim(H)>3$ it is true that \textbf{(1)} implies \textbf{(2)}
implies \textbf{(3)} implies \textbf{(4)}, and that \textbf{(4)} and
\textbf{(5)} are equivalent. However in that range \textbf{(3)} is weaker than
\textbf{(2)} which is weaker than \textbf{(1)}. Also, in that range a simple
statement in the style of \textbf{(4)} is not enough to get a representation
such as \textbf{(6)}. Our work for $n>3$ centers on understanding how to
replace \textbf{(4)}. The path to proving \textbf{(6)} implies \textbf{(1)}
depends on how the CPP is defined. We will avoid any work at that spot by
accepting Theorem \ref{embed} which states that for finite dimensional spaces
the existence of a representation as in \textbf{(6)} is implies the CPP.

\begin{proof}
[Proof of the Theorem]If \textbf{(1)} holds then so does \textbf{(2)} which is
just a restricted version of \textbf{(1)}. Condition \textbf{(2)} is enough to
appeal to Proposition \ref{same delta} (whose proof only uses the Pick
property, not the CPP) and obtain the equality of $\delta_{G}$ and $\delta
_{H},$ i.e. (\ref{d=d}), for each pair of indices. \textbf{(3)} is the weaker
statement that the equality holds for a single pair of indices. However any
one equality $\delta_{Gij}=\delta_{Hij}$ is enough to give the formula
(\ref{multiplier formula}) for the extremal multiplier for that particular
pair of indices; and that is what we need to go forward. By renumbering, and
without loss of generality, we suppose we have the particular case that
$\delta_{G12}=\delta_{H12}.$ In that case, we know from
(\ref{multiplier formula}) that%
\[
M_{x_{2},x_{1}}(\zeta)=\frac{1}{\delta_{12}}\left(  1-\frac{k_{21}k_{1}\left(
\cdot\right)  }{k_{11}k_{2}\left(  \cdot\right)  }\right)
\]
is a multiplier of norm one. Because of that and the fact that $M_{x_{2}%
,x_{1}}(x_{2})=0$ we must have
\[
\left\vert M_{x_{2},x_{1}}(x_{3})\right\vert \leq\delta_{G23}\leq\delta
_{H23};
\]
the first inequality by the definition of $\delta_{G23},$ the second because,
as we mentioned in Section \ref{extremal problems}, the $\delta_{G} $'s are
always dominated the $\delta_{H}^{\prime}$'s. Rearranging that inequality
gives statement \textbf{(4)}.

Statements \textbf{(4)} and \textbf{(5)} are equivalent by an algebraic
rewriting, similar to that connecting (\ref{second}) and (\ref{third}) in the
proof of Proposition \ref{sti proposition}. However instead of starting with
\[
\left\vert 1-\left\vert \frac{\widehat{k_{21}}\widehat{k_{13}}}%
{\widehat{k_{23}}}\right\vert \right\vert \leq\delta_{12}\delta_{13},
\]
we start with the stronger statement \textbf{(4)}, which, written out using
(\ref{L}), is
\begin{equation}
\left\vert 1-\frac{\widehat{k_{21}}\widehat{k_{13}}}{\widehat{k_{23}}%
}\right\vert \leq\delta_{12}\delta_{13}. \label{stronger}%
\end{equation}
We now follow the proof of Proposition \ref{sti proposition}. We square both
sides of (\ref{intermediate}) and replace the $\delta$'s with their definition
in terms of the $k$'s and obtain%
\[
\frac{\left\vert \widehat{k_{12}}\right\vert ^{2}\left\vert \widehat{k_{13}%
}\right\vert ^{2}}{\left\vert \widehat{k_{23}}\right\vert ^{2}}-2\frac
{\left\vert \widehat{k_{12}}\right\vert \left\vert \widehat{k_{13}}\right\vert
}{\left\vert \widehat{k_{23}}\right\vert }\operatorname{Re}\cos\arg\left(
\frac{\widehat{k_{21}}\widehat{k_{13}}}{\widehat{k_{23}}}\right)
+1\leq\left(  1-\left\vert \widehat{k_{12}}\right\vert ^{2}\right)  \left(
1-\left\vert \widehat{k_{13}}\right\vert ^{2}\right)  .
\]
Dividing by $\left\vert \widehat{k_{12}}\right\vert ^{2}\left\vert
\widehat{k_{13}}\right\vert ^{2}$, using the definition of $A_{123},$ and
rearranging then produces \textbf{(5)}.

We now go to the basic construction, showing that \textbf{(4)} insures that we
can select the required points in hyperbolic space. We know from our analysis
of normal forms that if we can find some $X\subset\mathbb{CH}^{k}$ so that
$H\sim DA_{k}(X)$ then we can find a $X=\left(  x_{1},x_{2},x_{3}\right)
\subset\mathbb{CH}^{2}$ in normal form, i.e. as described in \textbf{(5)}, and
having $H\sim DA_{2}(X),$ Hence the question is if we can find $s,w,t$ so that
the following system is satisfied. Here the $\delta$'s and $k$'s are data from
$H;$ $s,w,t$ are the unknowns:
\begin{align}
\delta_{\mathtt{12}}^{2}  &  =1-\frac{\left\vert k_{\mathtt{12}}\right\vert
^{2}}{k_{\mathtt{22}}k_{\mathtt{11}}}=1-\left\vert \widehat{k_{12}}\right\vert
^{2}=s^{2}\label{geometrh}\\
\delta_{\mathtt{1}3}^{2}  &  =1-\frac{\left\vert k_{\mathtt{13}}\right\vert
^{2}}{k_{\mathtt{33}}k_{\mathtt{11}}}=1-\left\vert \widehat{k_{13}}\right\vert
^{2}=\left\vert w\right\vert ^{2}+t^{2}\label{geom 2}\\
\delta_{\mathtt{2}3}^{2}  &  =1-\frac{\left\vert k_{\mathtt{23}}\right\vert
^{2}}{k_{\mathtt{22}}k_{\mathtt{33}}}=1-\left\vert \widehat{k_{23}}\right\vert
^{2}=1-\frac{\left(  1-\delta_{\mathtt{12}}^{2}\right)  \left(  1-\delta
_{\mathtt{1}3}^{2}\right)  }{\left\vert 1-s\bar{w}\right\vert ^{2}%
}\label{geom 3}\\
A_{123}  &  =\arg k_{12}k_{23}k_{31}=-\arg\left(  1-s\bar{w}\right)
\label{geom 4}%
\end{align}
We start by setting $x_{1}=(0,0)$ and $s=\delta_{12}\ $so that $x_{2}=(s,0). $
Once that is done, then (\ref{geom 3}) and (\ref{geom 4}) force the value of
$1-s\bar{w},$ and hence of $w.$ If we can show that $\left\vert w\right\vert
\leq\delta_{13}$ then we can select a unique nonnegative $t$ such that
(\ref{geom 2}) holds. At that point we will have that $\left(  w,t\right)  $
is in the ball and all the required equations are satisfied, and we will be
finished. To obtain the required estimate for $w$ note that, using
$x_{1}=(0,0)$ and $x_{2}=(s,0)$ and $x_{3}=(w,t)\ $and the formula for the
kernel function, the definition of $LF_{123}$ in (\ref{L}) gives
$LF_{123}=\left\vert w\right\vert $. Thus statement 4 simplifies to the
required $\left\vert w\right\vert \leq\delta_{13}.$

Combined with the earlier comments this completes the proof.
\end{proof}

\begin{corollary}
\label{disk}In the situation of the previous theorem the following are equivalent:

\begin{enumerate}
\item[4'] $LF_{123}^{2}=\delta_{13}^{2}$

\item[6'] For some $w\in\mathbb{C}\ $and $s>0,$ and with $\Phi(X)=\left\{
0,s,w\right\}  \subset\mathbb{B}^{1}=\mathbb{CH}^{1},$ we have $H\sim
DA_{1}\left(  \Phi(X)\right)  .$
\end{enumerate}
\end{corollary}

\begin{proof}
Using the fact $LF_{123}=\left\vert w\right\vert $ from the previous proof we
see that 4' is equivalent to $t=0.$
\end{proof}

\subsubsection{About $LF_{123}^{2}$\label{geometry}}

Condition \textbf{(4)} on $LF_{123}^{2}$ is related to the positivity of one
of the matrices $MQ$ introduced in (\ref{MQ}). It is a basic fact from the
theory of spaces with the CPP that a necessary and sufficient condition for a
finite dimensional $H$ to have the CPP is that the matrices (\ref{MQ}) be
positive semidefinite, \cite[Thm. 7.6]{AM} and Theorem \ref{MQ theorem}. If
$\dim(H)=n$ then the general theorem requires consideration of $n$ matrices of
size $\left(  n-1\right)  $ $\times\left(  n-1\right)  $. However in three
dimensions the situation simplifies and we only need consider the positivity
of a single $2\times2$ matrix from (\ref{MQ}):%
\[
MQ=%
\begin{pmatrix}
1-1/k_{22} & 1-1/k_{23}\\
1-1/k_{32} & 1-1/k_{33}%
\end{pmatrix}
.
\]
That matrix has positive diagonal elements and hence its positivity reduces to
the positivity of $\det\left(  MQ\right)  ,$ which is equivalent to Condition
\textbf{(4)} .

The statement $LF_{123}^{2}\leq\delta_{13}^{2}$ is also an inequality between
two Euclidean distances in $\mathbb{B}^{n};$ it compares the length of the
hypotenuse of a right triangle to the length of one of the other sides. After
having placed the points $\Phi(x_{1})$ and $\Phi(x_{2})$ at $(0,0)\ $and
$(s,0)$ we want to find $w$ and $t$ so that if we place $\Phi(x_{3})$ at
$(w,t)$ then the required equalities hold. We do that in two steps. First we
locate the point $(w,0),$ the projection of the not-yet-located final point
$\Phi(x_{3})$ $=(w,t),$ into the span of the points already selected,
$(0,0)\ $and $(s,0)$, (In this context the orthogonal Euclidean projection and
the hyperbolic nearest point projection are the same.) Once that is done, we
can find an appropriate $t$ if, but only if, the side $(0,0)(w,0)$ of the
resulting right triangle $(0,0)(w,0)(w,t)$ would be shorter than its
hypotenuse $(0,0)(w,t).$ That inequality is the statement that $\left\vert
(w,0)\right\vert \leq\delta_{13}\ $which can be reformulated as $LF_{123}%
^{2}\leq\delta_{13}^{2}.$ Thus we split finding the final vector into
computing $(w,0),$ its footprint in the span of the vectors already selected,
and the\textit{\ length of its footprint, }$LF.$ If the length of the
footprint is not longer than the length of the final vector then there is no
obstruction to locating the final point by specifying a nonnegative height for
that vector above the footprint.

This scheme for placing $\Phi(x_{3})$ is similar to one we used in adjoining
points to sets in the construction of normal forms in Theorem \ref{reduction}
and to the methods we use later in Theorems \ref{one side} and
\ref{embed tree}. The general situation is that we have placed points
$\left\{  x_{1}=0,x_{2},...,x_{k}\right\}  $ in the ball and need to place a
new point, $x_{k+1}.$ We let $S_{k}$ be the span of $\left\{  x_{j}\right\}
_{j=`}^{k}.$ We place $x_{k+1}$ by first identifying an auxiliary point
$P_{k}(x_{k+1}),$ the point that, if we knew the location of $x_{k+1},$ would
be the orthogonal projection of $x_{k+1}$ onto $S_{k}$; the nearest point
projection in terms of both the Euclidean and hyperbolic distances. The first
coordinates of $x_{k+1}$ will be the first coordinates of $P_{k}(x_{k+1}).$
The remaining data needed to describe the location of $x_{k+1}$ is its
distance from the point $P_{k}(x_{k+1}),$ and that Euclidean distance $d$
becomes the $k^{th}$ coordinate of $x_{k+1},$ the height of $x_{k+1}$ above
$S_{k}.$ In this language the estimate $LF_{123}^{2}\leq\delta_{13}^{2}$ in
the previous proof is essentially the requirement that $d$ not be negative. A
similar comment applies to the estimate (\ref{norm}) at the end of the proof
of Theorem \ref{one side}.

We discuss the geometric interpretation of the special values $LF_{123}^{2}=0
$ and $LF_{123}^{2}=\delta_{13}^{2}$ in Section \ref{geometry of X}.

\subsubsection{About $A(x,y,z)$ and Area\label{about A}}

By comparing (\ref{B}) to (\ref{third}) we see that the conditions in the
previous theorem imply STI. However STI itself is not sufficient for the
statements in the theorem. Here are examples of spaces with the STI which fail
the conclusion of Corollary \ref{pi/2} about the size of $A_{123}$ and hence
do not have the CPP.

\begin{example}
\label{arg}Pick $r,\lambda$ with $0<r<1,$ $\lambda>0,$ Let $\omega$ be a
primitive cube root of unity$.$ Set $y_{i}=r\omega^{j},$ $j=1,2,3.$ Let
$K(r,\lambda)$ be the $3\times3$ matrix with entries%
\[
k_{ij}=\left(  1-y_{i}\overline{y_{j}}\right)  ^{-\lambda}%
\]
It can be verified by hand that $K(r,\lambda)$ is a positive matrix and hence
determines a three dimensional RKHS, $H=H(r,\lambda)$. Alternatively,
$k(y,w)=\left(  1-y\bar{w}\right)  ^{-\lambda}$ is the reproducing kernel for
a space $\mathcal{D}_{\lambda}$ of holomorphic functions on the disk, and
$K(r,\lambda)$ is the Gram matrix of the $\mathcal{D}_{\lambda}$ kernel
functions for the points $y_{1},y_{2},y_{3}.$ The space $\mathcal{D}_{\lambda
}$ has the CPP if, but only if, $\lambda\leq1.$ In those cases $H(r,\lambda)$
inherits the CPP from the containing $\mathcal{D}_{\lambda}$ and hence there
is a map $\Phi$ of $X(H(r,\lambda))$ into $\mathbb{CH}^{n}$ so that
$H(r,\lambda)\sim DA_{n}(\Phi(X(H(r,\lambda))).$

However for some $\lambda$ there is no embedding. The inequality (\ref{STI})
is not the problem. The symmetry of the configuration under rotations of
$2\pi/3$ insures that all the $\delta_{ij}$ are the same, in which case
(\ref{STI}) is automatic. However for some parameter values the space
$H(r,\lambda)$ fails to satisfy the conclusion of Corollary \ref{pi/2} which
requires $\cos A_{123}>0.$ For $H(r,\lambda)$ we have
\begin{align*}
A_{123}  &  =\arg k_{12}k_{23}k_{31}\\
&  =\arg\left(  \frac{1}{\left(  1-r^{2}\omega^{1-2}\right)  \left(
1-r^{2}\omega^{2-3}\right)  \left(  1-r^{2}\omega^{3-1}\right)  }\right)
^{\lambda}\\
&  =3\lambda\arg\left(  1-r^{2}\omega\right)
\end{align*}
Thus $\cos A_{123}<0$ for some $r,\lambda.$
\end{example}

Suppose $H$ is three dimensional. The previous example shows that the values
of $A,$ which are determined by the kernel functions in $H,$ can indicate an
obstruction to having $H\sim DA_{n}(X).$ On the other hand, recalling the
comments in Section \ref{numerical}, if there is such a representation then
$A$ is also a geometric invariant of $X.$ In that case it makes sense to ask
for its geometric interpretation.

We regard a triple of points $X=\left\{  x,y,z\right\}  \subset\mathbb{CH}^{n}
$ as the vertices of a geodesic triangle, $\Delta\subset$ $\mathbb{CH}^{n}$, a
triangle with vertices $X$ and sides which are geodesic segments connecting
the vertices. There are natural ways to measure the size of the sides of
$\Delta,$ either with $\rho$ or with $\beta$, but there is not a simple notion
of the area of $\Delta$. In this section and the next we discuss the relation
between the values of $A$ and two substitutes for the area for $\Delta.$
Another theme that runs through both discussions, although we will not give it
a quantitative formulation, is that $A$ measures how well $\Delta$ fits into a
single complex geodesic.

The congruence class of a Euclidean triangle is determined by its three side
lengths, but the analogous statement fails in complex hyperbolic space. As
suggested by our parameter count, as shown by Brehm in his classic analysis of
triangles in both projective and hyperbolic space \cite{B}, and as can be seen
from (\ref{B}), the side length data is not enough to determine the triangle.
An additional parameter is needed. Various quantities are used for a fourth
parameter; here we are using the angular invariant $A$, a version of the
\textit{invariante angulaire} introduced by E. Cartan in \cite{C}. Related
invariants are discussed in \cite[Ch. 7]{Go}.

If $n=1$ we are in the unit disk where there is a natural notion of the
surface of the triangle. In that case we can use the classical
Poincare-Bergman area element to define/compute the area of $\Delta,$
$\operatorname*{Area}(\Delta)$. Furthermore, in that case
$\operatorname*{Area}(\Delta)=2A(x,y,z).$ That can be proved by taking
advantage of the classical formula relating $\operatorname*{Area}(\Delta)$ to
the angles of $\Delta,$ a detailed discussion is in \cite[Sec 1]{C}. However
if $\Delta$ is in general position in $\mathbb{CH}^{n}$ then there is no
natural notion of the surface of $\Delta$ on which to base a notion of
"surface area". Nevertheless it is still possible to define the
\textit{symplectic area }of $\Delta,$ $\mathcal{SA(\Delta)}$. Complex
hyperbolic space carries a natural symplectic two form, $\omega,$ a type of
area form. Given the sides of $\Delta,$ select a smooth real two manifold
$\Sigma(\Delta)$ connecting the three sides of $\Delta.$ Define the symplectic
area of the triangle $\Delta$ by $\mathcal{SA(\Delta)=}\int_{\Sigma(\Delta
)}\omega.$ Because $\omega$ is a closed form Stokes' theorem allows us to
evaluate this as a boundary integral over the sides of $\Delta$, in particular
the value does not depend on the choice of $\Sigma(\Delta).$ Because that is
the only use we make of $\Sigma(\Delta),$ we need not be explicit about the
details of its construction. On the disk $\mathbb{D=CH}^{1}\subset
\mathbb{CH}^{n}$ the symplectic form $\omega$ is the same as the hyperbolic
area element and so, in that case $\mathcal{SA(\Delta)=}\operatorname*{Area}%
(\Delta)=2A(x,y,z).$ However much more is true. For general $\Delta
\subset\mathbb{CH}^{n},$ $2A(x,y,z)$ $=\mathcal{SA(\Delta)}.$ This general
fact requires more work. It was proved by proved by Hangan and Masalla
\cite{HM} by explicit evaluation of the double integral. It can also be
proved, both in this context and much more general ones, using Stokes'
theorem, see the discussion is \cite{C}. Of course once $\Delta$ is in general
position $\mathcal{SA}(\Delta)$ is only an "area" in a metaphorical sense.
Note for instance that for the triangle $\Delta$ with vertices $\left\{
(0,0),(s,0),(0,t)\right\}  \subset\mathbb{CH}^{2},$ $s,t\in\mathbb{R}$ we have
$k_{23}=1$ and hence $A(x,y,z)=0.$ Alternatively, note that the $\ $two-form
$\omega$ vanishes on the real two-plane spanned by the vertices of $\Delta.$
Still, it is satisfying to phrase Brehm's theorem as saying the congruence
class of a triangle is determined by its side lengths and its area.

In fact, $A$ and variations on it have a much richer life than we have
discussed. One suggestion of this is that in complex projective space there is
a similar formula relating the argument of a product of kernel functions to
the invariant area of a triangle \cite[Sec 1.3.6]{Go}, Another similar
formula, using kernels of the Fock space, gives the area of Euclidean
triangles in the plane. Also, an invariant similar to $A$ can be defined using
the Bergman kernel function and hence has a natural definition on general
symmetric domains, and even more widely. The cocycle identity persists and, in
many cases, so does the fact that $A$ can be evaluated by integrating a
natural symplectic form over a triangle. All this suggests that $A$ might be a
valuable cohomological tool in studying symmetric domains and, more generally,
complex and symplectic manifolds. This is true, but we will not even begin
discussing details of these relations. More information as well as further
references are in \cite[Thm 4.8]{BS}, \cite[Section 5 ]{C},
\cite[Introduction]{BIW}, and \cite[Sections 1,2,3]{BI}.

\subsubsection{More About $A(x,y,z)$ and Area\label{projection}}

In the previous section we introduced $\mathcal{SA}(\Delta),$ a functional
related to area which gave a geometric interpretation to the invariant
$A_{123}.$ We now introduce another geometric functional, also related to
area, which turns out to equal $\left\vert \mathcal{SA}(\Delta)\right\vert $
and hence gives a slightly different geometric interpretation of $A_{123}.$

Suppose, again, $X=\left\{  x_{1},x_{2},x_{3}\right\}  \subset\mathbb{CH}^{n}$
and let $\Delta$ is the associated geodesic triangle. Set $H=DA_{n}(X).$

If $X$ is contained in a complex geodesic, $G,$ then there is a natural way to
define the area of $\Delta.$ Because $G$ is a complex geodesic there is a
hyperbolically isometric map of $\mathbb{CH}^{1}$ onto $G.$ By using the
geometry from $\mathbb{CH}^{1}$ there is then a natural interpretation of the
region of $G$ inside $\Delta.$ That map also can be used to carry the Poincare
Bergman area element to $G$ where it can be used to compute the area of
$\Delta,$ $\operatorname*{Area}(\Delta).$

If $X$ is not in a complex geodesic then we can push $X$ into a nearby complex
geodesic using a map $\Pi,$ and then use the functional $\operatorname*{Area}%
(\cdot)$ compute the area of of the triangle with vertices $\Pi X.$ More
precisely, the suppose $S$ is a side of $\Delta$. It is a geodesic segment and
hence is contained in a unique complex geodesic $G.$ Let $\Pi$ be the
hyperbolic nearest point projection of $\mathbb{CH}^{n}$ to $G,$ and let
$\Delta_{\Pi X}$ be the triangle in $G$ with geodesic sides and with vertices
$\Pi X$. We define the \textit{projected area} of $X$ to be
$\operatorname*{Area}(\Delta_{\Pi X}).$ It is a consequence of the next
theorem that the value of $\operatorname*{Area}(\Delta_{\Pi X})$ would be the
same if we did the similar construction using a one of the other sides of
$\Delta.$

If $X$ is in a complex geodesic $J$ then $G=J\ $and $\Pi$ is the identity on
$\Delta.$ In that case, combining this with the discussion in the previous
section we have the following chain of equalities: $\operatorname*{Area}%
(\Delta_{\Pi X})=\operatorname*{Area}(\Delta_{X})=\left\vert \mathcal{SA}%
(\Delta_{X})/2\right\vert =\left\vert A_{123}\right\vert .$ Although
$\operatorname*{Area}(\Delta_{X})$ is not defined for $X$ in general position,
the other three quantities are defined and, in fact, are equal. Thus, although
$\operatorname*{Area}(\Delta_{\Pi X})$ is not new numerical data$,$ it does
give an alternative geometric interpretation of $A_{123}$.

\begin{theorem}
For $X$ a three point set in $\mathbb{CH}^{n}$ and $\Delta_{X}$ the triangle
with $X$ as its vertex set, we have
\[
\operatorname*{Area}(\Delta_{\Pi X})=\left\vert \mathcal{SA}(\Delta
_{X})/2\right\vert =\left\vert A_{123}\right\vert .
\]

\end{theorem}

\begin{proof}
The theorem is a consequence of the following:

\begin{enumerate}
\item The construction of $\Delta_{\Pi X}$, and hence also the final
statement, are invariant under automorphisms of $\mathbb{CH}^{n}.$

\item If the complex geodesic $G$ equals $\mathbb{D}$, the intersection of
$\mathbb{B}^{n}$ with the $z_{1}$ axis; then the nearest point projection of
$\mathbb{B}^{n}$ onto $G\mathbb{\ }$is the same as the Euclidean orthogonal
projection of $\mathbb{B}^{n}$ onto $\mathbb{D}$.

\item If $X\in\mathcal{N}$ then the hyperbolic area of the triangle in
$\mathbb{D}$ with vertices given by the Euclidean projection of the set $X$
into $\mathbb{D}$ is $|\mathcal{SA}(\Delta)|.$

\item For any $\Delta,$ $\mathcal{SA}(\Delta)=2A_{123}.$
\end{enumerate}

The first statement holds by inspection of the definitions. The second is an
elementary exercise after expressing the pseudohyperbolic distance between
points in terms of Euclidean coordinates in $\mathbb{B}^{n}.$ The third
statement is proved in \cite{Go} as part of the proof of Theorem 7.1.11. As
mentioned in the previous section, the final statement is a result of Hangan
and Masala \cite{HM} and also has alternative proofs as described in \cite{C}.
\end{proof}

\subsection{$\dim(H)>3$ and the CPP\label{dim>3}}

Theorem \ref{embed} stated that the finite dimensional reproducing kernel
Hilbert spaces with the CPP are exactly the rescalings of spaces $DA_{n}(X).$
We now prove part of that theorem, namely:

\begin{theorem}
[{\cite[Thm. 7.28]{AM}}]\label{one side}If $H^{+}$ is a finite dimensional
RKHS with the CPP then there is a finite set $X^{+}$ in some $\mathbb{CH}^{n}$
such that $H^{+}\sim DA_{n}(X^{+}).$
\end{theorem}

This is a well established result. Our goal here is to showcase a geometric
argument similar to what we just used for $\dim(H)=3.$ With that in mind, we
will be less than fully detailed.

\begin{proof}
Theorem \ref{n=3} proves the result in the case $\dim(H^{+})=3.$ With that as
a starting point we prove the theorem by induction on the dimension of $H. $
Thus we need to know that we can extend the definition of the function $\Phi$
from a set $X$ to a larger set $X^{+}$ so that certain conditions are met.
Here is the precise formulation.

Suppose $H^{+}$ is an $n+1$ dimensional RKHS with the CPP, with $X^{+}%
=X(H^{+})=\left\{  x_{i}\right\}  _{i=1}^{n+1}$, with kernel functions
$\left\{  k_{i}\right\}  _{i=1}^{n+1}$ and with Gram matrix $K^{+}=\left(
k_{ij}\right)  .$ Let $H\ $ be the subspace spanned by $\left\{
k_{r}\right\}  _{r=1}^{n}$, and hence $X=X(H)=\left\{  x_{i}\right\}
_{j=1}^{n}.$ The subspace $H$ inherits the CPP and hence, by our induction
hypothesis, and taking note of Corollary \ref{small space}, there is a map
$\Phi:$ $X$ $\rightarrow\mathbb{CH}^{n}$ so that $H\sim DA_{n}(\Phi(X)).$ We
can suppose $\Phi(X)$ is in normal form in which case $\Phi(x_{1})=z_{1}=0$
and $\Phi(X)$ is contained in the subspace of $\mathbb{C}^{n-1}$ of
$\mathbb{C}^{n}$ characterized by having the last coordinate equal to zero. We
write $\Phi(x_{i})=z_{i},$ $i=1,...,n.$ Hence we have the following formula
for \underline{some} of the entries of the $(n+1)\times(n+1)$ matrix $K^{+};$
for $1\leq i,j\leq n$%
\begin{equation}
k_{ij}=(1-\overline{z_{i}}z_{j})^{-1}. \label{gram entries}%
\end{equation}
We want to find $z_{n+1}\in\mathbb{CH}^{n}$ with the property that if we
extend $\Phi$ to a map $\Phi^{+}$ defined on all of $X^{+}$ by setting
$\Phi^{+}(x_{n+1})=z_{n+1}$\ then we will have $H^{+}\sim DA_{n}(\Phi
^{+}(X^{+})).$

We write the candidate for $z_{n+1}$ as
\begin{equation}
z_{n+1}=w+ce_{n}=(c_{1},...,c_{n-1},0)+c(0,...,0,1) \label{form}%
\end{equation}
We will be finished if we construct $z_{n+1}$ so that (\ref{gram entries})
holds for the full Gram matrix $K^{+}.$ This involves conditions on the inner
products, $\left\langle z_{j},z_{n+1}\right\rangle $ for $1\leq j\leq n+1,$
but our construction insures that for $j<n+1$ the last coordinate of $z_{j}$
is $0$ and hence $\left\langle z_{j},z_{n+1}\right\rangle =\left\langle
z_{j},w\right\rangle .$ Thus we can write all of those conditions as
requirements for $w:$
\begin{equation}
\left\langle w,z_{j}\right\rangle =1-1/k_{n+1,j}\text{ \ }j=1,...,n
\label{inner product}%
\end{equation}

We now suppose temporarily that $H$ is \textit{generic, }that is, the set
$Z=\left\{  z_{i}\right\}  _{i=2}^{n}$ is linearly independent. We consider
the other case later. The set $Z$ is a basis of $\mathbb{C}^{n-1}.$ Let
$Z^{\ast}=\left\{  z_{i}^{\ast}\right\}  _{i=2}^{n}$ be the dual basis. We
set
\[
w=\sum\nolimits_{i=2}^{n}\left\langle w,z_{i}\right\rangle z_{i}^{\ast}%
=\sum\nolimits_{i=2}^{n}\left(  1-1/k_{ni}\right)  z_{i}^{\ast},
\]
The first equality holds for any $w\in\mathbb{B}^{n-1},$ the second insures
the that (\ref{inner product}) holds.

We also want $z_{n+1}$ given by (\ref{form}) to satisfy $\left\vert
z_{n+1}\right\vert ^{2}=1-1/k_{n+1,n+1}.$ We have specified $w$ and from
(\ref{form}) we know $\left\vert w\right\vert ^{2}+\left\vert c\right\vert
^{2}=\left\vert z_{n+1}\right\vert ^{2}.$ Hence, to insure we can find the
required $c$ we need to show that $w$ satisfies
\begin{equation}
\left\vert w\right\vert ^{2}\leq1-1/k_{n+1,n+1}, \label{norm estimate}%
\end{equation}

This requirement; that the length of the projection of the as yet undiscovered
target vector $z_{n+1}$ onto the linear span of the points already identified
is, in fact, less than that the desired length of $z_{n+1},$ is the higher
dimensional analog of the statement $LF_{123}^{2}\leq\delta_{13}^{2}$ in
Theorem \ref{n=3}.

We now use the CPP hypothesis to obtain the length estimate. There is no loss
in passing to a rescaling of $H^{+},$ and hence we suppose $k_{1}$ is
identically one. In that case the McCullough-Quiggen matrix of (\ref{MQ}),
with subscript $1,$ is%
\begin{equation}
MQ^{+}=MQ_{1}(H^{+})=\left(  1-\frac{1}{k_{rs}}\right)  _{2\leq r,s\leq n+1}.
\end{equation}
The fundamental fact about this matrix is the following.
\end{proof}

\begin{theorem}
[{McCullough-Quiggen \cite[Thm. 7.6]{AM}}]\label{MQ theorem}If $H^{+}$ has the
CPP then the matrix $MQ^{+}$ is positive semidefinite.
\end{theorem}

\begin{proof}
We assumed $H^{+}$ is generic so we would know that the vectors of $Z$ are
linearly independent, which is equivalent to $MQ$ being strictly positive
definite. We write $MQ^{+}$ in block form as%
\begin{equation}
MQ^{+}=%
\begin{pmatrix}
MQ & v^{\ast}\\
v & c
\end{pmatrix}
. \label{block}%
\end{equation}
Here $MQ$ is the sub-matrix of $MQ^{+}$ obtained by deleting the last row and
last column,%
\[
MQ=\left(  \left(  1-\frac{1}{k_{rs}}\right)  _{2\leq i,s\leq n}\right)
\]
The other terms are what are needed to complete $MQ^{+}$; $v=\left(  \left(
1-\frac{1}{k_{ns}}\right)  _{2\leq s\leq n}\right)  ,$ $v^{\ast}$ is the
conjugate transpose of $v,$ and $c$ is the scalar $1-k_{n+1,n+1}^{-1}.$ A
computational lemma now lets us recast the hypothesis on $MQ^{+}$ in the form
we will use.

\begin{lemma}
\label{se}$MQ^{+}$ is positive definite if and only if $MQ$ is positive
definite and
\begin{equation}
vMQ^{-1}v^{\ast}<c \label{schur estimate}%
\end{equation}

\end{lemma}
\end{proof}

\begin{proof}
[Proof of Lemma]This is the characterization of the positivity of $MQ^{+}$ in
terms of the Schur complement of $MQ$ \cite[pg. 86]{AM}$\mathbb{.}$

Our desired estimate for $\left\vert w\right\vert $ is given by
(\ref{norm estimate}). We want%
\begin{equation}
\left\vert w\right\vert ^{2}=\sum_{i,j=1}^{n}\left(  1-1/k_{in}\right)
\left\langle z_{i}^{\ast},z_{j}^{\ast}\right\rangle \overline{\left(
1-1/k_{jn}\right)  }\leq1-1/k_{n+1,n+1}. \label{norm}%
\end{equation}
Defining the matrix $J$ by $J=\left(  \left\langle z_{i}^{\ast},z_{j}^{\ast
}\right\rangle \right)  _{1,j=2}^{n}$ and recalling the notation used in
(\ref{block}) we can rewrite the desired inequality as $vJv^{\ast}<c.$ If we
show $J=MQ^{-1}$ then, appealing to Lemma \ref{se}, we will be finished. That
matrix equality is the standard relation between the Gram matrix of a basis
and the Gram matrix of its dual basis. We saw it earlier in (\ref{later})
where $K^{\#}$ had the role of $J.$

Suppose now that $H^{+}$ is not generic. Renumber the points in $Z=\left\{
z_{i}\right\}  _{i=2}^{n+1}$ so that $W=\left\{  z_{i}\right\}  _{i=2}^{s}$ is
a maximal linearly independent set. Form a basis $V$ of $\mathbb{C}^{n-1}$ by
adding $n-s$ vectors to $W,$ each orthogonal to the vectors in $Z.$ We
temporary abuse notation and let $Z^{\ast}=\left\{  z_{i}^{\ast}\right\}
_{i=1}^{n}$ be the basis dual to $V.$ Now we set
\[
w=\sum\limits_{i=2}^{s}\left\langle w,z_{i}\right\rangle z_{i}^{\ast}%
=\sum\limits_{i=2}^{s}\left(  1-1/k_{in}\right)  z_{i}^{\ast},
\]
and note that the upper index of summation is $s<n.$ It is immediate from the
definitions that (\ref{inner product}) holds for $1\leq j\leq s.$ For the
remaining values of $j$ note that $z_{j}$ is a linear combination of $z_{i}$
with $i\leq s$ and that $\left(  1-1/k_{jn}\right)  $ is the same linear
combination of $\left(  1-1/k_{in}\right)  .$ Combining these two facts
insures that (\ref{inner product}) holds for the full range of $j.$

One reason we are considering the non-generic situation separately is to avoid
having to work with the more complicated analog of Lemma \ref{se} that holds
when $MQ^{+}$ is semidefinite. In this case we want the estimate%
\begin{equation}
\left\vert w\right\vert ^{2}=\sum_{i,j=2}^{s}\left(  1-1/k_{in}\right)
\left\langle z_{i}^{\ast},z_{j}^{\ast}\right\rangle \overline{\left(
1-1/k_{jn}\right)  }\leq1-1/k_{n+1,n+1}.
\end{equation}
where now we have a different upper limit of summation. In this case we start
from the matrix $MQ^{+},$ which is positive semidefinite, and remove rows and
columns $s+1,...,n.$ The resulting matrix $MQ^{-}$ is still positive
semidefinite and furthermore its upper left block obtained by deleting the
last row and last column is strictly positive definite. This last property by
the fact that the vectors in $W$ are linearly independent. Given the strict
positivity of that block we have the analog of $vK^{-1}v^{\ast}<c$ and can
finish the argument as before.
\end{proof}

This result includes a statement about the Gram matrix of a four dimensional
$H$ which insures there is an embedding $\Phi(X),$ but the geometric content
of the result is elusive. Consider the following specific question. Given a
four point set $X$ in $\mathbb{CH}^{3},$ we can think of $X$ as the vertices
of a tetrahedron, That configuration is described by nine parameters; the six
distances between pairs of points together with the angular invariants
associated with any three of the triangular faces, the fourth angular
invariant being determined by the cocycle identity (\ref{cocycle}). The
inverse question is this; given four triangles in $\mathbb{CH}^{3}$, can they
be assembled as the faces of a tetrahedron? That is, is there a tetrahedron in
$\mathbb{CH}^{3}$ whose four triangular faces are congruent to the four given
triangles? \ The triangles must satisfy the obvious necessary conditions; side
lengths must match and the cocycle identity for the faces must hold. A
configuration which meets these conditions is described by the same nine
parameters used to describe a tetrahedron. However that is not the full story.
An example due to Quiggen \cite[Pg. 94]{AM} shows that there must be
additional conditions. He gives a four dimensional RKHS, $H$ such that each of
the four natural three dimensional RKHS subspaces, $\left\{  H_{i}\right\}
_{i=1}^{4}$ has the CPP, but $H$ does not. The fact that each $H_{i}$ has the
CPP insures that for each $i$ we can find a triple $X_{i}$ in $\mathbb{CH}%
^{n}$ with $H_{i}\sim DA_{n}(X_{i}).$ The kernel functions for $H_{i}$ are
kernel functions from $H $ and hence the side lengths of the triangle $X_{i}$
and its angular invariant are the same as would be computed from the Gram
matrix of $H.$ This insures that these four triangles satisfy the matching
side length conditions and also that the $A$'s satisfy the cocycle condition.
However if the triangles could be assembled into a tetrahedron with vertices
$X$ then we would have $H\sim DA_{n}(X)$ and hence $H$ would have the CPP; but
it does not. The obstruction to there being such an $X$ must be that the
inequality (\ref{schur estimate}) fails; but the geometry associated with that
failure is not clear.

In the proof of Proposition \ref{is r-model} we will see an example of how,
under restrictive assumptions, locally coherent information about three
dimensional subspaces can be spliced together to completely describe a larger space.

\section{The Geometry of Sets\label{geometry of X}}

\subsection{Three Point Sets}

We now look at three dimensional spaces $H\sim DA_{n}(X)$ and the relation
between the analytic and algebraic properties of $H\ $and the geometric
properties of $X.$ We focus on the complex hyperbolic analogs of the Euclidean
statements that a set of points is colinear or coplanar. Taking note of the
comments in Section \ref{larger spaces}, some of the results apply
\textit{mutatis mutandis }to subspaces of larger spaces and subsets of larger
sets. We consider additional results for larger spaces and sets in the next section.

We are interested in properties of $H$ that are unchanged by rescaling, and
properties of $X$ that are unchanged by automorphisms. With that in mind we
focus on $H=DA_{2}(X)$\ with $X$ a three point set in normal form. We denote
the collection of all such sets by $\mathcal{N}$.%

\begin{equation}
\mathcal{N=}\left\{  X=\left\{  x_{1},x_{2},x_{3}\right\}  =\left\{
(0,0),(s,0),(w,t),\text{ }s,t>0,w\in\mathbb{C}\right\}  \subset\mathbb{CH}%
^{2}\text{.}\right\}  \label{form of x}%
\end{equation}

\noindent We will be particularly interested the certain subsets of
$\mathcal{N}$;%

\begin{equation}%
\begin{array}
[c]{cc}%
\mathcal{A}=\left\{  X\in\mathcal{N}:\text{ }s>0,\text{ }w\in\mathbb{R}\text{,
}t=0\right\}  \text{ \ } & \text{in a geodesic curve
\ \ \ \ \ \ \ \ \ \ \ \ \ \ \ \ }\\
\mathcal{B}=\left\{  X\in\mathcal{N}:\text{ }s>0,\text{ }w\in\mathbb{R}\text{,
}t>0\right\}  \text{\ } & \text{in a real geodesic disk
\ \ \ \ \ \ \ \ \ \ \ }\\
\mathcal{C}=\left\{  X\in\mathcal{N}:\text{ }s>0,\text{ }w=0\text{,
}t>0\right\}  \text{ } & .....\text{with a right angle at }x_{1}=0\\
\mathcal{D}=\left\{  X\in\mathcal{N}:\text{ }s>0,\text{ }w=s\text{,
}t>0\right\}  \text{ } & .....\text{with a right angle at }x_{2}\text{
\ \ \ \ }\\
\mathcal{E}=\left\{  X\in\mathcal{N}:\text{ }s>0,\text{ }w\in\mathbb{C}\text{,
}t=0\right\}  \text{ } & \text{in a complex geodesic \ \ \ \ \ \ \ \ \ \ \ }%
\end{array}
\label{types}%
\end{equation}

\subsubsection{Points on a Geodesic}

For $X$ in $\mathcal{N}$, $x_{1}=0$. Hence any hyperbolic geodesic segment
containing $x_{1}$ sits in a Euclidean line through the origin. That
observation will sometimes let us use Euclidean coordinate geometry rather
than hyperbolic incidence geometry,

Any two points in $\mathbb{CH}^{2}$ determine a unique geodesic segment. If a
three point set $Y$ which lies on a geodesic is put in normal form then one
point will be at the origin and hence the geodesic will lie in a real line
through the origin. That is, if $Y\sim X\in\mathcal{N}$ then the points of $Y$
are on a single geodesic exactly if $X\in\mathcal{A}$. One invariant
characterization of that configuration is that the points produce equality in
the triangle inequality for the hyperbolic metric (and hence the "triangle" is
degenerate); equivalently, equality holds in the strong triangle inequality
for the pseudohyperbolic metric. In either case one can tell which point is
between the others by noting which of the two possible equalities holds.

Alternative characterizations can also be given; for instance a geodesic
segment lies on the intersection of a complex geodesic and a real geodesic
disk. Combining that with the results below characterizing those geometric
conditions we have the following.

\begin{proposition}
\label{on geodesic}Suppose $X$ is a three point set in $\mathbb{CH}^{n}$ and
$H=DA_{n}(X).$ The points of $X$ lie on geodesic if and only if $LF_{123}%
=\delta_{13}\ $and $A_{123}=0$.
\end{proposition}

\subsubsection{Points in a Real Geodesic Disk\label{real geodesic disk}}

The totally geodesic submanifolds of $\mathbb{CH}^{n}$ of real dimension one
are the geodesics, we just considered those. We now look at the two types of
totally geodesic submanifolds of real dimension two introduced in Section
\ref{complex hyperbolic space}. We begin with the real geodesic disks.

Any real geodesic disk is equivalent under an automorphism to the intersection
of the ball with the \textit{real }linear span of $(1,0,...,0)$ and
$(0,1,0,...,0)$. Hence a general three point set $Y$ is in a real geodesic
disk if and only if it is congruent to a set in $\mathcal{B}$. If
$X\in\mathcal{N}$ this is equivalent to having $k_{23}\in\mathbb{R}$. An
equivalent invariant statement is the following.

\begin{proposition}
\label{rgd}The set $Y$ lies in a real geodesic disk if and only if $Y\sim X$
for some $X\in\mathcal{B}$, if and only if $A_{123}=0.$
\end{proposition}

This proposition describes the geometry associated with the minimal value of
$A_{123}.$ The maximal value of $A_{123}$ for $H=DA_{2}(Y)$, $\left\vert
A_{123}\right\vert =\pi/2,$ is not attained for any $Y\subset\mathbb{CH}^{2}.$
However, if we extend the definition of $A_{123}$ by continuity to distinct
triples in $\partial\mathbb{B}^{2}=\mathbb{S}^{3},$ then that value can be
attained. It was shown by E. Cartan that the value $\pi/2$ is attained exactly
if $Y$ lies on the intersection of the boundary with the closure of a geodesic
disk (i.e. the three points lie in a \textit{chain}), For a full discussion
see \cite[Cor 7.1.3]{Go}.

There are two subsets of $\mathcal{B}$ that we want to look at more closely;
$\mathcal{C}$ and $\mathcal{D}$. The points of any $X\in\mathcal{B}$ sit
inside the real disk $\left\{  \left(  s,t\right)  :s,t\in\mathbb{R},\text{
}s^{2}+t^{2}<1\right\}  \subset\mathbb{B}^{2}.$ As we mentioned, that disk,
with the metric $\delta_{DA_{2}}=\rho,$ is the Beltrami-Klein model of
$\mathbb{RH}^{2},$ and in that model the Euclidean lines segments are the
hyperbolic geodesics. Hence the hyperbolic triangle with vertices at the
points of $X\in\mathcal{B}$ is the same as the Euclidean triangle. For sets
$X$ in $\mathcal{C}$ or $\mathcal{D}$ that triangle is a (Euclidean and
hyperbolic) right triangle. For $X\in\mathcal{C}$ the right angle is at
$x_{1}=0,$ $x_{2}\perp x_{3};$ for $X\in\mathcal{D}$ the right angle is at
$x_{2},$ $x_{2}\perp x_{2}-x_{3},$ These two configurations are actually not
very different. The negative of the ball involution interchanging $x_{1}$ and
$x_{2}$ interchanges sets in $\mathcal{C}$ with those in $\mathcal{D}$. These
types of Euclidean orthogonality are pervasive in subsets of the sets $X$ we
construct in Section \ref{tree spaces} where we study spaces $DA_{n}(X) $
associated with spaces of functions on trees.

Given the previous results, including Theorem \ref{three point extremals}, the
following equivalences are straightforward.

\begin{proposition}
\label{orthogonality}Given a three point set $Y$ in $\mathbb{CH}^{n}$, let $X
$ be the normal form of $Y,$ and $H=DA_{2}(X).$ The following are equivalent:

\begin{enumerate}
\item $X\in\mathcal{C}$.

\item $x_{2}\perp x_{3}.$

\item $k_{23}=1,$

\item $LF_{123}=0,$

\item $\Delta=\delta_{12}\delta_{13}/\delta_{23}.$
\end{enumerate}
\end{proposition}

Note that the first two statements are unchanged if $Y$ is replaced by $\Theta
Y$ with $\Theta\in\operatorname*{Aut}(\mathbb{CH}^{n}),$ and hence they are
statements about the hyperbolic geometry of $Y,$ The last two are
algebraic/analytic statements about $H=DA_{n}(Y)$ that are invariant under
rescaling of $H.$

The analogous result for $\mathcal{D}$ is

\begin{proposition}
\label{orthogonality2}Given a three point set $Y$ in $\mathbb{CH}^{n}$, let
$X$ be the normal form of $Y,$ and $H=DA_{2}(X).$ The following are equivalent:

\begin{enumerate}
\item $X\in\mathcal{D}$.

\item $x_{2}\perp x_{2}-x_{3}$

\item $k_{22}=k_{23}.$

\item $LF_{213}=0.$

\item $\Delta=\delta_{12}.$
\end{enumerate}
\end{proposition}

Note that the fourth statements in the previous two propositions have
different strings of indices.

\begin{proof}
The equivalence of the first four statements is straightforward. We now look
at the last condition. Recall that $X$ is in normal form. In that case
Condition \textbf{(3)}, $k_{22}=k_{23},$ implies $|\widehat{k_{13}}%
|^{2}=|\widehat{k_{23}}|^{2}|\widehat{k_{12}}|^{2}.$ Using this with
$\delta_{ij}^{2}=1-|\widehat{k_{ij}}|^{2}$ leads to $\delta_{23}^{2}%
=\delta_{13}^{2}-\delta_{12}^{2}\left(  1-\delta_{23}^{2}\right)  .$ Noting
again that $k_{22}=k_{23}$ and using (\ref{L}) this can be rewritten as
$\delta_{23}^{2}=\delta_{13}^{2}(1-LF_{123}^{2}\left(  1-\delta_{23}%
^{2}\right)  ).$ That equality combined with (\ref{delta with fl}) gives
Condition \textbf{(5)}. On the other hand if we have Condition \textbf{(5)}
then we have equality between solutions to two extremal problems, one for a
function required to vanish at $x_{2}$ and one for a function required to
vanish at both $x_{2}$ and $x_{3}.$ Both of the problems have unique
solutions, hence the solutions agree. In particular the solution to the first
problem must vanish at $x_{3}. $ Using that fact and the explicit formula for
that extremal function given by (\ref{hilbert space formula}) with $x=x_{1}$
and $y=x_{2}$ we see that we must have $k_{22}=k_{23}$ which is Condition
\textbf{(3)}.
\end{proof}

\subsubsection{Points on a Complex Geodesic\label{complex geodesic}, Pick
Spaces}

If $X\in\mathcal{N}$ and $X$ is contained in a complex geodesic then,
recalling that $x_{1}$ is at the origin, the complex geodesic must be the
intersection of the ball with the complex line $\mathbb{C}x_{2}.$ In short,
$X\in\mathcal{E}$. Direct substitution in (\ref{L}) then yields $LF_{123}%
^{2}=\delta_{13}^{2},$ Furthermore note from (\ref{if lf - 1}) that
$LF_{123}^{2}=\delta_{13}^{2}$ is equivalent to
\begin{equation}
\Delta=\delta_{12}\delta_{13}. \label{quality}%
\end{equation}

There are relations between multipliers that are a consequence of this. For
$i=2,3$ let $m_{i}$ be the multiplier of unit norm which vanishes at $x_{i}$
and maximizes $\operatorname{Re}m_{i}(x_{1});$ and let $m_{23}$ be the
multiplier of unit norm which vanishes at both $x_{2}$ and $x_{3}$ and
maximizes $\operatorname{Re}m_{23}(x_{1}).$ In this situation $m_{23}%
=m_{2}m_{3}.$ The second function is certainly a competitor for the extremal
problem defining the first function. Further, that space of competitors is one
dimensional. Hence $m_{23}$ must equal $m_{2}m_{3}/\left\Vert m_{2}%
m_{3}\right\Vert .$ Hence,
\[
\delta_{G12}\delta_{G13}=\Delta_{G}=m_{23}(x_{1})=\frac{m_{2}(x_{1}%
)m_{3}(x_{1})}{\left\Vert m_{2}m_{3}\right\Vert }=\frac{\delta_{G12}%
\delta_{G13}}{\left\Vert m_{2}m_{3}\right\Vert }.
\]
The first equality is (\ref{quality}), the second is from the definition of
$m_{23},$ the next is from the analysis we just did of $m_{23},$ and the last
is from the definitions of $m_{1}$ and $m_{2}.$

This string of equalities is equivalent to (\ref{quality}) and it is apparent
that it holds if and only if $\left\Vert m_{2}m_{3}\right\Vert =1.$ Hence
those conditions are equivalent. Further, given our analysis of $m_{23},$
those conditions are also equivalent to $m_{23}=m_{2}m_{3}.$

Let $M_{2}$ be the multiplier of norm one which vanishes at $x_{1}$ and
maximizes $\operatorname{Re}M(x_{2}),$ and similarly for $M_{3}.$ The
statement $LF_{123}^{2}=\delta_{13}^{2}$ is equivalent to the statement that
$M_{\text{2}}$ and $M_{3}$ are unimodular multipliers of each other. This
holds because comparing (\ref{L}) and (\ref{extremal g}) shows that
$LF_{123}^{2}=\delta_{13}^{2}\ $is equivalent to $\left\vert M_{2}%
(x_{3})\right\vert =\delta_{13}.$ Hence a unimodular multiple of $M_{2} $ is a
solution to the extremal problem defining $M_{3};$ but that solution is unique..

Continuing the notation of the previous paragraphs, we have proved the
following result.

\begin{proposition}
\label{in complex geodesic}Given a three point set $Y$ in $\mathbb{CH}^{n}$,
let $X$ be the normal form of $Y$ and $H=DA_{2}(X).$ The following are equivalent:

\begin{enumerate}
\item The points of $Y$ lie in a single complex geodesic.

\item $X\in\mathcal{E}$.

\item $LF_{123}^{2}=\delta_{13}^{2}.$

\item $M_{2}=\alpha M_{3}$ for some $\alpha,$ $\left\vert \alpha\right\vert
=1$

\item $m_{23}=m_{2}m_{3}$

\item $\left\Vert m_{2}m_{3}\right\Vert =1.$

\item $\Delta=\delta_{12}\delta_{13}.$
\end{enumerate}
\end{proposition}

Statements \textbf{(3)} through \textbf{(7)} are all invariant under rescaling
or automorphism. Hence we could have started our analysis by using Corollary
\ref{disk} and Condition \textbf{(3)} to reduce consideration of \textbf{(4)
-- (7)} to statements about multipliers on $H\subset H^{2}.$ By the CPP for
$H^{2}\ $the questions could then be reduced to statements about Hardy space
extremal problems. \ Those particular problems are elementary ones which are
solved by Blaschke products of degree one and two.

That Statements \textbf{(5) }or \textbf{(6) }imply Statement \textbf{(1) }also
follows from results by Cole, Lewis, and Wermer \cite{CLW} in their work
characterizing multiplier algebras of Pick spaces.

In the next result we do not assume $H\sim DA_{2}(X);$ the existence of $X$
and the fact that it can be selected in $\mathbb{CH}^{1}$ are the main conclusions.

Recall the language of Section \ref{orthogonal}; if $H\sim DA_{1}(X)$ for some
$X\in\mathcal{E}$, then $H$ is called an r-Pick space, and $H$ is called
r-orthogonal if it is the rescaling of a space whose Gram matrix is an orthogonal.

\begin{theorem}
\label{three dimensional}A three dimensional RKHS $H$ is r-orthogonal if and
only if it is an r-Pick space.
\end{theorem}

\begin{proof}
That such a space is r-orthogonal is the three dimensional case of Theorem
\ref{is orthogonal}.

In the other direction, we start with an $H$ which is r-orthogonal and, with
no loss of generality, do a preliminary basepoint rescaling with $x_{1}$ as
basepoint. Let $K$ be the Gram matrix of $H.$ Let $\widetilde{H}$, with Gram
matrix $\tilde{K},$ be the orthogonal RKHS which is a rescaled version of $H$.
That is, $\tilde{K}\tilde{K}^{t}=I$. Because $\widetilde{H}$ is a rescaling of
$H$ there is a diagonal matrix, $\Gamma,$ with nonzero diagonal entries
$\left\{  \gamma_{i}\right\}  ,$ such that%
\[
\tilde{K}=\Gamma K\overline{\Gamma}=(\gamma_{i}k_{ij}\bar{\gamma}_{j}%
)_{i,j=1}^{3}%
\]
is an orthogonal matrix.

Let $R_{i},$ $i=1,2,3$ be the rows of $\tilde{K},$ $R_{i}=(\gamma_{i}%
k_{i1}\bar{\gamma}_{1},\gamma_{i}k_{i2}\bar{\gamma}_{2},\gamma_{i}k_{i3}%
\bar{\gamma}_{3});$ and let $R_{i}^{\ast}$ be the column vectors that are
their adjoints. Noting that $K^{t}=\bar{K}$ we see that the $R_{i}^{\ast}$ are
also the column vectors of $\tilde{K}.$ Hence the fact that $\tilde{K}$ is
orthogonal implies that, for $i\neq j,$ $R_{i}R_{j}^{\ast}=0;$%
\[
0=R_{i}R_{j}^{\ast}=\sum\nolimits_{s}\gamma_{i}k_{is}\bar{\gamma}_{s}%
\overline{\gamma_{j}k_{sj}\bar{\gamma}_{s}}=\gamma_{i}\bar{\gamma}_{j}%
\sum\nolimits_{s}\left\vert \gamma_{s}\right\vert ^{2}k_{is}k_{js}.
\]
We now consider the cases $(i,j)=(1,2),$ $(3,1),$ $(2,3).$ By the basepoint
rescaling we have $k_{ij}=1$ if $i$ or $j$ is $1.$ Also, we can cancel the
initial factor $\gamma_{i}\bar{\gamma}_{j}$ in each equation. The resulting
equations are
\begin{align*}
(1,2)  &  :\left\vert \gamma_{1}\right\vert ^{2}+\left\vert \gamma
_{2}\right\vert ^{2}k_{22}+\left\vert \gamma_{3}\right\vert ^{2}k_{23}=0.\\
(3,1)  &  :\left\vert \gamma_{1}\right\vert ^{2}+\left\vert \gamma
_{2}\right\vert ^{2}k_{32}+\left\vert \gamma_{3}\right\vert ^{2}k_{33}=0,\\
(2,3)  &  :\left\vert \gamma_{1}\right\vert ^{2}+\left\vert \gamma
_{2}\right\vert ^{2}k_{22}k_{32}+\left\vert \gamma_{3}\right\vert ^{2}%
k_{23}k_{33}=0.
\end{align*}
We have assumed that there is a nontrivial set $\{$ $\left\vert \gamma
_{i}\right\vert ^{2}\}$ for which the equations hold. For that to happen the
matrix of coefficients of the $\{$ $\left\vert \gamma_{i}\right\vert ^{2}\}$
must be singular, and hence has determinant $0;$
\[
0=\det%
\begin{pmatrix}
1 & k_{22} & k_{23}\\
1 & k_{32} & k_{33}\\
1 & k_{22}k_{32} & k_{23}k_{33}%
\end{pmatrix}
,
\]
Expanding this, recalling that $k_{23}=\overline{k_{32}},$ and rearranging
gives%
\begin{equation}
\left\vert k_{23}\right\vert ^{2}\left(  k_{33}+k_{22}-1\right)  +k_{22}%
k_{33}-2k_{22}k_{33}\operatorname{Re}k_{23}=0 \label{cubic}%
\end{equation}
After dividing by $\left\vert k_{23}\right\vert ^{2}k_{33}k_{22}$ this yields%
\[
1-\frac{1}{k_{22}}-\frac{1}{k_{33}}+\frac{1}{k_{22}k_{33}}=1+\frac
{1}{\left\vert k_{23}\right\vert ^{2}}-2\operatorname{Re}\frac{1}{k_{23}}.
\]
This equation is $LF_{123}^{2}=\delta_{13}^{2}$ written in terms of the Gram
matrix entries. Once we have that, by Theorem \ref{n=3} we conclude that
$H\sim DA_{2}(X)$ for some $X,$ and then the previous proposition insures that
we can select $X$ in $\mathcal{E}.$
\end{proof}

As we mentioned earlier, we know of no counterexample to a higher dimensional
version of the previous result.

\subsection{Larger Sets\label{n>3}}

In the previous section we considered three dimensional Hilbert spaces
$H=DA_{n}(X)$ and related the structure of $H$ to the geometry of the three
point set $X.$ Now we consider $H=DA_{n}(X)$ with larger $X.$ Several times we
will use the argument used to show that a set in Euclidean space lies in a
line if every three points in it are colinear.

\subsubsection{Sets in a Geodesic}

Suppose $H=DA_{n}(X)$ and we have applied a preliminary automorphism to $X$ so
that $x_{1}$ is at the origin; and hence, also, $H$ is basepoint normalized.
In that case the matrix $MQ_{1}(X)$ defined in (\ref{MQ}) is the Gram matrix
of the set of vectors $\left\{  x_{i}\right\}  _{i=2}^{n}\subset\mathbb{C}%
^{n}$. Hence that matrix and variations on it can be used to study linear
independence among the $x_{i}.$ Also, because $H$ is basepoint normalized at
$x_{1},$ the formulas for $MQ_{1}(H)$ are quite simple. For instance, if $n=4$
then
\begin{equation}
MQ_{1}(H)=MQ_{1}(X)=\left(  1-\frac{1}{k_{ij}}\right)  _{2\leq i,j\leq4}.
\label{formm}%
\end{equation}

\begin{lemma}
Given $X=\{x_{i}\}_{i=1}^{4}\subset\mathbb{CH}^{n}=\mathbb{B}^{n},$ with
$x_{1}=0$, The set $X$ lies on a complex line through the origin if and only
if
\[
\det MQ_{1}(X\smallsetminus\{x_{3}\})=\det MQ_{1}(X\smallsetminus
\{x_{3}\})=0.
\]
The set $X$ sits in a complex subspace of dimension two if and only if $\det
MQ_{1}(X)=0.$
\end{lemma}

\begin{proof}
We just look at the second case. From (\ref{formm}) we find that $MQ_{1}(X)$
is the Gram matrix $\left(  \left\langle x_{i},x_{j}\right\rangle \right)
_{i,j-2,3,4}.$ That matrix is nonsingular exactly if the three $x$'s are
linearly independent.
\end{proof}

The second statement in the next result is included because of the analogy
with Proposition \ref{real} below.

\begin{proposition}
\label{geodesic}$X$ lies in a geodesic

\begin{enumerate}
\item if and only if for any $1\leq i<j\,<k\leq r$ we have $A_{ijk}=0$ and
$LF_{ijk}=\delta_{ik}$.

\item if and only if for any $1\leq i<j\,<k\leq r$ we have $A_{ijk}=0$ and
$\det MQ_{1}(\{x_{i},x_{j},x_{k}\})=0.$
\end{enumerate}
\end{proposition}

\begin{proof}
If three of the points are on a geodesic then Proposition \ref{on geodesic}
insures that we have the two equalities in the first statement. In the other
direction, note that both those equations and the fact of lying on a geodesic
are invariant under automorphisms of $\mathbb{CH}^{n}.$ Hence we can suppose
that $x_{1}$ is at the origin. In that case, by Proposition \ref{on geodesic},
we see that, for any index $j,$ the three points $\left\{  x_{1},x_{2}%
,x_{j}\right\}  $ lie on a geodesic. Because $x_{1}$ is at the origin that
geodesic must be a Euclidean line. Thus $x_{j}$ is on $L,$ the Euclidean line
through $x_{2}$ and the origin. Now note that $j$ was arbitrary; hence all of
the points lie on $L.$ That completes the proof for the first statement. That
the second statement is equivalent to the first can be seen by writing the two
differing expressions in terms of kernel functions.
\end{proof}

\subsubsection{Sets in Real Geodesic Disks or Totally Real Subspaces}

We suppose $X=\left\{  x_{i}\right\}  _{i=1}^{r}\subset\mathbb{CH}^{n}$, $r>3,
$ $H=DA_{n}(X).$ From Section \ref{real geodesic disk} we know that if, for
instance, $A_{123}=0$ then $\left\{  x_{1},x_{2},x_{3}\right\}  $ lies in real
geodesic disk; similarly if $A_{124}=0.$ However we cannot concatenate those
results. Knowing both is not sufficient to insure that all four points lie in
a single real geodesic disk. Consider, for instance, the origin and real
vectors $a,b,$ and $c$ which are mutually orthogonal. However if we control
the dimension of the real span of $\left\{  x_{i}\right\}  $ then we can go
forward; and we can control that dimension using the matrices $MQ$. We will
give a result with that dimension bounded by two but the general pattern will
be clear.

The first statement in the next proposition is a variation on Lemma 2.1 of
work by Burger and Iozzi, \cite{BI}, in which they consider sets that sit
inside totally real subspaces. We will follow their language. More information
about the geometry and properties of totally real subspaces is their paper and
in \cite{Go}.

We will say that a subspace $S$ of $\mathbb{CH}^{n}$ is a \textit{totally real
subspace} of dimension $k$ if it is a totally geodesic submanifold isometric
to $\mathbb{RH}^{k}.$ In particular, if $k=1$ then $S$ is an ordinary
geodesics and for $k=2$ it is a real geodesic disk. As before, the description
is clearer if we use a preliminary automorphism to reduce to the case of $S$
containing the origin of $\mathbb{CH}^{n}=\mathbb{B}^{n}.$ The geodesic
connecting the origin to any other point is a radial line segment. Hence $S$
is the intersection of $\mathbb{B}^{n}$ with a totally real vector subspace of
$\mathbb{C}^{n}$ of dimension $k$; that is, a real vector subspace of
$\mathbb{C}^{n}$ spanned by vectors $\left\{  v_{i}\right\}  _{i=1}^{k}$ with
all $\left\langle v_{i},v_{j}\right\rangle $ real.

\begin{proposition}
\label{real}Suppose $X=\left\{  x_{i}\right\}  _{i=1}^{N}\subset
\mathbb{CH}^{n}$ and $H=DA_{n}(X)$.

\begin{enumerate}
\item If $A_{ijk}=0$ \ for every $i,$ $j,$ $k$ then $X$ is inside a totally
real subspace of $\mathbb{CH}^{n}.$

\item If $A_{ijk}=0$ \ for every $i,$ $j,$ $k$ and, furthermore, for every
$i,$ $j,$ $k,$ $l$ we have, in the notation of (\ref{MQ}),
\[
\det MQ_{1}(\{x_{i},x_{j},x_{k},x_{l}\})=0;
\]
then $X$ is contained in a real geodesic disk.
\end{enumerate}
\end{proposition}

\begin{proof}
As before, we first use an automorphism to reduce to the case of $x_{1}$ at
the origin. Having done that, $A_{ijk}=0$ implies $k_{jk}$ is real which, in
turn, implies $\left\langle x_{i},x_{j}\right\rangle $ is real. Having that
for all $j,k$ gives the first conclusion. For the second statement note that
if all of $X$ sits in a single geodesic containing the origin then we are
done. Otherwise we can find $x_{i}$ and $x_{j}$ which are linearly
independent. Select any $x_{k}$ and consider the matrix $M=MQ_{1}%
(\{x_{1},x_{i},x_{j},x_{k}\}).$ By the hypothesis on the $A^{\prime}$s the
entries of $M$ are real and by the second part of the hypothesis $\det M=0.$
Hence $x_{k}$ is in the real linear span of $x_{i}$ and $x_{j}.$ Because
$x_{k}$ was arbitrary we have our conclusion.
\end{proof}

\subsubsection{Sets in a Complex Geodesic, Pick Spaces}

Some of the results in this paper have been for general $H,$ others for $H$
which have the CPP. The next result considers an intermediate case, the
assumptions of the result make it automatic that every subspace spanned by
three reproducing kernels has the CPP. We do not know if that is enough to
reach the conclusion of the proposition. On the other hand, if we assume that
$H$ itself has the CPP then the desired conclusions follow easily. In the
actual proposition we make the intermediate assumption that each subspace
spanned by four kernel functions has the CPP.

Suppose $H$ is a RKHS with kernel functions $\left\{  k_{i}\right\}
_{i=1}^{m} $ , $m\geq4,$ and $X=X(H)=\left\{  x_{i}\right\}  _{i=1}^{m}.$ For
any set $\Lambda$ of indices let $H(\Lambda)$ be the subspace spanned by
$\left\{  k_{i}\right\}  _{i\infty\Lambda}.$

\begin{proposition}
\label{is r-model}Suppose that for every four element set of indices,
$\Lambda,$ the space $H(\Lambda)$ has the CPP. Then $H\sim DA_{1}(Y)$ for some
$Y\subset\mathbb{CH}^{1},$ that is, $H$ is an r-Pick space, if and only if for
any $1\leq i<j\,<k\leq m,$ $LF_{ijk}=\delta_{ik},$
\end{proposition}

\begin{proof}
[Proof of a simpler result]If $H$ is an r-Pick space then we can apply
Proposition \ref{in complex geodesic} to all the three dimensional subspaces
of $H$ and obtain the condition on the $LF$'s. In the other direction, if we
had the stronger assumption that $H$ has the CPP then we could start with
$H\sim DA_{n}(Y)$ for some $Y=\left\{  y_{i}\right\}  $ in $\mathbb{CH}^{n}.$
Then, again by Proposition \ref{in complex geodesic}, we would see that every
three element subset of $Y$ lies in a complex geodesic. However If two complex
geodesics share a pair of points then they are the same. Hence $Y$ actually
sits in a single geodesic. The rescaling induced by the automorphism placing
$Y$ on the $z_{1}$ axis produces the required Pick space,
\end{proof}

\begin{proof}
Without loss of generality $H$ is basepoint normalized with $x_{1}$ as
basepoint. From our analysis of multipliers on two dimensional spaces in
Section \ref{two} we know that there is a unique multiplier\ $M$ on
$H(\left\{  1,2\right\}  )$ which is of norm one and has $M(x_{1})=0$ and
$M(x_{2})=\delta_{H(\left\{  1,2\right\}  )}(x_{1},x_{2})=\delta_{H}%
(x_{1},x_{2}).$ Suppose $2<r\neq s\leq m.$ By the hypothesis $H(\left\{
1,2,r,s\right\}  )$ has the CPP. Hence there is a norm one extension of $M$ to
a multiplier $M^{rs}$ on $H(\left\{  1,2,r,s\right\}  ).$ Furthermore $M^{rs}$
must be given by (\ref{multiplier formula}) which shows that $M^{rs}$ is
unique and also shows that $M^{rs}(x_{r})$ does not depend on the index $s.$
Define the set $Y=\left\{  y_{i}\right\}  _{i=1}^{m}\subset\mathbb{D=CH}^{1}$
by $y_{1}=M(x_{1})=0,$ $y_{2}=$ $M(x_{2})=\delta_{12}$ and, for $r>2,$
$y_{r}=M^{rs}(x_{r})$ (which we just noted does not depend on $s).$

To complete the proof we will show that $H\sim DA_{1}(Y)$. We will establish
that by showing that the Gram matrix $K(H)=\left(  k_{rs}\right)  $ equals the
Gram matrix $G(DA_{1}(Y))=J=\left(  j_{js}\right)  .$ By construction both
matrices have only $1$'s in their first row and first column. Now select $r,s
$ with $2\leq r,s\leq m;$ we want to show $k_{rs}=j_{rs}.$ Gram matrix
elements are stable under passage to subspaces spanned by reproducing kernels,
so we can study the gram matrix element $k_{rs}$ in the context of the Hilbert
space $H(\left\{  1,2,r,s\right\}  )$ (with the obvious modifications in
interpretation if $r=s$). We assumed that the four dimensional space
$H(\left\{  1,2,r,s\right\}  )$ has the CPP, and hence we have $H(\left\{
1,2,r,s\right\}  )\sim DA_{3}(\left\{  a,b,c,d\right\}  )$ for $\left\{
a,b,c,d\right\}  \subset\mathbb{CH}^{3}.$ The argument we gave in the "Proof
of a simpler result" shows that, in fact, $\left\{  a,b,c,d\right\}  $ lies in
a complex geodesic. Hence, using an automorphism we can reduce to the case
$\left\{  a,b,c,d\right\}  \subset\mathbb{D=CH}^{1}$, $a=0,$ $b=\delta_{12}.$
The Gram matrix entries of $H(\left\{  1,2,r,s\right\}  ),$ computed in the
space $H,$ match the Gram matrix entries of $DA_{1}(\left\{  a,b,c,d\right\}
)$ computed by regarding that space as a subspace of $DA_{1}$. In particular
$k_{rs}=\left(  1-\bar{c}d\right)  ^{-1}.$ If we can establish that $y_{r}=c$
and $y_{s}=d$ we will have the desired match. The same argument is used for
both equalities and we will just look at the first.

The spaces $H_{r}=H(\left\{  1,2,r\right\}  )$ and $DA_{1}(\left\{
a,b,c\right\}  )=.DA_{1}(\left\{  0,\delta_{12},c\right\}  )$ are
corresponding subspaces of the two four dimensional spaces we just looked at,
and hence $H_{r}\sim DA_{1}(\left\{  0,\delta_{12},c\right\}  )..$We know from
Theorem \ref{n=3} that this uniquely determines $c.$ Consider now any three
dimensional space $J$ with $X(J)=\left\{  j_{1}.j_{2},j_{3}\right\}  $ and for
which we know $J\sim DA_{1}(\left\{  \alpha,\beta,\gamma\right\}  )$ for some
$\left\{  \alpha,\beta,\gamma\right\}  \subset\mathbb{D}$. Let $N$ be the
unique multiplier on $J$ of norm one with $N(j_{1})=0\ $and $N(j_{2}%
)=\delta_{12}.$ By noting formula (\ref{multiplier formula}) and also looking
at the proof of Theorem \ref{n=3}, we see that that%
\[
J\sim DA_{1}(\left\{  N(j_{1}),N(j_{2}),N(j_{3})\right\}  )=DA_{1}(\left\{
0,\delta_{12},N(j_{3})\right\}  ).
\]
We now compare these facts. If $J$ is $H_{r}$ then $j_{3}$ is $x_{r}$ and $N$
is $M^{rs}$ restricted to $H_{r},$ Hence $N(j_{3})=M^{rs}(x_{r})=y_{r}.$ Thus
$H_{r}\sim DA_{1}(\left\{  0,\delta_{12},y_{r}\right\}  ).$ Comparing this
with the earlier unique representation in that form we conclude $y_{r}=c,$
which is what we needed to finish.
\end{proof}

The argument in the proof gives a type of description of r-Pick spaces.

\begin{corollary}
Suppose $H$ is an r-Pick space with $X(H)=\left\{  x_{i}\right\}  _{i=1}^{m},$
and $M$ is a multiplier on $H$ of norm one which, and for some $i,j,$ has
$\ \rho(M(x_{i}),M(x_{j}))=\delta_{H}(x_{i},x_{k}).$ Then $H\simeq
DA_{1}(\left\{  M(x_{i})\right\}  ).$
\end{corollary}

\section{Function Spaces on Trees\label{tree spaces}}

\subsection{Defining the Spaces}

In this section we study a class of Hilbert spaces $H$ of functions on trees
$\mathcal{T}$. Many natural examples of this type of space are infinite
dimensional and certainly some of what we do extends to that setting, but we
continue to assume our $H$ are finite dimensional.

We start with a rooted tree $\mathcal{T}$, a connected loopless graph with a
root vertex $o.$ For vertices $x,y\in\mathcal{T}$, we let $\left[  x,y\right]
$ denote the non-overlapping path connecting $x$ and $y.$ We will be informal
about whether that path consists of vertices, edges, or both. If $w,x,y,z$ are
vertices we will write $w<x$ if $w\in\lbrack o,x)$, write $x\wedge y=z$ if
$z=\sup\{[o,x]\cap\left[  o,y\right]  \},$ and denote the immediate
predecessor of $y,$ $\sup\left\{  [o,y)\right\}  ,$ by $y^{-}.$

One way to form a RKHS $H$ of functions on $\mathcal{T}$ with properties
related to the structure of $\mathcal{T}$ is suggested by the metaphor that
$"<"$\ reflects a flow of time or a flow of influence. With that in mind, we
define kernel functions $\{k_{x}\}$ with the value $k_{xy}$, $x,y\in T$,
determined by the "shared past" of $x$ and $y.$ Explicitly, we select a
function $\Omega$ defined on $\mathcal{T}$ which satisfies, for $x,y\in T,$%
\begin{equation}
\text{ }\Omega(o)=1,\text{ }\Omega(y)<\Omega(x)\text{ if }y<x.
\label{omega properties}%
\end{equation}
and define $k=k_{\Omega}$ by%
\begin{equation}
k_{xy}=k_{x\wedge y\text{ }x\wedge y}=\Omega(x\wedge y). \label{omega}%
\end{equation}
These conditions insure that $k_{xy}>0$ and that $k$ satisfies the Cauchy
Schwarz inequality, $k_{xy}^{2}\leq k_{xx}k_{yy}.$

In fact this definition insures that $k_{xy}$ is the reproducing kernel for a
space $H$ and that $H$ has the CPP. We will establish both facts by
explicitly\ constructing a map $\Phi$ of $\mathcal{T}$ into $\mathbb{CH}^{n}.$
The entries of the Gram matrix are real and hence, in the language of Section
\ref{conjugations}, the space is equal to its conjugate, $H=\overline{H};$
and, also, $\Phi(\mathcal{T)}=\overline{\Phi(\mathcal{T)}}\subset
\mathbb{RH}^{n}\subset\mathbb{CH}^{n}.$ Furthermore, if the kernel function
satisfies (\ref{omega}) then it also satisfies the weaker condition
\begin{equation}
\text{if }x<y\text{ then }k_{xy}=k_{xx}. \label{less than}%
\end{equation}
That condition is reflected in the shape of $\Phi(\mathcal{T)}$; triples of
points in $\Phi(\mathcal{T)}$ have the type of orthogonality described in
Proposition \ref{orthogonality2}.

In addition, independently of the construction of $\Phi,$ we will use the
algebraic structure of $k_{xy}$ to show that it is a reproducing kernel and
has the CPP.

\begin{theorem}
\label{embed tree}Let $K$ be the kernel function for $DA_{n}.$ If
$\mathcal{T}$ and $\Omega$ are as described above, then:

\begin{enumerate}
\item The function $k_{xy}$ in (\ref{omega}) is the reproducing kernel for a
RKHS, $H=H(\mathcal{T}$,$\Omega),$ of functions on $\mathcal{T}$.

\item The space $H$ has the CPP.

\item There is a map $\Phi_{\Omega}:\mathcal{T}\rightarrow\mathbb{B}^{n}$ with
$\Phi_{\Omega}(o)=0,$ and for all $x,y\in\mathcal{T}$, $\ k_{xy}%
=K_{\Phi_{\Omega}(x)\Phi_{\Omega}(y)}=\Omega(x\wedge y).$ Thus
$H=H(\mathcal{T}$,$\Omega)=DA_{n}(\Phi_{\Omega}(\mathcal{T))}$.
\end{enumerate}
\end{theorem}

\begin{proof}
First we will construct the map $\Phi_{\Omega}$ required for \textbf{(3)}.
Once we have that then statements \textbf{(1)} and \textbf{(2)} follow from
general facts about spaces $DA_{n}(X).$ We will then give an alternate proof
of \textbf{(1)} and \textbf{(2)} using a summation by parts formula for kernel
functions of the form described by (\ref{omega properties}) and (\ref{omega}).

To construct $\Phi$ we first construct the \textit{spine of }$T,$
$\operatorname*{Sp}(T),$ a set of strings of orthonormal vectors in
$\mathbb{B}^{n}$ which is indexed by elements of $T.$ Let $E=\left\{
e_{x}\right\}  _{x\in T}$ be a set of orthonormal vectors in $\mathbb{B}^{n}$.
For each $y\in T$ let $\left\{  o,y_{1},y_{2},...,y_{)}\right\}  $ be the
ordered string of vertices in the interval $\left[  o,y\right]  $. Let
$\operatorname*{st}(y)$ be the corresponding ordered string of elements of
$E,$ $\operatorname*{st}(y)=\{e_{0},e_{y_{1}},...,e_{y}\}.$ Set
$\operatorname*{Sp}(T)=\left\{  \operatorname*{st}(y):y\in T\right\}  .$

Using $\operatorname*{Sp}(T)$ we construct $\Phi$ by selecting appropriate
positive scalars $\left\{  c_{x}\right\}  _{x\in T}$ and setting
\begin{equation}
\Phi(y)=\sum_{w\in\operatorname*{st}(y)}c(w)e_{w}. \label{def phi}%
\end{equation}
We define the coefficients $\left\{  c(w)\right\}  _{w\in T}$ by induction on
the parameter $n(w)$, the number of edges in the path $\left[  o,w\right]  .$
The only $w$ with $n(w)=0$ is $w=o$ and we begin by setting $c(o)=0;$ that is,
we map the root vertex to the origin. Suppose now we have defined the
$\left\{  c(w)\right\}  $ for all $w$ with $n(w)\leq N.$ Select $z$ with
$n(z)=N+1.$ We have $n(z^{-})=N$, hence by our induction hypotheses and the
definition of $\Phi$, $\Phi(z^{-})$ is already defined. Set $\Phi
(z)=\Phi(z^{-})+c(z)e_{z}$ with $c(z)$ the positive number which we now
define. In order to have
\begin{equation}
k_{zz}=K_{\Phi(z)\Phi(z)}=\Omega(z) \label{rnel}%
\end{equation}
we need%
\begin{equation}
\left\langle \Phi(z),\Phi(z)\right\rangle =1-\frac{1}{\Omega(z)}. \label{only}%
\end{equation}
By our construction of the string $\operatorname*{st}(z),$ the $e_{w}$
corresponding to $w\in\lbrack o,z)$ are orthogonal to $e_{z}$. Hence we want%
\[
1-\frac{1}{\Omega(z)}=\left\Vert \Phi(z)\right\Vert ^{2}=\left\Vert \Phi
(z^{-})\right\Vert ^{2}+\left\Vert c(z)e_{z}\right\Vert ^{2}=1-\frac{1}%
{\Omega(z^{-})}+c(z)^{2}.
\]
Thus we want $c(z)^{2}=\Omega(z^{-})^{-1}-\Omega(z)^{-1}.$ Because $\Omega$ is
increasing that quantity is positive. Hence we can select $c(z)>0$ and
complete the definition of $\Phi(z).$ There is no obstacle in repeating this
process through the set of $z$ with $n(z)=N+1$ to complete the inductive step
in the definition. Thus we have $\Phi(z)$ defined for all $z, $ Note that the
construction insures that
\begin{equation}
1-\frac{1}{\Omega(z)}=\left\Vert \Phi(z)\right\Vert ^{2} \label{every}%
\end{equation}
holds for every $z.$

We now check that for any $z,w\in$ $T$ we have $k_{wz}=K_{\Phi(w)\Phi
(z)}=\Omega(w\wedge z)$ Taking note of the formula for $K$ it suffices to show
that
\[
\left\langle \Phi(w),\Phi(z)\right\rangle =\left\langle \Phi(w\wedge
z),\Phi(w\wedge z)\right\rangle =1-\frac{1}{\Omega(w\wedge z)}.
\]
The structure of the tree insures that $w\wedge z$ is a point on the geodesic
$\left[  o,w\right]  $ and on the geodesic $[o.z].$ Taking note of the
orthogonality relations in $\operatorname*{st}(w)$ and $\operatorname*{st}(z)$
and the formula (\ref{def phi}) we see that $\Phi(w)=\Phi(w\wedge z)+r(w,z)$
with $r(w,z)\perp\Phi(w)$ and also $\Phi(z)=\Phi(w\wedge z)+t(z,w)$ with
$t(z,w)\perp\Phi(z).$ Furthermore, taking note of the definition of $w\wedge
z$, the substrings $\operatorname*{st}(w)\smallsetminus\operatorname*{st}%
(w\wedge z)$ and $\operatorname*{st}(z)\smallsetminus\operatorname*{st}%
(w\wedge z)$ are disjoint. That implies $r(w,z)\perp t(z,w).$ Combining these
facts gives the first equality in the previous display follows. To obtain the
second equality follows from (\ref{every}).

That completes the proof of \textbf{(3) }which, as we noted, implies
\textbf{(1)}\ and \textbf{(2). }We now give an independent proof of
\textbf{(1)}\ and \textbf{(2)} using a summation by parts formula for bilinear
forms with kernel functions such as $\left\{  k_{xy}\right\}  $ which are
functions of $x\wedge y.$

It is convenient to introduce several operators on functions defined
on\ $\mathcal{T}$ . For $g$ a function on $\mathcal{T}$, $a,b,c\in\mathcal{T}$
we set
\begin{align*}
Ig(\alpha)  &  =\sum_{\tau\leq\alpha}g(\tau),\text{ }\\
I^{\ast}g(b)  &  =\sum_{\tau\geq\beta}g(\tau)\text{\ }\\
Dg(c)  &  =g(c)-g(c^{-})\text{ if }c\neq o,\text{ }Dg(o)=0,\text{ }%
\end{align*}
The notation follows the usage in \cite{ARS02} where $I$ as a discrete model
for integration and the operator $I^{\ast}$ is the adjoint of $I$ with respect
to the pairing of $\ell^{2}(\mathcal{T)}$. The operator $D,$ the difference
operator, is the one sided inverse to $I.$

The following summation by parts formula is Lemma 3 of \cite{ARS10}. It is
proved there by several lines of straightforward computation.

\begin{lemma}
[Summation by parts.]\label{parts}For any functions $h$ and $f$ defined on
$\mathcal{T}$ we have
\[
\sum_{x,y\in\mathcal{T}}h(x\mathcal{\wedge}y)f(x)\overline{f(y)}%
=h(o)\left\vert I^{\ast}f(o)\right\vert ^{2}+\sum_{z\in\mathcal{T}%
}(h(z)-h(z^{-}))\left\vert I^{\ast}f(z)\right\vert ^{2}.
\]

\end{lemma}

To establish \textbf{(1)} we need to know that for any function $f$ defined
$\mathcal{T}$ we have
\[
\sum_{x,y\in\mathcal{T}}k_{xy}f(x)\overline{f(y)}\geq0
\]
with equality only if $f$ is the zero function. We apply the lemma with
$h=\Omega.$ Because $\Omega$ is increasing the term $(h(z)-h(z^{-}))$ is
positive and hence the resulting bilinear form is positive definite. Because
$k$ is defined in terms of $\Omega$ by (\ref{omega}) this shows that $k$
defines a positive definite form, and hence\ is the kernel function of some
Hilbert space $H,$ Furthermore, by Theorem 7.28 of \cite{AM}, to show $H$ has
the CPP it suffices to show that, in addition, $1-1/k_{xy}$ generates a
positive bilinear form. That follows from the same lemma, applied this time to
the function $h=1-1/\Omega$ which is increasing because $\Omega$ is.
\end{proof}

\begin{corollary}
[Infinite divisibility of kernel functions]If $k_{xy}$ and $\Omega$ are as in
the previous theorem, and if $\Lambda$ is any strictly increasing function
with $\Lambda(1)=1$ then $k_{xy}^{\Lambda}=\Lambda(k_{xy})$, the kernel
function associated with $\Omega_{\Lambda}=\Lambda(\Omega)$ through
(\ref{omega}), is the kernel function of the RKHS with CPP. In particular, for
$0<\lambda<\infty,$ $k_{xy}^{\lambda}$ is such a kernel.
\end{corollary}

This is a consequence of the theorem and the observation that if $\Omega$
satisfies (\ref{omega properties}) then so does $\Lambda(\Omega).$ The spaces
with kernel function $k_{xy}^{2}$ arise naturally in the study of Hankel forms
on $H$ and have an independent intrinsic description, there is some discussion
of this and further references in \cite{FR}.

This corollary does not hold for a general RKHS $H$ with the CPP. If $k$ is
the kernel function of any such $H$ then, for $0<\lambda<1,$ $k^{\lambda}$ is
the kernel function of a space with the CPP \cite[Remark 8.10]{AM}$.$ However
that range of $\lambda$ is sharp. That is shown by the family of spaces
$\mathcal{D}_{\lambda}$ of Example \ref{arg}, a family which includes the
Hardy space for $\lambda=1$ and the Bergman space at $\lambda=2.$

It is possible to reverse the construction in the theorem and recover the tree
from the Hilbert space. For instance, suppose we have a RKHS $H$ with its set
of reproducing kernels $\left\{  k_{\lambda}\right\}  _{\lambda\in\Lambda}%
\ $and that all $k_{\lambda\mu}=\left\langle k_{\lambda},k_{\mu}\right\rangle
$ are real. In analogy with (\ref{less than}), define a partial order
$\preccurlyeq$ on $\Lambda$ by $\sigma\preccurlyeq\tau$ if $k_{\sigma\tau
}=k_{\sigma\sigma}.$ Suppose there is an element $\alpha$ so that for all
$\lambda\in\Lambda$ we have $k_{\alpha\lambda}=1,$ or, equivalently,
$k_{\alpha\alpha}=1$ and for all $\lambda$, $\alpha\preccurlyeq\lambda.$
Suppose further that for each $\lambda$ the segment $\left[  \alpha
.\lambda\right]  =\left\{  \mu\in\Lambda:\mu\preccurlyeq\lambda\right\}  $ is
totally ordered by $\preccurlyeq.$ This is enough data to form $\mathcal{T}$ ,
a rooted tree with $\Lambda$ as its vertex set and $\alpha$ as the root. If we
define $\Omega$ on $\mathcal{T}$ by requiring (\ref{rnel}) hold then our space
$H$ is the space $H(\mathcal{T} $,$\Omega)$ produced by the earlier
construction. In fact, if we do not start with a Hilbert space, but just start
with a real valued function $k_{\cdot\cdot}$ on $\Lambda\times\Lambda$ which
induces a partial order of the type described then the previous discussion
produces a tree and the Hilbert space of functions on that tree having a
kernel function with the CPP.

Special cases of the previous theorem are proved in \cite{Haa} and \cite{N},
Although those proofs are formulated very differently, they center on
constructing strings of orthonormal vectors similar to our $\operatorname*{Sp}%
(T).$ In fact, given the structural form of $k$ it is not hard to see that
such strings of orthonormal vectors must provide the framework of any mapping
such as $\Phi.$

\subsection{Formulas for the Norm}

We can think of $\Omega(t)$ as defining the length of the path $\left[
0,t\right]  $ and let $\omega$ as the length of the individual segments. We
then have $\Omega(t)=I\omega(t)=\sum_{o<s\leq t}\omega\left(  s\right)  $, or,
equivalently, $\omega=D\Omega.$

When the kernel function of $H$ is of the form $k_{xy}=\Omega(x\wedge y)$ we
can write the distance function $\delta_{H}$ using $\Omega.$ For
$y\in\mathcal{T}$ and $y^{-}$ the predecessor of $y$ we have $k_{y^{-}y^{-}%
}=k_{yy^{-}}=\Omega\left(  y^{-}\right)  $ and $k_{yy}=\Omega(y).$ Thus%
\[
\delta_{H}^{2}(y,y^{-})=1-\frac{\Omega(y^{-})^{2}}{\Omega(y^{-})\Omega
(y)}=\frac{\Omega(y)-\Omega(y^{-})}{\Omega(y)}=\frac{D\Omega(y)}{\Omega
(y)}=\frac{\omega(y)}{\Omega(y)}.
\]
The final expressions suggest an analogy with the expression $\partial
_{\gamma}\log\left\Vert k_{y}\right\Vert $ for a continuous variable $y$.

Using the definition of $\omega$ and the summation by parts formula we can
write the norm of
\begin{equation}
f(y)=\sum c_{x}k_{x}(y). \label{f}%
\end{equation}
in two ways, one involving the values of $f(y)$, the other involving the
coefficients $c_{x}.$ The sets of data $\left\{  f(y)\right\}  $ and $\left\{
c_{x}\right\}  $ are dual to each other; the reproducing kernels generate the
evaluation functionals and the vectors in the basis which is dual to the basis
of reproducing kernels generate the coefficient functionals.

\begin{corollary}
Given $f$ as in (\ref{f}) we have
\begin{align}
\left\Vert f\right\Vert ^{2}  &  =\left\vert I^{\ast}f(o)\right\vert ^{2}%
+\sum_{z>o}\omega(z)\left\vert I^{\ast}(c_{x})(z)\right\vert ^{2},\text{
and}\nonumber\\
\left\Vert f\right\Vert ^{2}  &  =\left\vert I^{\ast}f(o)\right\vert ^{2}%
+\sum_{z>o}\omega(z)^{-1}\left\vert Df(z)\right\vert ^{2}. \label{sobolev}%
\end{align}

\end{corollary}

\begin{proof}
The first statement follows directly from the summation by parts formula. The
second follows from the first as soon as we show that $Df(z)=\omega(z)I^{\ast
}(c_{y})(z).$ Both sides are linear functions of $f$ and hence it suffices to
do the verification for $f=k_{x}$ Select $x$ and $z.$ If $x>z$ or $x$ is not
comparable to $z$ then $k_{x}(z)=k_{x}(z^{-})$ and hence $Dk_{x}(z)=0=I^{\ast
}(c_{y})(z).$ The other possibility is that $x\leq z$ in which case, taking
note of the definitions of $k,\Omega,$ and $\omega,$ we have
\[
Dk_{x}(z)=-k_{x}(z^{-})+k_{x}(z)=-\Omega(z^{-})+\Omega(z)=\omega
(z)=\omega(z)I^{\ast}(c_{y})(z).
\]

\end{proof}

\subsection{Examples}

\subsubsection{Dirichlet-Sobolev Spaces}

Classical Dirichlet type spaces and Sobolev spaces are characterized by
integrability conditions on derivatives. Analogous spaces on trees are
obtained putting summability conditions on differences.

The dyadic Dirichlet space is a basic example. Let $\mathcal{T}_{2}$ be a
rooted dyadic tree and let $I,$ $I^{\ast},$ and $D$ be as above, and select
$\omega$ to be identically one. Define $\mathcal{D(T}_{2}\mathcal{)}$, the
\textit{dyadic Dirichlet space }to be the Hilbert space $H(\mathcal{T}_{2}%
$,$\Omega)$ produced in the previous theorem, the space of
functions\textit{\ }$f$ on $\mathcal{T}_{2}$ for which%
\[
\left\Vert f\right\Vert ^{2}=\left\vert I^{\ast}f(o)\right\vert ^{2}%
+\sum_{z\in\mathcal{T}}\left\vert Df(z)\right\vert ^{2}<\infty.
\]
That space models the classical Dirichlet space, the space of functions $f$
holomorphic on the disk for which
\[
\left\Vert f\right\Vert ^{2}=\left\vert f(0)\right\vert ^{2}+\frac{1}{\pi}%
\int\int_{\left\vert z\right\vert <1}\left\vert \frac{d}{dz}f(z)\right\vert
^{2}dA(z)<\infty.
\]
The space $\mathcal{D(T}_{2}\mathcal{)}$, and and related spaces have been
studied by the author and collaborators, both for their intrinsic interest and
as a tool in the study of spaces of smooth functions; \cite{ARS02},
\cite{ARS06}, \cite{ARSW11b}, \cite{ARSW14}, \cite{ARSW18}.

\subsubsection{Exponentials of Distances}

Suppose that a rooted tree $\mathcal{T}$ carries a geodesic distance function
$d$; a nonnegative function such that, for $x,y,z\in\mathcal{T}$ with
$y\in\left[  x,z\right]  $ we have $d(x,z)=d(x,y)+d(y,z).$ Any such function
is obtained by assigning a nonnegative length to each edge and letting
$d(x,z)$ be the length of the geodesic path connecting $x$ and $y.$ Such
distance functions automatically satisfy the following useful relationship:
for $x,y\in\mathcal{T}$
\begin{equation}
d(x,y)=d(o,y)+d(o,y)-2d(o,x\wedge y). \label{distance between}%
\end{equation}
Interestingly, this can be rewritten as%
\[
d(o,x\wedge y)=(d(o,y)+d(o,y)-d(x,y))/2.
\]
Hence, by definition, $d(o,x\wedge y)$ equals the Gromov product $(x|y)_{o}$.
For more about that quantity see, for instance, \cite{V}.

Select $\Lambda>1.$ Given $\mathcal{T}$ and $d$ we consider the space
$H=H(\mathcal{T}$,$d,o)=H(\mathcal{T}$,$d,\Lambda,o)$ with kernel functions%
\[
k_{xy}=k_{x\wedge y,x\wedge y}=\Lambda^{d(o,x\wedge y)}=\Lambda^{(x|y)_{o}}.
\]
This is an instance of our earlier construction with $\Omega(s)=\Lambda
^{d(o,s)}$ and it has several attractive computational properties.

If we change the choice of root vertex on the tree then we can build a new
Hilbert space using the same distance function. If $\tilde{o}$ is the new root
then there is also a new order structure $\widetilde{<}$ and hence, also a new
meet operation $\tilde{\wedge}.$ We can then form the Hilbert space
$\widetilde{H}=H(\mathcal{T}$,$d,\tilde{o})$ with kernel function
\[
\tilde{k}_{xy}=\tilde{k}_{x\tilde{\wedge}y,x\tilde{\wedge}y}=\Lambda
^{d(\tilde{o},x\tilde{\wedge}y)}.
\]
Although we have changed the root, we have not changed the tree or the
distance function.

\begin{proposition}
Changing the root of $\mathcal{T}$ produces a rescaling of $H;$%
\[
H(\mathcal{T},d,o)\sim H(\mathcal{T},d,\tilde{o}).
\]

\end{proposition}

\begin{proof}
This is an immediate consequence of the definitions, the computational
properties of the function $\Lambda^{x},$ and the following equation which
relates the new geometry to the old;%
\begin{equation}
d(\tilde{o},x\tilde{\wedge}y)=d(o,x\wedge y)-d(o,\tilde{o})+d(o,x\wedge
\tilde{o})+d(o,\tilde{o}\wedge y). \label{change}%
\end{equation}
That equation is Lemma 4 of \cite{ARS10}, where it is described as "clear
after making sketches for the various cases".
\end{proof}

Another interesting rescaling of $H(\mathcal{T},d,\Lambda,o)$ is the
normalized kernel rescaling. That is, we pass to the space defined by the new
kernel functions
\[
j_{xy}=\frac{k_{vy}}{k_{xx}^{1/2}k_{yy}^{1/2}}.
\]
In this rescaling all the kernel functions are unit vectors, we always have
$\left\vert j_{xy}\right\vert \leq1,$ and the point $o$ does not play a
distinguished role. With $\Gamma=\Lambda^{-1/2},$ and using
(\ref{distance between}). we have
\[
j_{xy}=\frac{\Lambda^{d(o,x\wedge y)}}{\Lambda^{d(o,x)/2}\Lambda^{d(o,y)/2}%
}=\Gamma^{\left(  d(o,x)+d(o,y)-2d(o,x\wedge y)\right)  }=\Gamma^{d(x,y)}.
\]
Thus the new kernel function only depends on the distance between the points
and, in particular does not depend on the root. We denote the space with these
kernel functions by $J(\mathcal{T}$,$\Gamma,d)$.

\begin{proposition}
Given $\mathcal{T}$, $d,$ and $\Gamma,$ $0<\Gamma<1.$ The space $J(\mathcal{T}%
$,$\Gamma,d)$ with reproducing kernels $j_{x,y}=\Gamma^{d(x,y)}$ is a RKHS
which has the CPP. For $\Lambda=\Gamma^{-2}$ and any choice of basepoint $o$
in $\mathcal{T}$, $J(\mathcal{T}$,$\Gamma,d)\sim H(\mathcal{T},d,\Lambda,o).$
\end{proposition}

\begin{proof}
The space $J$ was constructed as a rescaling of $H$ and thus it is a RKHS with
the CPP. The second statement then follows from the previous proposition.
\end{proof}

The spaces $H(\mathcal{T},d,\Lambda,o)$ are instances of tree
Dirichlet-Sobolev spaces characterized by (\ref{sobolev}). They also show up
in other places and for other reasons, \cite{Haa}, \cite{N}, \cite{ARS10}. One
practical fact about the spaces is that they are well suited for making
explicit computations and estimates, \cite{Haa}, \cite{N}. The spaces are also
useful models for spaces holomorphic functions on $\mathbb{B}^{1}$ and, more
generally, Hilbert spaces that arise in the harmonic analysis of
$\operatorname*{Aut}\left(  \mathbb{B}^{1}\right)  $ and related groups,
\cite{N}, \cite{ARS10}.

There is an additional use of these spaces which goes beyond our discussion
here but we would like to at least mention. Questions involving function
spaces on the disk often lead to questions in the function theory of the
boundary circle. Those questions can be quite delicate, with subtle issues in
capacity theory replacing more familiar analysis of smooth functions. A
similar thing happens with function theory on trees, analysis in the Hilbert
space of functions on the tree leads to questions about functions on the ideal
boundary of the tree. In some cases that analysis on the ideal boundary is
much more transparent and tractable than its continuous analog, and it gives
both a tool and a guide for the more classical case. For instance, this is a
basic theme in \cite{N} and is explored \cite{ARSW14}.

\section{The Multiplier Algebra\label{multiplier algebra}}

If $A$ is the multiplier algebra of a finite dimensional RKHS $H$ with the
CPP; $A=\operatorname*{Mult}\left(  H\right)  ,$ $H\sim DA_{n}(X),$
$X\subset\mathbb{CH}^{n},\ $then many of the results in Section
\ref{geometry of X} can be used to pass analytic and geometric information
between $A,$ $H$ and $X.$ In fact much more is true. It is a theorem of Hartz
\cite[Sec. 3]{Ha} that $H$ and $X$ are determined (up to the natural
equivalence relations) by the structure of $A.$ Here is his theorem formulated
to emphasize the geometry of the unit ball of $A.$

We are assuming $A=\operatorname*{Mult}(DA_{n}(X)),$ and without loss of
generality we assume $X$ is in normal form. An $m\in A$ is determined by the
vector $a(m)=(\alpha_{1},...,\alpha_{n})$ where $m(x_{i})=a_{i}.$ Using those
vectors as coordinates we identify $A$ with the space $\mathbb{C}^{n},$ with
coordinatewise multiplication, and with the norm $\left\Vert \cdot\right\Vert
$ induced by $A$. Let $\left(  A\right)  _{1}$ be the closed unit ball of $A$
viewed as a subset of $\mathbb{C}^{n}.$ For $1\leq i\leq n$ let $S_{i}$ be the
hyperplane on which the $i^{th}$ coordinate vanishes, $S_{i}=\left\{
(\alpha_{1},...,\alpha_{n})\in\mathbb{C}^{n}:a_{i}=0\right\}  .$ For $1<j\leq
n$ let $e_{j}$ be the point of $S_{1}\cap\left(  A\right)  _{1}$ that gives
the maximum value of the functional $\operatorname{Re}a_{j}.$ Thus the
coordinates of $e_{i}$ are the values taken by the multiplier $m_{j}$ which
satisfies $\left\Vert m_{j}\right\Vert \leq1,$ $m_{j}(x_{1})=0,$ and
$\operatorname{Re}m_{j}(x_{j})$ is maximal. Because we are assuming that
$A=\operatorname*{Mult}(H)$ for an $H$ with the CPP, we know that that $m_{j}$
is unique and is given by (\ref{multiplier formula}). Because $X$ is in normal
form that formula simplifies. We have
\begin{align*}
e_{j}  &  =(e_{j1},e_{j2},...,e_{jn})=\left(  m_{1j}(x_{1}),...,m_{1j}%
(x_{n})\right) \\
m_{1j}(x_{r})  &  =\frac{1}{\delta_{j}(x_{1},x_{j})}\left(  1-\frac
{k_{j1}k_{1r}}{k_{11}k_{jr}}\right)  =\left(  1-k_{jj}^{-1}\right)
^{-1/2}\left(  1-\frac{1}{k_{jr}}\right)
\end{align*}
For all $j,$ $e_{j1}=0.$ It is clear from this formula that the $\left(
n-1\right)  ^{2}$ remaining $e_{jk}$ are sufficient to reconstruct the Gram
matrix of $DA_{n}(X).$ Thus

\begin{theorem}
[{\cite[Sec. 3]{Ha}}]If
\begin{equation}
A=\operatorname*{Mult}(DA_{n}(X)), \label{mult}%
\end{equation}
then the Hilbert space $H=DA_{n}(X)$ is determined up to rescaling,
equivalently, the set $X\subset\mathbb{CH}^{n}$ is determined up to
automorphism, by the $(n-1)^{2}$ complex numbers $F=\left\{  e_{jk}\right\}
_{j,k=2}^{n}.$
\end{theorem}

Here is a slightly weaker variation on the theorem using parameters that are
more algebraic. Given $A=\operatorname*{Mult}(DA_{n}(X))$ we extend the
notation of (\ref{delta g}) to
\[
\Delta(i;j,k)=\sup\left\{  \operatorname{Re}m(x_{i}):m\in A\text{, }%
m(x_{j})=m(x_{k})=0,\left\Vert m\right\Vert _{A}\leq1\right\}  .
\]
A geometric description of these numbers is that $\Delta(i;j,k)$ is the
maximal value of $\operatorname{Re}a_{i}$ in $S_{j}\cap S_{k}\cap\left(
A\right)  _{1}.$

Set%
\[
\mathfrak{D}\left(  X\right)  =\left\{  \delta_{ij}:1\leq i<j\leq n\right\}
\cup\left\{  \Delta(1;j,k):1<j<k\leq n\right\}
\]
The set of invariants $\mathfrak{D}$ determines the congruence class of $X$ up
to a finite set of ambiguity.

\begin{theorem}
Given $A=\operatorname*{Mult}(DA_{n}(X)),$ there are at most $2^{(n^{2}%
-3n)/2}$ distinct congruence classes $\mathcal{Y}$ of sets in $\mathbb{CH}%
^{n}$ for which $Y\in\mathcal{Y}$ implies $\mathfrak{D}\left(  Y\right)
=\mathfrak{D}\left(  X\right)  ,$
\end{theorem}

\begin{proof}
We see from Theorem \ref{three point extremals} that once we have
$\mathfrak{D}$ then we know $\cos A_{1jk}$ and hence we know the $A_{1jk}$ up
to sign. By Theorem \ref{n=3} and the comments which follow it, we then know
the congruence class of the triangle with vertices $\left\{  x_{1},x_{j}%
,x_{k}\right\}  $ up to a possible anticonformal conjugation. Thus
$(n^{2}-3n)/2$ binary choices determine the set of congruence classes of those
triangles. From Theorem \ref{reduction} we see that each set of choices
corresponds to at most one class $\mathcal{Y}_{i}.$
\end{proof}

In fact that bound is attained, see \cite{BE}.

The previous two theorems, as well as many of the previous results were
specifically about algebras of the form $A=\operatorname*{Mult}(DA_{n}(X)).$
There are closely related classes of algebras, for instance commutative finite
dimensional algebras of operators on Hilbert space, and one can ask about
their properties or ask how to recognize algebras of the type
$\operatorname*{Mult}(DA_{n}(X))$ among them. There is interesting literature
on these questions, including in particular the question of how to identify
\textit{Pick algebras, }algebras of the type\textit{\ }$\operatorname*{Mult}%
(DA_{1}(X)).$ Here are references to some of that work that seems related in
spirit to what we do here: \cite{CW}, \cite{CLW}, \cite{L}, \cite{MP},
\cite{P}, \cite{P2}, \cite{PS}.

\section{Beyond Spaces with Complete Pick Kernels}

Geometers who study moduli for finite subsets of $\mathbb{CH}^{n}$ frequently
also consider similar questions for finite subsets of complex projective
space, $\mathbb{CP}^{n}$, and there are very strong analogies between those
results and the results for $\mathbb{CH}^{n}$, \cite{B}, \cite{BE}, \cite{HS}.
It would be interesting to know how questions about point sets in
$\mathbb{CP}^{n}$ are related to Hilbert space questions. With that in mind we
mention that there are RKHS, $H,$ on the Riemann sphere for which the
associated $\delta_{H}$ is the natural metric for $\mathbb{CP}^{1}$, see, for
instance, the discussion of spin coherent states in, for instance, \cite{P}.

Finally, finite sets $X$ in $\mathbb{CH}^{n}$ are finite metric spaces with
additional structure inherited from $\mathbb{CH}^{n}.$ It would be interesting
to have an intrinsic, geometric, description of that type of structure on a
set $X,$ one not dependent on its realization inside hyperbolic space and
perhaps without references to Hilbert spaces or multiplier algebras.


\begin{thebibliography}{9999999}                                                                                          %


\bibitem[AM]{AM}Agler, J. and McCarthy J. Pick Interpolation and Hilbert
Function Spaces, Graduate Studies in Mathematics, 44, 2002.

\bibitem[AR]{AR}Arcozzi, N., Rochberg, R. Topics in dyadic Dirichlet spaces.
New York J. Math. 10 (2004), 45--67.

\bibitem[ARS02]{ARS02}Arcozzi, N., Rochberg, R., Sawyer, E. Carleson measures
for analytic Besov spaces. Rev. Mat. Iberoamericana 18 (2002), no. 2, 443--510.

\bibitem[ARS06]{ARS06}Arcozzi, N., Rochberg, R., Sawyer, E. Carleson measures
and interpolating sequences for Besov spaces on complex balls. Mem. Amer.
Math. Soc. 182 (2006), no. 859.

\bibitem[ARS07]{ARS07}Arcozzi, N., Rochberg, R., Sawyer, The Diameter
Space---A Restriction of the Drury-Arveson Hardy Space, Function Spaces, Fifth
Conference on Function Spaces, K. Jarosz ed., Contemporary Mathematics 435,
Amer. Math. Soc. 2007, 21-42.

\bibitem[ARS10]{ARS10}Arcozzi, N., Rochberg, R., Sawyer, E. Two variations on
the Drury-Arveson space, Proceedings of a conference on Hilbert Spaces of
Analytic Functions, CRM Proceedings and Lecture Notes, vol. 51, Amer. Math.
Soc., Providence, RI, 2010, pp. 41-58.

\bibitem[ARSW11a]{ARSW11a}Arcozzi, N. Rochberg, R. Sawyer, E. Wick, B. D.
Distance functions for reproducing kernel Hilbert spaces. Function spaces in
modern analysis, 25--53, Contemp. Math., 547, Amer. Math. Soc., Providence,
RI, 2011.

\bibitem[ARSW11b]{ARSW11b}Arcozzi, N., Rochberg, R., Sawyer, E., Wick, B. D.
The Dirichlet space: a survey \ New York Journal of Mathematics (2011) 45-86.

\bibitem[ARSW14]{ARSW14}Arcozzi, N., Rochberg, R., Sawyer, E., Wick, B. D.
Potential theory on trees, graphs and Ahlfors-regular metric spaces. Potential
Anal. 41 (2014), no. 2, 317--366.

\bibitem[ARSW18]{ARSW18}Arcozzi, N., Rochberg, R., Sawyer, E., Wick B.,
Dirichlet Spaces and Related Spaces, book manuscript, 2018.

\bibitem[BS]{BS}Bercenau, S., Schlichenmainer, \ Coherent state embeddings,
Polar Divisors, and Cauchy formulas, J. Geom. Phys. 34 (2000), no. 3-4, 336--358.

\bibitem[B]{B}Brehm, U. The shape invariant of triangles and trigonometry in
two-point homogenous spaces. Geom. Dedicata, 33 (1990) 59-76.

\bibitem[BE]{BE}Brehm, U., Et-Taoui Congruence criteria for finite subsets of
complex projective and complex hyperbolic spaces. Manuscripta Math. 96 (1998) 81-95.

\bibitem[BI]{BI}Burger, M., Izzo A. Bounded cholmology and totally real
subspaces in complex hyperbolic geometry. Ergod. Th. and Dynam. Syst. (2012) 467-478.

\bibitem[BIW]{BIW}Burger, M., Izzo A , and Wienhard, A. Hermitian symmetric
spaces and K\"{a}hler rigidity. Transform. Groups 12 (2007), no. 1, 5--32.

\bibitem[C]{C}Clerc, J.-L. An invariant for triples in the Shilov boundary of
a bounded symmetric domain. Comm. in Anal. and Geom. (2007) \ 147-174.

\bibitem[CO]{CO}Clerc, J.-L., Orsted, B. The Masov Index Revisited,
Transformation Groups, (2001), 303-320.

\bibitem[CLW]{CLW}Cole, B., Lewis, K., Wermer, J. A characterization of Pick
Bodies, J. Lond. Math. Soc 48 (1993) 316-328.

\bibitem[CW]{CW}Cole, B., Wermer, J. Isometries of certain operator algebras.
Proc. Amer. Math. Soc. 124 (1996), no. 10, 3047--3053.

\bibitem[DW]{DW}Duren P., Weir, R. The pseudohyperbolic metric and the Bergman
spaces in the ball, Trans AMS 358 (2007) 63-76.

\bibitem[FR]{FR}Ferguson, S., Rochberg, R. Description of certain quotient
Hilbert modules. Operator theory 20, 93--109, Theta Ser. Adv. Math., 6, Theta,
Bucharest, 2006.

\bibitem[GMR]{GMR}Garcia, S. Mashreghi, J. Ross, W. Introduction to model
spaces and their operators, Cambridge Studies in Advanced Mathematics, 148.
Cambridge University Press, Cambridge, 2016.

\bibitem[Go]{Go}Goldman, W. Complex Hyperbolic Geometry, Okford Mathematical
Monographs, Oxford University Oressm 1999.

\bibitem[G]{G}Gusevskii, N, The invariants of finite configuration in complex
hyperbolic geometry. Advamced School and Workshop on Discrete Groups in
Complex Geometry, Abdas Salem Institute of Theoretical Physics, July 2010.

\bibitem[Haa]{Haa}Haagerup, W. An example of a non nuclear $C^{\ast}-$algebra
which has the metric approximation property, Inventiones Math. 50 (1979) 279-293.

\bibitem[HS]{HS}Hakim, J., Sandler, H. The moduli of $n+1$ points in complex
hyperbolic $n-$space Geom Dedicata 97 (2003) 3-15.

\bibitem[HM]{HM}Hangan, Th.,Masala G. A geometric interpretation of the shape
invariant for geodesic triangles in complex projective spaces, Geom Dedicata
49 (1994) 129-134.

\bibitem[Ha]{Ha}Hartz, M. On the Isomorphism problem for multiplier algebras
of Nevanlinna-Pick spaces, arX05108v1.

\bibitem[L]{L}Lotto, B. A. von Neumann's inequality for commuting,
diagonalizable contractions. I. Proc. Amer. Math. Soc. 120 (1994), no. 3, 889--895.

\bibitem[MP]{MP}Mittal, M., Paulsen, V. Operator algebras of functions, J
Functional Anal, 258 (w010), 3195-3225.

\bibitem[N]{N}Neretin, Yu. A. Groups of heirarchomorphisms of trees and
related Hilbert spaces, J. Functional Anal. 200 (2003) 505-535.

\bibitem[P]{P}Paulsen, V. Matrix-valued interpolation and hyperconvex sets
Integr. Eqn. Op. Theory 41 (2001) 38-62.

\bibitem[P2]{P2}Paulsen V., Operator Algebras of Idempotents, J. Functional
Anal. 181 (2001) 209-236.

\bibitem[PS]{PS}Paulsen, V., Solazzo V. Interpolation and balls in
$\mathbb{C}^{k}$ J. Operator Theory 60 (2008), no. 2, 379--398.

\bibitem[Pe]{Pe}Perelomov, A. Generalized Coherent States and Their
Aplications, Texts and Monographs in Physics, Springer Verlag 1986.

\bibitem[Ru]{Ru}Rudin, W. Function Theory in the unit ball of $C^{n}$,
Springer-Verlag 1980.

\bibitem[Sh]{Sh}Shalit, O, Operator theory and functioh theory in the
Drury-Arveson Space and its quotients. Operator Theory (2015)1125-1180.

\bibitem[Sa]{Sa}Sawyer, E. Function theory: interpolation and corona problems.
Fields Institute Monographs, 25. American Mathematical Society, Providence,
RI; 2009.

\bibitem[Se]{Se}Seip, K. Interpolation and sampling in spaces of analytic
functions, University Lecture Series 33, American Mathematical Society 2004.

\bibitem[V]{V}V\"{a}is\"{a}l\"{a}, J. Gromov hyperbolic spaces. Expo. Math. 23
(2005), no. 3, 187--231.
\end{thebibliography}
\end{document}